\documentclass{amsart}%
\usepackage{amssymb}
\usepackage{amsmath}
\usepackage{amsfonts}
\usepackage{graphicx}%
\providecommand{\U}[1]{\protect\rule{.1in}{.1in}}
\newtheorem{theorem}{Theorem}[section]
\newtheorem{corollary}[theorem]{Corollary}
\newtheorem{proposition}[theorem]{Proposition}
\newtheorem{lemma}[theorem]{Lemma}
\theoremstyle{definition}
\newtheorem{definition}[theorem]{Definition}

\theoremstyle{remark}

\numberwithin{equation}{section}
\begin{document}
\title{MONOTONE COMPLETE C*-ALGEBRAS AND GENERIC DYNAMICS}
\author{Kazuyuki SAIT\^{O}}
\address{2-7-5 Yoshinari, Aoba-ku, Sendai, 989-3205, Japan }
\email{yk.saito@beige.plala.or.jp}
\author{J.D. Maitland WRIGHT}
\address{Mathematics Institute, University of Aberdeen, Aberdeen AB24 3UE }
\curraddr{Christ Church, University of Oxford, Oxford 0X1 1DP}
\email{j.d.m.wright@abdn.ac.uk ; maitland.wright@chch.ox.ac.uk}
\thanks{}
\thanks{}
\date{}
\dedicatory{ }
\begin{abstract}
Let $S$ be the Stone space of a complete, non-atomic, Boolean algebra. Let $G$
be a countably infinite group of homeomorphisms of $S$. Let the action of $G$
on $S$ have a free dense orbit. Then we prove that, on a generic subset of
$S,$ the orbit equivalence relation coming from this action can also be
obtained by an action of the Dyadic Group, $\bigoplus\mathbb{Z}_{2}.$ As an
application, we show that if $M$ is the monotone cross product $C^{\ast}%
$-algebra, arising from the natural action of $G$ on $C(S)$, and if the
projection lattice in $C(S)$ is countably generated then $M$ can be
approximated by an increasing sequence of finite dimensional subalgebras. On
each $S,$ in a class considered earlier, we construct a natural action of
$\bigoplus\mathbb{Z}_{2}$ with a free, dense orbit. Using this we exhibit\ a
huge family of small monotone complete $C^{\ast}$-algebras, $(B_{\lambda
},\lambda\in\Lambda)$ with the following properties: (i) Each $B_{\lambda}$ is
a Type III factor which is not a von Neumann algebra. (ii) Each $B_{\lambda}$
is a quotient of the Pedersen-Borel envelope of the Fermion algebra and hence
is strongly hyperfinite. The cardinality of $\Lambda$ is $2^{c}$, where
$c=2^{\aleph_{0}}.$ When $\lambda\neq\mu$ then $B_{\lambda}$ and $B_{\mu}$
take different values in the classification semi-group; in particular, they
cannot be isomorphic.

\end{abstract}
\maketitle

%\title[short text for running head]{full title}

%Only \author and \address are required; other information is
%optional.  Remove any unused author tags.

%author one information
%\author[short version for running head]{name for top of paper}

%author two information

%\subjclass is required.

%Abstract is required.

%Text of article.

%Bibliographies can be prepared with BibTeX using amsplain,
%amsalpha, or (for "historical" overviews) natbib style.
%Insert the bibliography data here.

\section{Introduction: monotone complete $C^{\ast}$-algebras}

Let $A$ be a $C^{\ast}$-algebra. Its self-adjoint part, $A_{sa},$ is a
partially ordered, real Banach space whose positive cone is $\{zz^{\ast}:z\in
A\}.$ If each upward directed, norm-bounded subset of $A_{sa},$ has a least
upper bound then $A$ is said to be \textit{monotone complete. }Each monotone
complete $C^{\ast}$-algebra has a unit element (this follows by considering
approximate units). Unless we specify otherwise, all $C^{\ast}$-algebras
considered in this note will possess a unit element. Every von Neumann algebra
is monotone complete but the converse is false.

Monotone complete $C^{\ast}$-algebras arise in several different areas. For
example, each injective operator system can be given the structure of a
monotone complete $C^{\ast}$-algebra, in a canonical way. Injective operator
spaces can be embedded as "corners" of monotone complete $C^{\ast}$-algebras,
see Theorem 6.1.3 and Theorem 6.1.6 \cite{m} and \cite{t,u}.

When a monotone complete $C^{\ast}$-algebra is commutative, its lattice of
projections is a complete Boolean algebra. Up to isomorphism, every complete
Boolean algebra arises in this way.

We recall that each commutative (unital) $C^{\ast}$-algebra can be identified
with $C(X),$ the algebra of complex valued continuous functions on some
compact Hausdorff space $X$. Then $C(X)$ is monotone complete precisely when
$X$ is \textit{extremally disconnected}, that is, the closure of each open
subset of $X$ is also open.

Monotone complete $C^{\ast}$-algebras are a generalisation of von Neumann
algebras. The theory of the latter is now very well advanced. In the seventies
the pioneering work of Connes, Takesaki and other giants of the subject
transformed our knowledge of von Neumann algebras, see \cite{zq}. By contrast,
the theory of monotone complete $C^{\ast}$-algebras is very incomplete with
many fundamental questions unanswered. But considerable progress has been made
in recent years.

This article follows on from \cite{zk} where we introduced a classification
semi-group for small monotone complete $C^{\ast}$-algebras which divides them
into $2^{c}$ distinct equivalence classes. But it is not necessary to have
read that paper in order to understand this one. Our aim is to be
comprehensible by anyone with a grounding in functional analysis and some
exposure to the more elementary parts of $C^{\ast}$-algebra theory, say, the
first chapter of \cite{zr}.

A monotone complete $C^{\ast}$-algebra is said to be a \textit{factor} if its
centre is one dimensional; we may regard factors as being as far removed as
possible from being commutative. Just as for von Neumann algebras, each
monotone complete factor is of Type I or Type II$_{1}$ or Type II$_{\infty},$
or Type III. Old results of Kaplansky \cite{z,za,zb} imply that each Type I
factor is a von Neumann algebra. This made it natural for him to ask if this
is true for every factor. The answer is "no", in general. We call a factor
which is not a von Neumann algebra \textit{wild.}

A $C^{\ast}$-algebra is \textit{separably representable} when it has an
isometric $\ast$-representation on a separable Hilbert space. As a consequence
of more general results, Wright \cite{zzb} showed that if a monotone complete
factor is separably representable (as a $C^{\ast}$-algebra) then it is a von
Neumann algebra. So, in these circumstances, Kaplansky's question has a
positive answer.

Throughout this note, a topological space is said to be \textit{separable} if
it has a countable dense subset; this is a weaker property than having a
countable base of open sets. (But if the topology is metrisable, they
coincide.) Akemann \cite{a} showed that a von Neumann algebra has a faithful
representation on a separable Hilbert space if, and only if its state space is separable.

We call a $C^{\ast}$-algebra with a separable state space \textit{almost
separably representable}. Answering a question posed by Akemann, see \cite{a},
Wright \cite{zzf} gave examples of monotone complete $C^{\ast}$-algebras which
have separable state spaces but which are NOT separably representable.

If a monotone complete factor $M$ possesses a strictly positive functional and
is not a von Neumann algebra then, as an application of a more \ general
result in \cite{zzg} $M$ must be of Type III, see also \cite{zf}. Whenever an
algebra is almost separably representable then it possesses a strictly
positive functional. (See Corollary 3.2). So if a wild factor is almost
separably representable then it must be of Type III.

A (unital) $C^{\ast}$-algebra\ $A$\ is said to be \textit{small} if there
exists a unital complete isometry of $A$ into$\ L(H),$ where $H$ is a
separable Hilbert space. See \cite{zg} and \cite{v,zk}.

It turns out that $A$ is small if, and only if, $A\otimes M_{n}(%
%TCIMACRO{\U{2102} }%
%BeginExpansion
\mathbb{C}
%EndExpansion
)$ has a separable state space for $n=1,2,...,$\cite{zg}. So clearly every
small $C^{\ast}$-algebra is almost separably representable. We do not know if
the converse is true, but it is true for monotone complete factors, \cite{zg}.

Examples of (small) wild factors were hard to find. The first examples were
due to Dyer \cite{k} and Takenouchi \cite{zp}. As a consequence of a strong
uniqueness theorem of Sullivan-Weiss-Wright \cite{zn}, it turned out that the
Dyer factor and the Takenouchi factor were isomorphic. See also \cite{ze}
where the Dyer factor was identified as a monotone cross product of the
Dixmier algebra by an action of the dyadic rationals.

Another method of finding wild factors was given by Wright \cite{zz}. He
showed that each $C^{\ast}$-algebra $A$\ could be embedded in its "regular
$\sigma-$completion", $\widehat{A}$. When $A$ is separably representable then
$\widehat{A}$ is monotone complete and almost separably representable.
Furthermore, when $A$ is infinite dimensional, unital and simple, then
$\widehat{A}$ is a wild factor. But it was very hard to distinguish between
these factors. Indeed one of the main results of \cite{zn} showed that an
apparently large class of wild factors were, in fact, a unique (hyperfinite)
factor. Some algebras were shown to be different in \cite{zh,w,zj}. In 2001 a
major breakthrough by Hamana \cite{v} showed that there were $2^{c}$
non-isomorphic (small) wild factors, where $c=2^{\aleph_{o}}.$ This pioneering
paper has not yet received as much attention as it deserves.

In \cite{zk} we introduced a quasi-ordering between monotone complete
$C^{\ast}$-algebras. From this quasi-ordering we defined an equivalence
relation and used this to construct a\ classification semi-group $\mathcal{W}
$ for a class of monotone complete $C^{\ast}$-algebras. This semi-group is
abelian, partially ordered, and has the Riesz decomposition property. For each
monotone complete, small $C^{\ast}$-algebra $A$ we assign a "normality
weight", $w(A)\in\mathcal{W}$. If $A$ and $\ B$ are algebras then $w(A)=w(B),$
precisely when these algebras are equivalent\textit{.} It turns out that
algebras which are very different\textit{\ }can be equivalent. In particular,
the von Neumann algebras correspond to the zero element of the semi-group. It
might have turned out that $\mathcal{W}$ is very small and fails to
distinguish between more than a few algebras. This is not so; the cardinality
of $\mathcal{W}$ is $2^{c},$ where $c=2^{\aleph_{0}}.$

One of the useful properties of $\mathcal{W}$ is that it can sometimes be used
to replace problems about factors by problems about commutative algebras
\cite{zk}. For example, let $G_{j}$ be a countable group acting freely and
ergodically on a commutative monotone complete algebra $A_{j}$ $(j=1,2).$ By a
cross-product construction using these group actions, we can obtain monotone
complete $C^{\ast}$-factors $B_{j}$ $(j=1,2).$ Then it is easy to show that
$wA_{j}=wB_{j}$. So if the commutative algebras $A_{1}$ and $A_{2}$ are not
equivalent, then $wB_{1}\neq wB_{2}.$ In particular, $B_{1}$and $B_{2}$ are
not isomorphic.

Influenced by $K-$theory, it is natural to wish to form the Grothendieck group
of the semi-group $\mathcal{W}$. This turns out to be futile, since this
Grothendieck group is trivial, because every element of $\mathcal{W}$ is
idempotent. By a known general theory \cite{r}, this implies that
$\mathcal{W}$ can be identified with a join semi-lattice. The Riesz
Decomposition Property for the semigroup turns out to be equivalent to the
semi-lattice being distributive. So the known theory of distributive join
semi-lattices can be applied to $\mathcal{W}$.

To each monotone complete $C^{\ast}$-algebra $A$ we associated a
\textit{spectroid} invariant $\partial A$\textit{\ \cite{zk}}. Just as a
spectrum is a set which encodes information about an operator, a spectroid
encodes information about a monotone complete $C^{\ast}$-algebra. It turns out
that equivalent algebras have the same spectroid. So the spectroid may be used
as a tool for classifying elements of $\mathcal{W}$. For a generalisation of
spectroid, see \cite{zv}.

One of the many triumphs of Connes in the theory of von Neumann algebras, was
to show that the injective von Neumann factors are precisely those which are
hyperfinite \cite{f}, see also \cite{zq}. It is natural to conjecture an
analogous result for wild factors. See \cite{e,zw} But this is not true. For,
by applying deep results of Hjorth and Kechris \cite{xa}, it is possible to
exhibit a small, wild, hyperfinite factor which is not injective. We shall
give details of this, and other more general results in a sequel to this paper.

When dealing with monotone complete $C^{\ast}$-algebras, saying precisely what
we mean by "hyperfinite", "strongly hyperfinite", "approximately finite
dimensional" and "nearly approximately finite dimensional", requires subtle
distinctions which are not needed in von Neumann algebra theory. See Section
12 for details.

(\# )Let $\Lambda$ be a set of cardinality $2^{c},$ where $c=2^{\aleph_{0}}.$
Then we showed in \cite{zk} that there exists a family of monotone complete
$C^{\ast}$-algebras $\{B_{\lambda}:\lambda\in\Lambda\}$ with the following
properties. Each $B_{\lambda}$ is a monotone complete factor of Type III, and
also a small $C^{\ast}$-algebra. For $\lambda\neq\mu,$ $B_{\lambda}$ and
$B_{\mu}$ have different spectroids and so $wB_{\lambda}\neq wB_{\mu}$ and, in
particular, $B_{\lambda}$ is \ not isomorphic to $B_{\mu}.$ We show, in
Section 12, that, by using the machinery constructed in this paper, we may
choose each $B_{\lambda}$ such that it is generated by an increasing sequence
of full matrix algebras.

\section{Introduction: Generic dynamics}

An elegant account of Generic Dynamics is given by Weiss \cite{zuw}; the term
first occurred in \cite{zn}. In these articles, the underlying framework is a
countable group of homeomorphisms acting on a complete separable metric space
with no isolated points (a perfect Polish space). This corresponds to dynamics
on a unique compact, Hausdorff, extremally disconnected space (the Stone space
of the complete Boolean algebra of regular open subsets of $%
%TCIMACRO{\U{211d} }%
%BeginExpansion
\mathbb{R}
%EndExpansion
$).

Let $G$ be a countable group. Unless we specify otherwise, $G$ will always be
assumed to be infinite and equipped with the discrete topology. Let $X$ be a
Hausdorff topological space with no isolated points. Further suppose that $X$
is a Baire space i.e. such that the only meagre open set is the empty set.
(This holds if $X$ is compact or a G-delta subset of a compact Hausdorff space
or is homeomorphic to a complete separable metric space.) A subset $Y$\ of $X$
is said to be \textit{generic, }if $X\backslash Y$ is meagre.\textit{\ }

Let $\varepsilon$ be an action of $G$ on $X$\ as homeomorphisms of $X$.

In classical dynamics we would require the existence of a Borel measure on $X
$ which was $G$-invariant or quasi-invariant, and discard null sets. In
topological dynamics, no measure is required and no sets are discarded. In
generic dynamics, we discard meagre Borel sets.

We shall concentrate on the situation where, for some $x_{0}\in X,$ the orbit
$\{ \varepsilon_{g}(x_{0}):g\in G\}$ is dense in $X$. Of course this cannot
happen unless $X$ is separable. Let $S$ be the Stone space of the (complete)
Boolean algebra of regular open sets of $X.$ Then, see below, the action
$\varepsilon$ of $G$ on $X$ induces an action $\widehat{\varepsilon}$ of $G$
as homeomorphisms of $S;$ which will also have a dense orbit.

When, as in \cite{zuw}, $X$ is a perfect Polish space, then, as mentioned
above, $S$ is unique; it can be identified with the Stone space of the regular
open sets of $%
%TCIMACRO{\U{211d} }%
%BeginExpansion
\mathbb{R}
%EndExpansion
.$ But if we let $X$ range over all separable compact subspaces of the
separable space, $2^{%
%TCIMACRO{\U{211d} }%
%BeginExpansion
\mathbb{R}
%EndExpansion
},$ then we obtain $2^{c}$ essentially different $S$; where $S$ is compact,
separable and extremally disconnected. For each such $S,$ $C(S)$ is a
subalgebra of $\ell^{\infty}.$

Let $E$ be the relation of orbit equivalence on $S.$ That is, $sEt$, if, for
some group element $g,$ $\widehat{\varepsilon}_{g}(s)=t.$ Then we can
construct a monotone complete $C^{\ast}$-algebra $M_{E}$ from the orbit
equivalence relation. When there is a free dense orbit, the algebra will be a
factor with a maximal abelian subalgebra, $A,$ which is isomorphic to $C(S).$
There is always a faithful, normal, conditional expectation from $M_{E}$ onto
$A.$

For $f\in C(S),$ let $\gamma^{g}(f)=f\circ\widehat{\varepsilon}_{g^{-1}}$.
Then $g\rightarrow\gamma^{g}$ is an action of $G$ as automorphisms of $C(S).$
Then we can associate a monotone complete $C^{\ast}$-algebra $M(C(S),G),$ the
\textit{monotone cross-product }(see \cite{zp}) with this action. When the
action $\widehat{\varepsilon}$ is free, then $M(C(S),G)$ \ is naturally
isomorphic to $M_{E}$. In other words, the monotone cross-product does not
depend on the group, only on the orbit equivalence relation. This was a key
point in \cite{zn} where a strong uniqueness theorem was proved.

In this article we consider $2^{c}$ algebras $C(S),$ each taking different
values in the weight semi-group $\mathcal{W}$. (Here $c=2^{\aleph_{0}},$ the
cardinality of $%
%TCIMACRO{\U{211d} }%
%BeginExpansion
\mathbb{R}
%EndExpansion
.$)

There is no uniqueness theorem but we do show the following. Let $G$ be a
countably infinite group. Let $\alpha$ be an action of $G$ as homeomorphisms
of $S$ and suppose this action has a single orbit which is dense and free.
Then, modulo meagre sets, the orbit equivalence relation obtained can also be
obtained by an action of $\bigoplus\mathbb{Z}_{2}$ as homeomorphisms of $S.$

This should be compared with the situation in classical dynamics. e.g. It is
shown in \cite{g} that any action by an amenable group is orbit equivalent to
\ an action of $%
%TCIMACRO{\U{2124} }%
%BeginExpansion
\mathbb{Z}
%EndExpansion
.$ But, in general, non-amenable groups give rise to orbit equivalence
relations which do not come from actions of $%
%TCIMACRO{\U{2124} }%
%BeginExpansion
\mathbb{Z}
%EndExpansion
.$

On each of $2^{c},$ essentially different, compact extremally disconnected
spaces we construct a natural action of $\bigoplus\mathbb{Z}_{2}$ with a free,
dense orbit. This gives rise to a family of monotone complete $C^{\ast}
$-algebras, $(B_{\lambda},\lambda\in\Lambda)$ with the properties (\#)
described above.

Let $E$ be the orbit equivalence relation arising from a free, ergodic action
of $G.$ Furthermore, suppose that the complete Boolean algebra of projections
in $C(S)$ is countably generated. Let $\mathcal{N}(M_{E}$ $)$ be the smallest
monotone closed $\ast-$subalgebra of $M_{E}$ which contains the normalising
unitaries of $A$ (that is the set of all unitaries $u$ such that $u^{\ast
}Au=A.$). Then $\mathcal{N}(M_{E}$ $)$ is an approximately finite dimensional
(AFD) factor. More precisely there is an increasing sequence of finite
dimensional, unital, $\ast-$subalgebras of $\mathcal{N}(M_{E}$ $),$ whose
union $\sigma-$generates $\mathcal{N}(M_{E}$)$.$ (In contrast to the situation
for von Neumann factors, we do not know whether we can always take these
finite dimensional subalgebras to be full matrix algebras.) As pointed out
above, we need to make a number of subtle distinctions when approximating
monotone complete algebras by finite dimensional subalgebras; see Section12
for details. For example $M_{E}$ is "nearly AFD". But in Section 11 we
construct huge numbers of examples of $\bigoplus\mathbb{Z}_{2}$ actions on
spaces $S$, which give rise to factors which we show to be strongly hyperfinite.

\section{Monotone $\sigma-$complete $C^{\ast}$-algebras}

Although our focus is on monotone complete $C^{\ast}$-algebras, we also need
to consider more general objects, the monotone $\sigma-$complete $C^{\ast}$-algebras.

A $C^{\ast}$-algebra is \textit{monotone} $\sigma-$\textit{complete} if each
norm bounded, monotone increasing sequence of self-adjoint elements has a
least upper bound.

\begin{lemma}
Let $A$ be a monotone $\sigma-$complete $C^{\ast}$-algebra. Let there exist a
positive linear functional $\mu:A\rightarrow%
%TCIMACRO{\U{2102} }%
%BeginExpansion
\mathbb{C}
%EndExpansion
$ which is faithful. Then $A$ is monotone complete. Let $\Lambda$ be a
downward directed subset of $A_{sa}$ which is bounded below. Then there exists
a monotone decreasing sequence $(x_{n}),$ with each $x_{n}\in\Lambda$, such
that the greatest lower bound of $(x_{n}),$ $%
%TCIMACRO{\tbigwedge \limits_{n=1}^{\infty}}%
%BeginExpansion
{\textstyle\bigwedge\limits_{n=1}^{\infty}}
%EndExpansion
x_{n},$ is the greatest lower bound of $\Lambda.$

\begin{proof}
See \cite{zm}.
\end{proof}
\end{lemma}

\begin{corollary}
When an almost separably representable algebra is unital and monotone
$\sigma-$complete then it is monotone complete.
\end{corollary}

\begin{proof}
If $A$ is almost separably representable we can find states $(\phi
_{n})(n=1,2...)$ which are dense in its state space. Then $\phi=\sum
_{n=1}^{\infty}\frac{1}{2^{n}}\phi_{n}$ is a faithful, positive linear functional.
\end{proof}

Let $A$ be a $C^{\ast}$-subalgebra of $L(H).$ Let $V$ be a real subspace of
the real Banach space $L(H)_{sa}$. We call $V$ a $\sigma-$closed subspace of
$L(H)_{sa}$ if, whenever $(a_{n})$ is an upper bounded, monotone increasing
sequence in $V$ then its limit in the weak operator topology is in $V.$
Consider the family of all $\sigma-$closed subspaces which contain $A_{sa},$
then the intersection of this family is the (smallest) $\sigma-$closed
subspace containing $A_{sa}.$ By a theorem of Pedersen this is the
self-adjoint part of a monotone $\sigma-$complete $C^{\ast}$-subalgebra of
$L(H).$ See Theorem 4.5.4 \cite{zd}.

Let $B$ be a (unital) $C^{\ast}$-algebra. Let us recall some well known
classical results \cite{zr,zd}. Let $(\pi,H)$ be the universal representation
of $B$ i.e. the direct sum of all the GNS representations corresponding to
each state of $B.$ Then the second dual of $B$, $B^{\prime\prime},$ may be
identified with the von Neumann envelope of $\pi(B)$ in $L(H).$ Let
$B_{sa}^{\infty}$ be the smallest subspace of $B_{sa}^{\prime\prime},$( the
self-adjoint part of $B^{\prime\prime}$), which is closed under taking limits
(in the weak operator topology) of bounded, monotonic sequences. Let
$B^{\infty}=B_{sa}^{\infty}+iB_{sa}^{\infty}.$ Then by Pedersen's theorem
$B^{\infty}$ is a monotone $\sigma-$complete $C^{\ast}$-subalgebra of
$B^{\prime\prime}.$ We call $B^{\infty}$ the \textit{Pedersen-Borel} envelope
of $B.$(Pedersen called this simply the "Borel envelope", it has also, with
some justice, been called the Baire envelope.)

Let $B$ be a monotone $\sigma-$complete $C^{\ast}$-algebra. We recall that
$V\subset B_{sa}$ is a $\sigma-$\textit{subspace} of $B_{sa},$ if $V$ is a
real vector subspace of $B_{sa}$ such that, whenever $(b_{n})$ is a monotone
increasing sequence in $V,$ which has a supremum $b$ in $B_{sa},$ then $b\in
V.$ (In particular, the $\sigma-$subspaces of $L(H)$ are precisely the
$\sigma-$closed subspaces of $L(H).$)

A $\sigma-$subalgebra of $B$ is a $\ast$-subalgebra whose self-adjoint part is
a $\sigma-$subspace of $B_{sa}.$ It follows from Lemma 1.2 \cite{zx} that each
$\sigma-$subalgebra is closed in norm and hence is a C*-subalgebra; see also
\cite{e}.

Further, $J$ is a $\sigma-$ideal of $B$ if $J$ is a $C^{\ast}$-ideal of $B$
and also a $\sigma-$subalgebra of $B.$

When $B$ and $A$ are monotone $\sigma-$complete $C^{\ast}$-algebras, a
positive linear map $\phi:B\rightarrow A$ is said to be $\sigma-$%
\textit{normal} if, whenever $(b_{n})$ is monotone increasing and bounded
above, then $\phi$ maps the supremum of $(b_{n}),$ $\ $to the supremum of
$(\phi(b_{n})),$i.e. $\phi(%
%TCIMACRO{\tbigvee \limits_{n=1}^{\infty}}%
%BeginExpansion
{\textstyle\bigvee\limits_{n=1}^{\infty}}
%EndExpansion
b_{n})=$ $%
%TCIMACRO{\tbigvee \limits_{n=1}^{\infty}}%
%BeginExpansion
{\textstyle\bigvee\limits_{n=1}^{\infty}}
%EndExpansion
\phi(b_{n}).$

\begin{lemma}
Let $A$ be a monotone $\sigma-$complete $C^{\ast}$-algebra and let $J$ be a
$\sigma-$ideal of $A.$ Let $q$ be the quotient homomorphism of $A$ onto $A/J.$
Then $A/J$ is monotone $\sigma-$complete and $q$ is $\sigma-$normal. Let
$(c_{n})$ be a monotone increasing sequence in the self-adjoint part of $A/J$
which is bounded above by $c.$ Then there exists a monotone increasing
sequence $(a_{n})$ in $A_{sa}$ such that $q(a_{n})=c_{n}$ for each $n$ and
$\ (a_{n})$ is bounded above by $a$ where $q(a)=c.$
\end{lemma}

\begin{proof}
This follows from Proposition 1.3 and Lemma 1.1 \cite{zx}, see also \cite{ze}.
\end{proof}

The following representation theorem was proved by Wright in \cite{zy}. It may
be thought of as a non-commutative generalisation of a theorem of Loomis and
Sikorski in Boolean algebras \cite{s}.

\begin{proposition}
Let $B$ be a monotone $\sigma-$complete $C^{\ast}$-algebra. Then there exists
a $\sigma-$normal homomorphism, $\pi,$ from $B^{\infty}$ onto $B,$ such that
$\pi(b)=b$ for every $b\in B.$ Let $J$ be the kernel of the homomorphism
$\pi.$ Then $J$ is a $\sigma-$ideal of $B^{\infty}$and $B=B^{\infty}/J.$
\end{proposition}

\begin{corollary}
Let $A$ and $B$ be $C^{\ast}$-algebras and let $B$ be monotone $\sigma
-$complete. Let $\phi:A\rightarrow B$ be a positive linear map. Then $\phi$
has a unique extension to a $\sigma-$normal positive linear map,
$\widehat{\phi},$ from $A^{\infty}$ into $B.$ When $\phi$ is a $\ast
$-homomorphism the following hold: First $\widehat{\phi}$ is also a $\ast
$-homomorphism. Secondly, the range of $\widehat{\phi}$ is a $\sigma-$
subalgebra of $B.$ Thirdly, the self adjoint part of $\widehat{\phi}%
[A^{\infty}]$\ is the smallest $\sigma-$subspace of $B_{sa}$ which contains
$\phi\lbrack A_{sa}]. $ Finally, the kernel of $\widehat{\phi}$ is a $\sigma
-$ideal, $J$, such that \ \ $A^{\infty}/J\approx\widehat{\phi}[A^{\infty}].$
\end{corollary}

For a proof, see Proposition 1.1 \cite{zz}.

\ 

REMARK Let $S$ be a subset of a monotone $\sigma-$complete $C^{\ast}$-algebra
$B.$ Let $A$ be the smallest (unital) $C^{\ast}$-subalgebra of $B$ which
contains $S.$ Let $\phi$ be the inclusion map from $A$ into $B$. By applying
the preceding result, the smallest $\sigma-$subspace of $B_{sa}$ which
contains $A_{sa}$ is the self adjoint part of a $\sigma-$subalgebra $C $ of
$B.$ It is now natural to describe $C$ as the $\sigma-$subalgebra of $B $
which is $\sigma-$generated by $S.$

\section{Extending continuous functions\textbf{\ }}

We gather together some topological results which will be useful later. The
most important of these is Theorem 4.7. We hope the presentation here is clear
enough to ensure that the reader can reconstruct any missing proofs without
difficulty. If we have misjudged this, we apologise and refer the reader to
\cite{q}, see also \cite{b}.

Throughout this section, $K$ is a compact Hausdorff space and $D$ is a dense
subset of $K$, equipped with the relative topology induced by $K.$ It is easy
to see that $K$ has no isolated points if, and only if $D$ has no isolated points.

Let us recall that a topological space $T$ is \textit{extremally disconnected}
if the closure of each open subset is still an open set.

When $K$ is extremally disconnected then, whenever $Z$ is a compact Hausdorff
space and $f:D\rightarrow Z$ is continuous, there exists a unique extension of
$f$ to a continuous function $F:K\rightarrow Z.$ In other words, $K$ is the
Stone-Czech compactification of $D.$ (This is Theorem 4.7.)

For any compact Hausdorff space $K,$ the closed subsets of $D,$ in the
relative topology, are all of the form $F\cap D$ where $F$ is a closed subset
of $K.$ For any $S\subset K,$ we denote the closure of $S$ (in the topology of
$K$) by $cl(S).$ For $S\subset D,$ we note that the closure of this set in the
relative topology of $D$ is $cl(S)\cap D.$ We denote this by $cl_{D}(S).$ We
also use $intS$ for the interior of $S$ and, when $S\subset D,$ the interior
with respect to the relative topology is denoted by $int_{D}S.$

The following lemmas are routine point-set topology.

\begin{lemma}
Let $K$ be a compact Hausdorff space and $D$ a dense subset of $K.$

(i) \textit{For any open subset }$U$\textit{\ of }$K$\textit{\ we have}
$cl(U)=cl(U\cap D).$

(ii) \textit{Let }$U,V$\textit{\ be open subsets of }$K.$\textit{Then
}$V\subset Cl(U)$\textit{\ if, and only if,}

$V\cap D\subset Cl(U\cap D)\cap D=cl_{D}(U\cap D).$

(iii) \textit{Let }$U$ be an open subset of $K.$ Then

$D\cap int(clU)=int_{D}(cl_{D}(U\cap D)).$

(iv) If $U$ is a regular open subset of $K$ then $U\cap D$ is a regular open
subset of $D$ in the relative topology of $D.$ Conversely, if $E$ is a regular
open subset of $D$ in the relative topology, then $E=V\cap D$ where $V$ is a
regular open subset of $K.$
\end{lemma}

For any topological space $Y$ we let $RegY$ denote the Boolean algebra of
regular open subsets of $Y$.

Let $H:\mathcal{P}(K)\rightarrow\mathcal{P(}D)$ be defined by $H(S)=S\cap D. $

\begin{lemma}
The function $H,$ when restricted to $RegK,$ is a Boolean isomorphism of
$RegK$ onto $RegD.$
\end{lemma}

\begin{lemma}
A Hausdorff topological\textbf{\ } space $T$ is extremally disconnected if,
and only if, each regular open set is closed, and hence clopen.
\end{lemma}

\begin{corollary}
Let $D$ be a dense subset of a compact Hausdorff extremally disconnected space
$S.$ Let $D$ be equipped with the relative topology. Then $D$ is an extremally
disconnected space.
\end{corollary}

\begin{proof}
Let $V$ be a regular open subset of $D.$ Then, by Lemma 4.1, part (iv), there
exists $U,$ a regular open subset of $S,$ such that $V=U\cap D.$ By Lemma 4.3,
$U$ is a clopen subset of $S.$ Hence $V$ is a clopen subset of $D$ in the
relative topology. Again appealing to Lemma 4.3, we have that $D$ is an
extremally disconnected space.
\end{proof}

\begin{lemma}
Let $D$ be an extremally disconnected topological space. Also let $D$ be
homeomorphic to a subspace of a compact Hausdorff space. Then $\beta D$, its
Stone-Czech compactification, is extremally disconnected.
\end{lemma}

\begin{lemma}
Let $D$ be a dense subspace of a compact Hausdorff extremally disconnected
space $Z$. When $A$ is a clopen subset of $D$ in the relative topology, then
$clA$ is a clopen subset of $Z.$\ Let $A$ and $B$ be disjoint clopen subsets
of $D,$ in the relative topology. Then $clA$ and $clB$ are disjoint clopen
subsets of $Z.$
\end{lemma}

The following result was given by Gillman and Jerison, see page 96 \cite{q},
as a byproduct of other results. The argument given here may be slightly
easier and more direct.

\begin{theorem}
Let $D$ be a dense subspace of a compact Hausdorff extremally disconnected
space $S$. Then $S$ is the Stone-Czech compactification of $D.$ More
precisely, there exists a unique homeomorphism from $\beta D$ onto $S$ which
restricts to the identity homeomorphism on $D.$
\end{theorem}

\begin{proof}
Since $D$\textbf{\ }is a subspace of the compact Hausdorff space,\textbf{\ }%
$S$, $D$\textbf{\ }is completely regular and hence has a well defined
Stone-Czech compactification. By the fundamental property of $\beta D$, there
exists a unique continuous surjection $\alpha$ from $\beta D$ onto $S, $ which
restricts to the identity on $D.$

Let $a$ and $b$ be distinct points in $\beta D.$ Then there exist disjoint
clopen sets $U$ and $V$ such that $a\in U$ and $b\in V.$ Let $A=U\cap D$ and
$B=V\cap D$ then $U=cl_{\beta D}A$ and $V=cl_{\beta D}B.$ So $\alpha\lbrack
U]\subset cl_{S}A$ and $\alpha\lbrack V]\subset cl_{S}B.$ By Lemma 4.6,
$cl_{S}A$ and $cl_{S}B$ are disjoint. Hence $\alpha(a)$ and $\alpha(b)$ are
distinct points of $S.$ Thus $\alpha$ is injective. It now follows from
compactness, that $\alpha$ is a homeomorphism.
\end{proof}

\section{Ergodic discrete group actions on topological spaces}

In this section, $Y$ is a Hausdorff topological space which has no isolated
points. For example, a compact Hausdorff space with no isolated points, or a
dense subset of such a space.

When $G$ is a group of bijections of $Y,$ and $y\in Y,$ we denote the orbit
$\{g(y):g\in G\}$ by $G[y].$

\begin{lemma}
\textit{Let }$G$\textit{\ be a countable group of homeomorphisms of }%
$Y$\textit{.}

\textit{(i) If there exists }$x_{0}\in Y$\textit{\ such that the orbit
}$G[x_{0}]$\textit{\ is dense in }$Y$\textit{\ then every }$G-$%
\textit{invariant open subset of }$Y$\textit{\ is either empty or dense.}

\textit{(ii) If every non-empty open }$G-$invariant \textit{subset of }%
$Y$\textit{\ is dense then, for each }$x$\textit{\ in }$Y,$\textit{\ the orbit
}$G[x]$\textit{\ is either dense or nowhere dense.}
\end{lemma}

\begin{proof}
(i) Let $U$ be a $G-$invariant open set which is not empty. Since $G[x_{0}]$
is dense, for some $g\in G,$we have $g(x_{0})\in U.$ But $U$ is $G-$invariant.
So $x_{0}\in U.$ Hence $G[x_{0}]\subset U.$ So $U$ is dense in $Y. $

(ii) Suppose $y$ is an element of $Y$ such that $G[y]$ is not dense in
$Y.$\ Then $Y\backslash clG[y]$ $\ $is a non-empty $G-$invariant open set. So
it is dense in $Y.$ So $clG[y]$ has empty interior.
\end{proof}

When $G$ is a group of homeomorphisms of $Y$ its action is said to be
\textit{ergodic} if each $G-$invariant open subset of $Y$ is either empty or
dense in $Y.$

\begin{lemma}
Let $Y$ be an extremally disconnected space. Let $G$ be a group of
homeomorphisms of $Y.$\ Then the action of $G$ is ergodic, if, and only if,
the only $G-$invariant clopen subsets are $Y$ and $\varnothing.$
\end{lemma}

\begin{proof}
Let $U$ be a $G-$invariant open set. Then $clU$ and $Y\backslash clU$ are $G-
$invariant clopen sets. Then $U$ is neither empty nor dense, if, and only if,
$clU$ and $Y\backslash clU$ are non-trivial clopen sets.
\end{proof}

\section{Induced actions}

Let $X$ be a compact Hausdorff space. Then, see Lemma 13 \cite{zk}, $X$ is
separable if, and only if $C(X)$ is isomorphic to a closed (unital)
*-subalgebra of $\ell^{\infty}.$

The regular $\sigma$-completion of an arbitrary $C^{\ast}$-algebra was defined
in \cite{zz}. For the commutative algebra, $C(X),$ its regular $\sigma
$-completion can be identified with the monotone $\sigma$-complete $C^{\ast}%
$-algebra $B_{0}^{\infty}(X)/M_{0}(X),$ where $B_{0}^{\infty}(X)$ is the
algebra of bounded Baire measurable functions on $X$ and $M_{0}(X)$ is the
ideal of all $f$ in $B_{0}^{\infty}(X)$ for which $\{x:f(x)\neq0\}$ is meagre.
Let $S$ be the structure space of $B_{0}^{\infty}(X)/M_{0}(X)$ i.e. this
algebra can be identified with $C(S).$

Let $j:C(X)\rightarrow B_{0}^{\infty}(X)/M_{0}(X)$ be the natural embedding.
This is an injective (isometric) *-homomorphism.

Suppose that $X$ is separable. Then there exists an injective *-homomorphism
$h:C(X)\rightarrow\ell^{\infty}.$ Since $\ell^{\infty}$ is monotone complete,
$h$ extends to a homomorphism $H:C(S)\rightarrow\ell^{\infty}.$ From standard
properties of regular $\sigma-$completions \cite{zz}, $H$ is also injective.
Hence $C(S)$ supports a strictly positive linear functional. By Lemma 3.1 it
follows that $C(S)$ is monotone complete and hence $S$ is extremally disconnected.

Since $H$ is an injective homomorphism, it follows that there is a surjective
continuous map from $\beta%
%TCIMACRO{\U{2115} }%
%BeginExpansion
\mathbb{N}
%EndExpansion
$ onto $S.$ So $S$ is separable.

\ 

REMARK Because $C(S)$ is monotone complete, $S$ is the Stone structure space
of the complete Boolean algebra of regular open subsets of $X$. It follows
from the Birkhoff-Ulam Theorem in Boolean algebras, and the linearisation
arguments in \cite{zza} that $C(S)$ can be identified with $B^{\infty
}(X)/M(X)$ where $B^{\infty}(X)$ is the algebra of bounded Borel measurable
functions on $X$ and $M(X)$ is the ideal of all bounded Borel functions with
meagre support. In other words, $B^{\infty}(X)/M(X)$ is isomorphic to
$B_{0}^{\infty}(X)/M_{0}(X)$ which is isomorphic to $C(S).$

By the usual duality between compact Hausdorff spaces and commutative (unital)
$C^{\ast}$-algebras, there is a continuous surjection $\rho$ from $S $ onto
$X$ such that $j(f)=f\circ\rho$ for each $f$ in $C(X).$

By the basic properties of regular $\sigma$-completions, for each self-adjoint
$b$ in $C(S),$ the set

$\{j(a):a\in C_{%
%TCIMACRO{\U{211d} }%
%BeginExpansion
\mathbb{R}
%EndExpansion
}(X)$ and $j(a)\leq b\}$ has $b$ as its least upper bound in $C(S)_{sa}=C_{%
%TCIMACRO{\U{211d} }%
%BeginExpansion
\mathbb{R}
%EndExpansion
}(S).$

\begin{lemma}
Let $Y$ be a subset of $S$ such that $\rho\lbrack Y]$ is dense in $X.$ Then
$Y$ is dense in $S.$
\end{lemma}

\begin{proof}
Let us assume that $Y$ is not dense in $S.$ Then there exists a non-empty
clopen set $E$ which is disjoint from $clY.$

Let $j(a)\leq\chi_{E}.$ Then $j(a)(s)\leq0$ for $s\in Y.$ So $a(\rho(s))\leq0$
for $\rho(s)\in\rho\lbrack Y].$ Hence $a\leq0.$ But this implies $\chi_{E}%
\leq0$ which is a contradiction.
\end{proof}

Let $Y$ be any compact Hausdorff space. Let $Homeo(Y)$ be the group of all
homeomorphisms from $Y$ onto $Y$. Let $AutC(Y)$ be the group of all $\ast
-$automorphisms of $C(Y)$. For $\phi\in Homeo(Y)$ let $h_{\phi}(f)=f\circ\phi$
for each $f\in C(Y).$ Then $\phi\rightarrow h_{\phi}$ is a bijection from the
group $Homeo(Y)$ onto $AutC(Y)$ which switches the order of multiplication. In
other words it is a group anti-isomorphism.

Let $\theta$ be a homeomorphism of $X$ onto $X.$ As above, let $h_{\theta}$ be
the corresponding $\ast-$automorphism of $C(X).$ Also $f\rightarrow
f\circ\theta$ \ induces an automorphism $\widehat{h_{\theta}}$ of $B^{\infty
}(X)/M(X).$ Since $B^{\infty}(X)/M(X)$ can be identified with $C(S), $ there
exists $\widehat{\theta}$ in\ $Homeo(S)$ corresponding to $\widehat{h_{\theta
}}$ $.$ Clearly, $\ \widehat{h_{\theta}}$ restricts to the automorphism,
$h_{\theta},$ of $C(X).$

\begin{lemma}
$\widehat{h_{\theta}}$ is the unique automorphism of $C(S)$ which is an
extension of $h_{\theta}.$ Hence $\widehat{\theta}$ is uniquely determined by
$\theta.$ Furthermore, the map $\theta\rightarrow\widehat{\theta\text{ }}$ is
an injective group homomorphism from $Homeo(X)$ into $Homeo(S).$\textbf{\ }
\end{lemma}

\begin{proof}
Let $H$ be an automorphism of $B^{\infty}(X)/M(X)=C(S)$, which is an extension
of $h_{\theta}$. Let $b$ be a self-adjoint element of $B^{\infty}(X)/M(X).$
Then, for $a\in C_{%
%TCIMACRO{\U{211d} }%
%BeginExpansion
\mathbb{R}
%EndExpansion
}(X),a\leq b$ if, and only if, $Ha\leq$ $Hb$ i.e. $h_{\theta}a\leq Hb.$ So
$Hb$ is the supremum of $\{h_{\theta}(a):a\in C_{%
%TCIMACRO{\U{211d} }%
%BeginExpansion
\mathbb{R}
%EndExpansion
}(X),$ $a\leq b\}.$ Hence $H=\widehat{h_{\theta}}.$ That is,
$\widehat{h_{\theta}}$ $\ $is the unique extension of $h_{\theta}$ to an
automorphism of $C(S).$

Let $h_{1}$ and $h_{2}$ be in $AutC(X).$ Then for $a\in C(X),$ we have%
\[
\widehat{h_{1}h_{2}}(a)=h_{1}h_{2}(a)=\widehat{h_{1}}h_{2}(a)=\widehat{h_{1}%
}\widehat{h_{2}}(a).
\]

By uniqueness, it now follows that $\widehat{h_{1}h_{2}}=\widehat{h_{1}%
}\widehat{h_{2}}.$ Hence $h\rightarrow\widehat{h}$ is an injective group
homomorphism of $AutC(X)$ into $AutC(S).$ So the map $\theta\rightarrow
\widehat{\theta}$ is the composition of a group anti-isomorphism with an
injective group homomorphism composed with a group anti-isomorphism. So it is
an injective group homomorphism.
\end{proof}

\begin{corollary}
$\theta(\rho s)=\rho(\widehat{\theta}s)$ for each $s\in S.$
\end{corollary}

\begin{proof}
For $a\in C(X),\ s\in S,$%
\[
a\circ\theta(\rho s)=h_{\theta}(a)(\rho s)=\widehat{h_{\theta}}%
(j(a)(s)=j(a)(\widehat{\theta}s)=a(\rho(\widehat{\theta}s)).
\]
Hence $\theta(\rho s)=\rho(\widehat{\theta}s).$
\end{proof}

Throughout this paper, unless we specify otherwise, $G$ is a countable
infinite group. Let $\varepsilon:G\rightarrow Homeo(X)$ be a homomorphism into
the group of homeomorphisms of $X.$ That is, $\varepsilon$ is an action of $G$
on $X.$ For each $g\in G,$ let $\widehat{\varepsilon}_{g}$ be the
homeomorphism of $S$ onto $S$ induced by $\varepsilon_{g}.$ Then
$\widehat{\varepsilon}$ is the action of $G$ on $S$ induced by $\varepsilon.
$\bigskip

Let us recall that an action $\varepsilon:G\rightarrow Homeo(X)$ is
\textit{non-degenerate }if it is injective. We shall normally only use
non-degenerate actions.

\begin{proposition}
Let $x_{0}$ be a point in $X$ such that the orbit $\{\varepsilon_{g}%
(x_{0}):g\in G\}$ is dense in $X.$ Let $s_{0}\in S$ such that $\rho
s_{0}=x_{0}.$ Then $\{\widehat{\varepsilon}_{g}(s_{0}):g\in G\}$is an orbit
which is dense in $S.$
\end{proposition}

\begin{proof}
By Corollary 6.3, $\varepsilon_{g}(x_{0})=\rho(\widehat{\varepsilon}_{g}%
(s_{0})).$ It now follows from Lemma 6.1, that the orbit $\{
\widehat{\varepsilon}_{g}(s_{0}):g\in G\}$ is dense in $S.$
\end{proof}

\begin{definition}
An orbit $\{\varepsilon_{g}(x_{0}):g\in G\}$ is said to be \textbf{free}%
\textit{\ }if, for $g\neq\imath$, $\varepsilon_{g}(x_{0})\neq x_{0}.$
Equivalently, for $g\neq\imath,$ $\varepsilon_{g}$ leaves no point of the
orbit fixed.
\end{definition}

It is easy to see that the existence of at least one free orbit implies that
the action is non-degenerate.

\begin{definition}
Let $Y$ be a subset of $X$ which is invariant under the action $\varepsilon$.
Then the action $\varepsilon$ is \textbf{free} on $Y$ if, for each $y\in Y, $
the orbit $\{ \varepsilon_{g}(y):g\in G\}$is free.
\end{definition}

\begin{lemma}
Let $G,X$ and $\varepsilon$ be as above. Let $x_{0}\in X$ be such that the
orbit $\{\varepsilon_{g}(x_{0}):g\in G\}$ is both dense and free. Then there
exists a $G$-invariant $Y,$ which is a dense $G_{\delta}$ subset of $X $ such
that for $g\neq\imath,$ $\varepsilon_{g}$ has no fixed point in $Y. $ Also
$x_{0}\in Y.$
\end{lemma}

\begin{proof}
Fix $g\neq\imath,$ let $K_{g}=\{x\in X:\varepsilon_{g}(x)=x\}.$ Then $K_{g}, $
the fix-point set of $\varepsilon_{g},$ is closed. Let $U$ be the interior of
$K_{g}.$ Then the orbit $\{ \varepsilon_{h}(x_{0}):h\in G\}$ is disjoint from
$K_{g}.$ So its closure is disjoint from $U.$ But since the orbit is dense,
this means that $K_{g}$ has empty interior.

Let $Z=%
%TCIMACRO{\tbigcup }%
%BeginExpansion
{\textstyle\bigcup}
%EndExpansion
\{K_{g}:g\in G,g\neq\imath\}.$ Then $Z$ is the union of countably many closed
nowhere dense sets. A calculation shows that
\[
\varepsilon_{h}[K_{g}]=K_{hgh^{-1}}%
\]
and from this it follows that $Z$ is $G$-invariant.

Put $Y=X\backslash Z.$ Then $Y$ has all the required properties.
\end{proof}

\begin{theorem}
Let $G,X$ and $\varepsilon$ be as above. Let $\widehat{\varepsilon}$ be the
action of $G$ on $S$ induced by the action $\varepsilon$ on $X.$ Let $x_{0}\in
X$ such that the orbit $\{ \varepsilon_{g}(x_{0}):g\in G\}$ is both dense and
free. Let $s_{0}\in S$ such that $\rho s_{0}=x_{0}.$ Then $\{
\widehat{\varepsilon}_{g}(s_{0}):g\in G\}$ is a dense free orbit in $S.$

Furthermore, there exists $Y,$ a $G$-invariant, dense $G_{\delta}$ subset of
$S,$ with $s_{0}\in Y,$ such that the action $\widehat{\varepsilon}$ is free
on $Y.$
\end{theorem}

\begin{proof}
By Corollary 6.3, $\varepsilon_{g}(\rho s_{0})=\rho(\widehat{\varepsilon}%
_{g}s_{0}).$ That is, $\varepsilon_{g}(x_{0})=\rho(\widehat{\varepsilon}%
_{g}s_{0}).$

It now follows from Lemma 6.1 that the orbit $\{ \widehat{\varepsilon}%
_{g}(s_{0}):g\in G\}$ is dense in $S.$

Now suppose that $\widehat{\varepsilon}_{h}s_{0}=s_{0}.$ Then $\rho
(\widehat{\varepsilon}_{h}s_{0})=\rho(s_{0}).$ So $\varepsilon_{h}%
(x_{0})=x_{0}.$ Hence $h=\imath.$ It now follows that $\{\widehat{\varepsilon
}_{g}(s_{0}):g\in G\}$ is a dense free orbit in $S.$

The rest of the theorem follows by applying Lemma 6.7.
\end{proof}

REMARK Let $D$ be a countable dense subset of a compact Hausdorff space $K.$
Let $\alpha$ be a homeomorphism of $D$ onto $D.$ Then, in general, $\alpha$
need not extend to a homeomorphism of $K.$ But, from the fundamental
properties of the Stone-Czech compactification, $\alpha$ does extend to a
unique homeomorphism of $\beta D,$ say $\theta_{\alpha}.$ Let $S_{1}$ be the
Gelfand--Naimark structure space of $B^{\infty}(\beta D)/M(\beta D)$. Then,
from the results of this section, $\theta_{\alpha}$ induces a homeomorphism
$\widehat{\theta_{\alpha}}$ of $S_{1}.$ Let $S$ be the structure space of
$B^{\infty}(K)/M(K).$ Then, by Lemma 4.2, $S$ is homeomorphic to $S_{1}.$
Hence each homeomorphism of $D$ induces a canonical homeomorphism of $S.$ So
each action of $G$, as homeomorphisms of $D,$ induces, canonically, an action
of $G$ as homeomorphisms of $S.$

\section{Orbit equivalence}

\ Let $S$ be a compact Hausdorff extremally disconnected space with no
isolated points. Let $\varepsilon$ be an action of $G$ as homeomorphisms of
$S$ which is non-degenerate.

\begin{definition}
Let $Z$ be a $G$-invariant subset of $S.$ Then the action $\varepsilon$ is
said to be \textbf{pseudo-free} on $Z$ \textit{\ }if, for every $g\in G,$ the
fixed point set $\{z\in Z:\varepsilon_{g}(z)=z\}$ is a clopen subset of $Z$ in
the relative topology.
\end{definition}

REMARK\ If an action is free on $Z$ then, for $g\neq\imath,$ its fixed point
set is empty. So each free action is also pseudo-free. In particular, each
free orbit is also pseudo-free.

In the rest of this section, $s_{0}\in S$ such that\ the orbit $D=\{
\varepsilon_{g}(s_{0}):g\in G\}$ is dense in $S.$ To simplify our notation, we
shall write "$g$" for $\varepsilon_{g}.$ The restriction of $g$ to $D$ is a
homeomorphism of $D$ onto $D.$ We shall abuse our notation by also denoting
this restriction by "$g$".

From the results of Section 4, $S$ is the Stone-Czech compactification of $D.
$ So any homeomorphism of $D$ has a unique extension to a homeomorphism of
$S.$

\begin{lemma}
Let $O$ be a non-empty open subset of $S.$ Then $O\cap D$ is an infinite set.
\end{lemma}

\begin{proof}
Suppose $O\cap D$ is a finite set, say, $\{p_{1},p_{2},\cdots,p_{n}\}.$

Then $O\setminus\{p_{1},p_{2},\cdots,p_{n}\}$ is an open subset of $S$ which
is disjoint from $D$. But $D$ is dense in $S$. Hence $O=\{p_{1},p_{2}%
,\cdots,p_{n}\}.$ So $\{p_{1}\}$ is an open subset of $S.$ But $S$ has no
isolated points. So this is a contradiction.
\end{proof}

Let $Z$ be a $G$-invariant dense subset of $S$ and let $h$ be a bijection of
$Z$ onto itself. Then $h$ is said to be \textit{strongly }$G$%
-\textit{decomposable} over $Z$ if there exist a sequence of pairwise disjoint
clopen subsets of $Z,$ $(A_{j})$ where $Z=\bigcup A_{j},$ and a sequence
$(g_{j})$ in $G$ such that
\[
h(x)=g_{j}(x)\text{ for }x\in A_{j}.
\]
When this occurs, $h$ is a continuous, open map. Hence it is a homeomorphism
of $Z$ onto $Z.$

We also need a slightly weaker condition. Let $h$ be a homeomorphism of $S$
onto itself. Then $h$ is $G$-\textit{decomposable} (over $S)$ if there exist a
sequence of pairwise disjoint clopen subsets of $S,$ $(K_{j})$ where $\bigcup
K_{j}$ is dense in $S,$ and a sequence $(g_{j})$ in $G$ such that%

\[
h(x)=g_{j}(x)\text{ for }x\in K_{j}.
\]

REMARK The set $\bigcup K_{j}$ is an open dense set, hence its compliment is a
closed nowhere dense set.

\begin{lemma}
Let $h$ be a homeomorphism of $D$ onto $D.$ Let $h$ be strongly $G$%
-decomposable over $D$. Let $\widehat{h}$ be the unique extension of $h$ to a
homeomorphism of $S.$ Then $\widehat{h}$ is $G$-decomposable over $S.$
\end{lemma}

\begin{proof}
Let $(A_{j})$ be a sequence of pairwise disjoint clopen subsets of $D$. Then,
by Lemma 4.6, $(clA_{j})$ is a sequence of pairwise disjoint clopen subsets of
$S.$ Let $(g_{j})$ be a sequence in $G$ such that $h(x)=g_{j}(x)$ for $x\in
A_{j}.$ Then, by continuity, $\widehat{h}(x)=g_{j}(x)$ for $x\in clA_{j}.$
Also the open set $\bigcup clA_{j}$ is dense in $S$ (because it contains $D).$
\end{proof}

Let $\Gamma$ be a countable, infinite group of homeomorphisms of $S$ which
acts transitively on $D.$

If each $\gamma\in$ $\Gamma$ is strongly $G$-decomposable over $D$ \ and each
$g\in G$ is strongly $\Gamma$-decomposable over $D$ then $\Gamma$ and $G $ are
said to be \textit{strongly equivalent.}

Let us recall that orbit equivalence, with respect to the action of $G$ on $S
$, is defined by
\[
x\sim_{G}y\text{ if, and only if }g(x)=y\text{ for some }g\text{ in }G.
\]

\begin{lemma}
Let $\Gamma$ and $G$ be strongly equivalent over $D$. Then there exists a
$G_{\delta}$ set $Y,$ where $D\subset Y\subset S$, and $Y$ is both
$G$-invariant and $\Gamma$-invariant, such that $\Gamma$ and $G$ are strongly
equivalent over $Y.$
\end{lemma}

\begin{proof}
Let $\Lambda$ be the countable group generated by $\Gamma$ and $G.$

By Lemma 7.3, for each $\gamma\in\Gamma,$ there is an open subset of $S$,
$U_{\gamma},$ such that $D\subset U_{\gamma}$ and $\gamma$ decomposes with
respect to $G,$ over $U_{\gamma}.$ Similarly, for each $g\in G,$ there is a
corresponding open subset of $S$, $V_{g},$ such that $D\subset V_{g}$ and $g$
decomposes with respect to $\Gamma,$ over $V_{g}.$

Let $W$ be the intersection of all the $U_{\gamma}$ and all the $V_{g}$ then
$W$ is a $G_{\delta}$ subset of $S$ such that $D\subset W.$ Now let $Y$ be the
intersection of $\{\lambda\lbrack W]:\lambda\in\Lambda\}.$ Then $Y $ is the
required $G_{\delta}$ set.
\end{proof}

\begin{corollary}
Let $\Gamma$ and $G$ be strongly equivalent over $D$. Then there exists a
$G_{\delta}$ set $Y,$ where $D\subset Y\subset S$, and $Y$ is both
$G$-invariant and $\Gamma$-invariant, such that the orbit equivalence
relations $\sim_{G}$ and $\sim_{\Gamma}$ coincide on $Y.$
\end{corollary}

\begin{proof}
Straightforward.
\end{proof}

\begin{lemma}
Let each $\gamma$ in $\Gamma$ be strongly $G$-decomposable over $D.$ Let the
action of $G$ be pseudo-free on $D.$ Then $\Gamma$ and $G$ are strongly
equivalent over $D.$
\end{lemma}

\begin{proof}
Let $g\in G.$ Since $\Gamma$ acts transitively on \ $D$, there exists
$\gamma_{1}$ in $\Gamma$ such that $g(s_{0})=\gamma_{1}(s_{0}).$

Since $\gamma_{1}$ is strongly $G$-decomposable over $D,$ there exist a clopen
set $A_{1}\subset D$ with $s_{0}\in A_{1}$ and $g_{1}$ in $G,$ such that
\[
\gamma_{1}(s)=g_{1}(s)\text{ for all }s\text{ in }A_{1}.
\]
In particular, $g(s_{0})=g_{1}(s_{0}).$ So $g_{1}^{-1}g(s_{0})=s_{0}.$ But the
action of $G$ is pseudo-free. So there exists a clopen neighbourhood of
$s_{0},$ $K_{1}\subset A_{1}$ such that $g_{1}^{-1}g(s)=s$ for $s\in K_{1}.$
Then $g(s)=\gamma_{1}(s)$ for $s\in K_{1}.$

Since $D$ is countable, we can find a sequence of disjoint clopen sets,
$(K_{j})$ and a sequence $(\gamma_{j})$ in $\Gamma$ such that $g(s)=\gamma
_{j}(s)$ for $s\in K_{j},$ and $\bigcup K_{j}=D.$ Thus $g$ is strongly
$\Gamma$-decomposable over $D.$ So $\Gamma$ and $G$ are strongly equivalent.
\end{proof}

\begin{lemma}
Let $A$ and $B$ be disjoint clopen subsets of $D.$ Let $a\in A$ and $b\in B.$
Then there exists a homeomorphism $h$ from $D$ onto $D$ with the following
properties. First $h$ is strongly $G$-decomposable. Secondly $h$ interchanges
$A$ and $B$ and leaves each point of $D\backslash(A\cup B)$ fixed. Thirdly,
$h(a)=b.$ Fourthly $h=h^{-1}.$
\end{lemma}

\begin{proof}
Since $a$ and $b$ are in the same orbit of $G,$ there exists $g_{1}$ in $G$
such that $g_{1}(a)=b.$ Then $A\cap g_{1}^{-1}[B]$ is a clopen neighbourhood
of $a$ which is mapped by $g_{1\text{ }}$into $B.$ Since $S$ is extremally
disconnected and has no isolated points and by making use of Lemma 4.6, we can
find a strictly smaller clopen neighbourhood of $a,$ say $A_{1}.$ By dropping
to a clopen sub-neighbourhood if necessary, we can also demand that
$g_{1}[A_{1}]$ is a proper clopen subset of $B.$ Let $B_{1}=g_{1}[A_{1}].$

By Lemma 7.2, $A$ and $B$ are infinite sets. Since they are subsets of $D,$
they are both countably infinite. Enumerate them both. Let $a_{2}$ be the
first term of the enumeration of $A$ which is not in $A_{1}$ and let $b_{2}$
be the first term of the enumeration of $B$ which is not in $B_{1}.$ Then
there exists $g_{2}$ in $G$ such that $g_{2}(a_{2})=b_{2}.$ Now let $A_{2}$ be
a clopen neighbourhood of $a_{2},$ such that $A_{2}$ is a proper subset of
$A\setminus A_{1}$ and $g_{2}[A_{2}]$ is a proper subset of $B\setminus
B_{1}.$ Proceeding inductively, we obtain a sequence, $(A_{n})$ of disjoint
clopen subsets of $A$; a sequence $(B_{n})$ of disjoint clopen subsets of $B$
and a sequence $(g_{n})$ from $G$ such that $g_{n}$ maps $A_{n}$ onto $B_{n}.
$ Furthermore $A=\bigcup A_{n}$ and $B=\bigcup B_{n}.$

We define $h$ as follows. For $s\in A_{n},$ $h(s)=g_{n}(s).$ For $s\in B_{n},
$ $h(s)=g_{n}^{-1}(s).$ For $s\in D\setminus(A\cup B),h(s)=s$. Then $h $ has
all the required properties.
\end{proof}

\begin{lemma}
Let $\alpha$ and $\beta$ be homeomorphisms of $D$ onto itself. Suppose that
each homeomorphism is strongly $G$-decomposable. Then $\beta\alpha$ is
strongly $G$-decomposable.
\end{lemma}

\begin{proof}
Let $\{A_{i}:i\in\mathbb{N}\}$ be a partition of $D$ into clopen sets and
$(g_{i}^{\alpha})$ a sequence in $G$ which gives the $G$-decomposition of
$\alpha.$ Similarly, let $\{B_{j}:j\in\mathbb{N}\}$ be a partition of $D$ into
clopen sets and $(g_{j}^{\beta})$ a sequence in $G$ which gives the
$G$-decomposition of $\beta.$ Then $\{A_{i}\cap\alpha^{-1}[B_{j}%
]:i\in\mathbb{N},j\in\mathbb{N}\}$ is a partition of $D$ into clopen sets.

Let $s\in A_{i}\cap\alpha^{-1}[B_{j}].$ Then $\beta\alpha(s)=g_{j}^{\beta
}(\alpha(s))=g_{j}^{\beta}g_{i}^{\alpha}(s).$
\end{proof}

\begin{lemma}
Let $D$ be enumerated as $(s_{0},s_{1},...).$ Let $(D_{k})(k=1,2...)$ be a
monotone decreasing sequence of clopen neighbourhoods of $s_{0}$ such that
$s_{n}\notin D_{n}$ for any $n$. Then the following statements hold.

(a) There is a sequence $(h_{k})(k=1,2...)$ of homeomorphisms of $D$ onto $D$
where $h_{k}=h_{k}^{-1}.$ \ For $1\leq k\leq n,$ the $h_{k}$ are mutually
commutative. Each $h_{k\text{ }}$is strongly $G-$decomposable over $D.$

(b) For each positive integer $n,$ there exists a finite family of pairwise
disjoint, clopen subsets of $D,$
\[
\{K^{n}(\alpha_{1},\alpha_{2},...,\alpha_{n}):(\alpha_{1},\alpha_{2}%
,...\alpha_{n})\in{\mathbb{Z}_{2}}^{n}\}
\]
whose union is $D.$

(c) Let $K_{0}=D.$ For $1\leq p\leq n-1$ $K^{p}(\alpha_{1},\alpha
_{2},...,\alpha_{p})=K^{p+1}(\alpha_{1},\alpha_{2},...,\alpha_{p},0)\cup
K^{p+1}(\alpha_{1},\alpha_{2},...,\alpha_{p},1).$

(d) For $1\leq p\leq n,$ $K^{p}(0,0,...,0)\subset D_{p}$ and $s_{0}\in
K^{p}(0,0,...,0).$

(e) Let $(\alpha_{1},\alpha_{2},...,\alpha_{p})\in{\mathbb{Z}_{2}}^{p}$ where
$1\leq p\leq n.$ Then the homeomorphism $h_{1}^{\alpha_{1}}h_{2}^{\alpha_{2}%
}...h_{p}^{\alpha_{p}}$ interchanges $K^{p}(\beta_{1},\beta_{2},...,\beta
_{p})$ with $K^{p}(\alpha_{1}+\beta_{1},\alpha_{2}+\beta_{2},...,\alpha
_{p}+\beta_{p}).$

(f) For each $n$, $\{s_{0},s_{1},...,s_{n}\}\subset\{h_{1}^{\alpha_{1}}%
h_{2}^{\alpha_{2}}...h_{n}^{\alpha_{n}}(s_{0}):(\alpha_{1},\alpha
_{2},...,\alpha_{n})\in{\mathbb{Z}_{2}}^{n}\}.$

(g) For each $s\in D,$ if $h_{1}^{\alpha_{1}}h_{2}^{\alpha_{2}}...h_{n}%
^{\alpha_{n}}(s)=s$ then $\alpha_{1}=\alpha_{2}=...=\alpha_{n}=0.$
\end{lemma}

\begin{proof}
We give an inductive argument. First, let $A=D_{1}$ and let $B=D\backslash
D_{1}.$ By applying Lemma 7.7, there exists a homeomorphism $h_{1}$ of $D$
onto itself, where $h_{1}$ interchanges $D_{1}$ and $D\setminus D_{1},$ and
maps $s_{0}$ to $s_{1}.$ (So (f) holds for $n=1.$) Also $h_{1}=h_{1}^{-1}.$

For any $s\in D,$ $h_{1}(s)\ $and $s$ are elements of disjoint clopen sets.
Hence (g) holds for $n=1.$

Now let $K^{1}(0)=D_{1}$ and $K^{1}(1)=D\backslash D_{1}.$ Let us now suppose
that we have constructed the homeomorphisms $h_{1},h_{2},...,h_{n}$ and the
clopen sets
\[
\{K^{p}(\alpha_{1},\alpha_{2},...\alpha_{p}):(\alpha_{1},\alpha_{2}%
,...\alpha_{p})\in{\mathbb{Z}_{2}}^{p}\} \ \text{ for }p=1,2,...,n.
\]
We now need to make the $(n+1)$th step of the inductive construction. For some
$(\alpha_{1},\alpha_{2},...,\alpha_{n})\in\{0,1\}^{n}$, $s_{n+1}\in
K^{n}(\alpha_{1},\alpha_{2},...,\alpha_{n}).$ Let $c=h_{1}^{\alpha_{1}}%
h_{2}^{\alpha_{2}}...h_{n}^{\alpha_{n}}(s_{n+1}).$ Then $c\in K^{n}%
(0,0,...,0).$

If $c\neq s_{0}$ let $b=c.$ If $c=s_{0}$ then let $b$ be any other element of
$K^{n}(0,0,...,0).$ Now let $A$ be a clopen subset of $K^{n}(0,0,...,0)\cap
D_{n+1}$ such that $s_{0}\in A$ and $b\notin A.$ Let $B=K^{n}%
(0,0,...,0)\setminus A.$ We apply Lemma 7.7 to find a homeomorphism $h $ of
$D$ onto itself, which interchanges $A$ and $B,$ leaves every point outside
$A\cup B$ fixed, maps $s_{0}$ to $b$, and $h=h^{-1}.$ Also, $h$ is strongly
$G$-decomposable.

Let $K^{n+1}(0,0...,0)=A$ and $K^{n+1}(0,0...,1)=B.$ By construction, (d)
holds for $p=n+1.$ Let $K^{n+1}(\alpha_{1},\alpha_{2},...,\alpha_{p}%
,0)=h_{1}^{\alpha_{1}}h_{2}^{\alpha_{2}}...h_{n}^{\alpha_{n}}[A]$ and
$K^{n+1}(\alpha_{1},\alpha_{2},...,\alpha_{p},1)=h_{1}^{\alpha_{1}}%
h_{2}^{\alpha_{2}}...h_{n}^{\alpha_{n}}[B].$ Then (b) holds for $n+1$ and (c)
holds for $p=n.$

We now define $h_{n+1}$ as follows. For $s\in K^{n}(\alpha_{1},\alpha
_{2},...,\alpha_{n}),$
\[
h_{n+1}(s)=h_{1}^{\alpha_{1}}h_{2}^{\alpha_{2}}...h_{n}^{\alpha_{n}}%
hh_{1}^{\alpha_{1}}h_{2}^{\alpha_{2}}...h_{n}^{\alpha_{n}}(s).
\]
\textbf{Claim 1 } $h_{n+1}$ commutes with $h_{j}$ for $1\leq j\leq n.$

To simplify our notation we shall take $j=1,$ but the calculation works in
general, since each of $\{h_{r}:r=1,2...n\}$ commutes with the others.

Let $s\in D.$ Then $s\in K^{n}(\alpha_{1},\alpha_{2},...,\alpha_{n})$ for some
$(\alpha_{1},\alpha_{2},...,\alpha_{n})\in{\mathbb{Z}_{2}}^{n}.$

So $h_{1}s\in K^{n}(\alpha_{1}+1,\alpha_{2},..,\alpha_{n}).$ Then
\[
h_{n+1}(h_{1}s)=h_{1}^{\alpha_{1}+1}h_{2}^{\alpha_{2}}...h_{n}^{\alpha_{n}%
}hh_{1}^{\alpha_{1}+1}h_{2}^{\alpha_{2}}...h_{n}^{\alpha_{n}}(h_{1}s).
\]
So
\begin{align*}
h_{n+1}h_{1}(s)  &  =h_{1}h_{1}^{\alpha_{1}}h_{2}^{\alpha_{2}}...h_{n}%
^{\alpha_{n}}hh_{1}^{\alpha_{1}}h_{2}^{\alpha_{2}}...h_{n}^{\alpha_{n}}%
h_{1}h_{1}(s)\\
&  =h_{1}h_{n+1}(s).
\end{align*}
From this we see that $h_{n+1}$ commutes with $h_{1}$. Similarly $h_{n+1}$
commutes with $h_{j}$ for $2\leq j\leq n.$

\textbf{Claim 2 }$h_{n+1}$is $G$-decomposable.

By Lemma 7.8, $h_{1}^{\alpha_{1}}h_{2}^{\alpha_{2}}...h_{n}^{\alpha_{n}}%
hh_{1}^{\alpha_{1}}h_{2}^{\alpha_{2}}...h_{n}^{\alpha_{n}}$ is $G$%
-decomposable. So, on restricting to the clopen set $K^{n}(\alpha_{1}%
,\alpha_{2},...,\alpha_{n})$ this gives that $h_{n+1}$ is $G$-decomposable
over each $K^{n}(\alpha_{1},\alpha_{2},...,\alpha_{n}).$ Hence $h_{n+1}$ is
$G$-decomposable.

So, by Claim 1 and Claim 2, (a) holds for $n+1.$ It is straightforward to show
that (b),(c), (d) and (e) hold for $n+1.$

Now consider (f). Either $c=s_{0}$ in which case, $s_{0}=h_{1}^{\alpha_{1}%
}h_{2}^{\alpha_{2}}...h_{n}^{\alpha_{n}}(s_{n+1})$ which gives $s_{n+1}%
=h_{1}^{\alpha_{1}}h_{2}^{\alpha_{2}}...h_{n}^{\alpha_{n}}(s_{0}),$ or $c\neq
s_{0},$ in which case
\[
h_{n+1}(h_{1}^{\alpha_{1}}h_{2}^{\alpha_{2}}...h_{n}^{\alpha_{n}}%
(s_{n+1}))=h(h_{1}^{\alpha_{1}}h_{2}^{\alpha_{2}}...h_{n}^{\alpha_{n}}%
(s_{n+1}))=s_{0}.
\]
This gives $s_{n+1}=h_{n+1}h_{1}^{\alpha_{1}}h_{2}^{\alpha_{2}}...h_{n}%
^{\alpha_{n}}(s_{0}).$ Because the homeomorphisms commute, this gives (f) for
$n+1.$

Finally consider (g). Let $s\in D$ with $h_{1}^{\alpha_{1}}h_{2}^{\alpha_{2}%
}...h_{n+1}^{\alpha_{n+1}}(s)=s.$ If $\alpha_{n+1}=0$ then (g) implies
$\alpha_{1}=\alpha_{2}=...=\alpha_{n}=0.$ So now suppose $\alpha_{n+1}=1$. Let
$h_{1}^{\beta_{1}}h_{2}^{\beta_{2}}...h_{n}^{\beta_{n}}(s)\in K^{n}%
(0,0...,0).$ Then, since the $h_{r}$ all commute, we can suppose without loss
of generality that $s\in$ $K^{n}(0,0...,0).$ Then $h_{n+1}(s)=h_{1}%
^{\alpha_{1}}h_{2}^{\alpha_{2}}...h_{n}^{\alpha_{n}}(s).$ But $h_{n+1}$ maps
$K^{n}(0,0...,0)$ to itself and $h_{1}^{\alpha_{1}}h_{2}^{\alpha_{2}}%
...h_{n}^{\alpha_{n}}$ maps $K^{n}(0,0...,0)$ to $K^{n}(\alpha_{1},\alpha
_{2},...,\alpha_{n}).$ So $h_{n+1}(s)\in K^{n}(0,0...,0)\cap K^{n}(\alpha
_{1},\alpha_{2},...,\alpha_{n}).$ But this intersection is only non-empty if
$\alpha_{1}=\alpha_{2}=...=\alpha_{n}=0.$ So $h_{n+1}(s)=s.$ But $h_{n+1\text{
}}$acting on $K^{n}(0,0...,0),$ interchanges $K^{n+1}(0,0...,0)$ with
$K^{n+1}(0,0...,1)$.

So $s\in K^{n+1}(0,0...,1)\cap K^{n+1}(0,0...,0),$ which is impossible.
\end{proof}

Let us recall that $\bigoplus\mathbb{Z}_{2}$ is the direct sum of an infinite
sequence of copies of $\mathbb{Z}_{2}.$ So each element of the group is an
infinite sequence of 0s and 1s, with 1 occurring only finitely many
times.\textbf{\ } We sometimes refer to it as the Dyadic Group.

\begin{theorem}
Let $S$ be a compact Hausdorff extremally disconnected space with no isolated
points. Let $G$ be a countably infinite group. Let $\varepsilon:G\rightarrow
Homeo(S)$ be a non-degenerate action of $G$ as homeomorphisms of $S$. Let
$s_{0}$ be a point in $S$ such that the orbit $\{\varepsilon_{g}(s_{0}):g\in
G\}=D,$\ is dense and pseudo-free. Then there exist an action $\gamma
:\bigoplus\mathbb{Z}_{2}\rightarrow Homeo(S)$ and $Y$, a G-delta subset of $S$
with $D\subset Y,$ such that the following properties hold.

(1) The orbit $\{\gamma_{\delta}(s_{0}):\delta\in\bigoplus\mathbb{Z}_{2}\}$ is
free and coincides with the set $D.$

(2) The groups of homeomorphisms $\varepsilon\lbrack G]$ and $\gamma
\lbrack\bigoplus\mathbb{Z}_{2}]$ are strongly equivalent over $Y,$ which is
invariant under the action of both these groups.

(3) The orbit equivalence relations corresponding, respectively, to
$\varepsilon\lbrack G]$ and $\gamma\lbrack\bigoplus\mathbb{Z}_{2}]$ coincide
on $Y.$

(4) $\gamma$ is an isomorphism.
\end{theorem}

\begin{proof}
We replace $"G"$ by "$\varepsilon\lbrack G]$" in the statement of Lemma 7.9 to
find a sequence $(h_{r})$ of homeomorphisms of $D$ onto itself with the
properties listed in that lemma. Each $h_{r}$ has a unique extension to a
homeomorphism $\widehat{h}_{r}$ of $S$ onto itself. For each $\mathbf{\alpha}$
$\in\bigoplus\mathbb{Z}_{2},$ there exist a natural number $n$ and
$(\alpha_{1},\alpha_{2},...,\alpha_{n})\in{\mathbb{Z}_{2}}^{n}$ such that
$\mathbf{\alpha}=(\alpha_{1},\alpha_{2},...\alpha_{n},0,0,...).$ We define
$\gamma_{\mathbf{\alpha}}\mathbf{=}\widehat{h}_{1}^{\alpha_{1}}\widehat{h}%
_{2}^{\alpha_{2}}...\widehat{h}_{n}^{\alpha_{n}}.$ Then $\gamma$ is a
homomorphism of $\bigoplus\mathbb{Z}_{2}$ into $Homeo(S).$

By Lemma 7.9 (f), the orbit $\{\gamma_{\mathbf{\alpha}}(s_{0}):\mathbf{\alpha
}\in\bigoplus\mathbb{Z}_{2}\}$ coincides with $D.$ By Lemma 7.9 (g), this
orbit is free for the action $\gamma.$

By Lemma 7.9 and Lemma 7.4, there exists a $G_{\delta}$ set $Y$, $D\subset
Y\subset S,$ which is invariant under the action of both $\gamma
\lbrack\bigoplus\mathbb{Z}_{2}]$ and $\varepsilon\lbrack G]$ also
$\varepsilon\lbrack G]$ and $\gamma\lbrack\bigoplus\mathbb{Z}_{2}]$ are
strongly equivalent over $Y.$ The statement (3) now follows from Corollary
7.5. Finally, statement (4) follows from part (g) of Lemma 7.9.\bigskip
\end{proof}

We shall make no direct use of $%
%TCIMACRO{\U{2124} }%
%BeginExpansion
\mathbb{Z}
%EndExpansion
-$actions in this article but the following corollary seems worth including.
We sketch an argument which makes use of the above result and the notation
used in Lemma 7.9.

\begin{corollary}
There exists a homeomorphism $\phi:S\rightarrow S$ and a dense $G_{\delta}%
$-subset $S_{0}\subset S$ with the following properties. First, $S_{0}$ is
invariant under the action of $G.$ Secondly, $\phi\lbrack S_{0}]=S_{0}.$
Thirdly, the orbit equivalence relation coming from $G$ and the $%
%TCIMACRO{\U{2124} }%
%BeginExpansion
\mathbb{Z}
%EndExpansion
-$orbit equivalence relation coming from $\phi,$ coincide on $S_{0}.$

\begin{proof}
(Sketch) It follows from the preceding theorem that we may replace $G$ by
$\bigoplus\mathbb{Z}_{2}.$ More precisely, we shall let $G$ be the group of
homeomorphisms of $D$ generated by $(h_{r})(r=1,2...).$

Let $E_{1}=K^{1}(0)$ and $F_{1}=K^{1}(1).$ Let $E_{j+1}=K^{j+1}(\mathbf{1},0)
$ where $\mathbf{1}\in{\mathbb{Z}_{2}^{j}}$ and let $F_{j+1}=K^{j+1}%
(\mathbf{0},1)$ where $\mathbf{0}\in{\mathbb{Z}_{2}^{j}}.$

We observe that if $K^{n}(\alpha_{1},\alpha_{2},...,\alpha_{n})$ has non-empty
intersection with

$K^{n+p}(\beta_{1},\beta_{2},...,\beta_{n+p})$ then $\alpha_{1}=\beta
_{1},...,\alpha_{n}=\beta_{n}.$ From this it follows that $(E_{n})(n=1,2...)$
is a sequence of pairwise disjoint clopen subsets of $D.$ Similarly,
$(F_{n})(n=1,2...)$ is a sequence of pairwise disjoint clopen subsets of $D.$
We find that $D=%
%TCIMACRO{\dbigcup \limits_{n=1}}%
%BeginExpansion
{\displaystyle\bigcup\limits_{n=1}}
%EndExpansion
$ $E_{n}$ and $%
%TCIMACRO{\dbigcup \limits_{n=1}}%
%BeginExpansion
{\displaystyle\bigcup\limits_{n=1}}
%EndExpansion
F_{n}=D\backslash\{s_{0}\}.$

For $s\in E_{n}$ let $\phi(s)=h_{1}h_{2}...h_{n}(s).$ Then $\phi$ is a
continuous map of $E_{n}$ onto $F_{n}.$ From this it is straightforward to see
that $\phi$ is a continuous bijection of $D$ onto $D\backslash\{s_{0}\}.$
Similarly, $\phi^{-1}$ is a continuous bijection from $D\backslash\{s_{0}\}$
onto $D.$

Since $S$ has no isolated points, $D\backslash\{s_{0}\}$ is dense in $S.$ But,
see Section 4, $S$ can be identified with the Stone-Czech compactification of
any dense subset of itself. Applying this to $\phi^{-1}$ and $\phi$ we find
continuous extensions which are homeomorphisms of $S$ onto $S$ and which are
inverses of each other. We abuse notation and denote the extension of $\phi$
to the whole of $S$ by $\phi.$ Then $j\rightarrow\phi^{j}$ is the $%
%TCIMACRO{\U{2124} }%
%BeginExpansion
\mathbb{Z}
%EndExpansion
-$action considered here; let $\Delta$ be the group generated by $\phi$. We
shall further abuse our notation by writing $h_{r}$ for $\widehat{h_{r}}$, the
extension to a homeomorphism of $S.$

On applying Lemma 4.6, we see that $(clE_{n})(n=1,2...)$ is a sequence of
pairwise disjoint clopen subsets of $S.$ So its union is a dense open subset
of $\dot{S}$ which we shall denote by $O_{1}.$ By continuity, for $s\in
clE_{n}$ we have $\phi(s)=h_{1}^{1}h_{2}^{1}...h_{n}^{1}(s).$
Similarly,$(clF_{n})(n=1,2...)$ is a sequence of pairwise disjoint clopen
subsets of $S$ whose union, $O_{2},$ is also dense in $S.$

Let $\Gamma$ be the countable group generated by $\phi$ and $G.$ Let $S_{0} $
be the intersection $%
%TCIMACRO{\dbigcap }%
%BeginExpansion
{\displaystyle\bigcap}
%EndExpansion
\{$ $\gamma\lbrack O_{1}\cap O_{2}];$ $\gamma\in\Gamma\}.$ Then $S_{0}$ is a
dense $G_{\delta}-$subset of $S$ which is invariant under the action of
$\Gamma.$ From the definition of $\phi,$ it is clear that $\phi$ is strongly
$G-$decomposable over $S_{0}.$ (Recall that we have identified $G$ with
$\bigoplus\mathbb{Z}_{2}.$) Similarly, $\phi^{-1}$ is also strongly
$G-$decomposable over $S_{0}.$ Hence each element of $\Delta$ is strongly
$\ G-$decomposable over $S_{0}.$

Let $H(\Delta)$ be the group of all homeomorphisms $h,$ of $S$ onto $S,$ such
that $h$ is strongly $\Delta-$decomposable with respect to a finite partition
of $S$ into clopen sets. We shall show that $h_{1}$ is in $H(\Delta).$

For $s\in E_{1}=K^{1}(0)$ we have $\phi(s)=h_{1}(s)$ and, for $s\in
F_{1}=K^{1}(1)$, $\phi^{-1}(s)=h_{1}(s).$ We observe that $clE_{1}$ and
$clF_{1}$ are disjoint clopen sets whose union is $S.$ Also,%

\[
h_{1}=\phi\chi_{clE_{1}}+\phi^{-1}\chi_{clF_{1}}%
\]

So $h_{1}\in H(\Delta).$

We now suppose that $h_{1},...h_{n}$ are in $H(\Delta)$. We wish to show
$h_{n+1}\in H(\Delta).$ Let $s\in K^{n+1}(\mathbf{\beta},0)$ where
$\mathbf{\beta\in%
%TCIMACRO{\U{2124} }%
%BeginExpansion
\mathbb{Z}
%EndExpansion
}_{2}^{n}.$\ By Lemma 7.9 (e) $h_{1}^{\beta_{1}+1}...h_{n}^{\beta_{n}+1}(s)\in
K^{n+1}(\mathbf{1},0).$ So, from the definition of $\phi,$%

\[
\phi(h_{1}^{\beta_{1}+1}...h_{n}^{\beta_{n}+1}(s))=h_{1}^{\beta_{1}}%
...h_{1}^{\beta_{n}}h_{n+1}^{1}(s)
\]

Making use of commutativity of the $h_{r},$we get $h_{n+1}(s)==h_{1}%
^{\beta_{1}}...h_{n}^{\beta_{n}}\phi h_{1}^{\beta_{1}+1}...h_{n}^{\beta_{n}%
+1}(s)$ Then, by using continuity, this holds for each $s\in clK^{n+1}%
(\mathbf{\beta},0).$ By a similar argument, for $s\in clK^{n+1}(\mathbf{\beta
},1)$ we get $h_{n+1}(s)=h_{1}^{\beta_{1}+1}...h_{n}^{\beta_{n}+1}\phi
h_{1}^{\beta_{1}}...h_{n}^{\beta_{n}}(s).$

Since $\{clK^{n+1}(\mathbf{\alpha}):\mathbf{\alpha}\in%
%TCIMACRO{\U{2124} }%
%BeginExpansion
\mathbb{Z}
%EndExpansion
_{2}^{n+1}\}$ is a finite collection of disjoint clopen sets whose union is
$S,$it follows that $h_{n+1}\in H(\Delta).$

So, by induction, $G\subset H(\Delta).$

It now follows that, on $S_{0},$the orbit equivalence relation coming from the
action of $G$ coincides with the orbit equivalence relation arising from the $%
%TCIMACRO{\U{2124} }%
%BeginExpansion
\mathbb{Z}
%EndExpansion
-$action generated by $\phi.$
\end{proof}
\end{corollary}

REMARKS In the above, $D$ is not a subset of $S_{0}.$ Indeed, it is not
obvious that $S_{0}$ has a dense orbit. So is it possible to modify the
construction of $\phi$ so that it becomes a bijection of $D$ onto itself ?

\section{Monotone complete $C^{\ast}$-algebra of an equivalence relation}

The idea of constructing a $C^{\ast}$-algebra or a von Neumann algebra from a
groupoid has a long history and a vast literature; there is an excellent
exposition in \cite{zq}. Here, instead of general groupoids, we use an
equivalence relation with countable equivalence classes. Our aim is to
construct monotone complete (monotone $\sigma$-complete) algebras by a
modification of the approach used in \cite{zn}. We try to balance conciseness
with putting in enough detail to convince the reader that this is an easy and
transparent way to construct examples of monotone complete $C^{\ast}%
$-algebras. It makes it possible to obtain all the algebras which arise as a
monotone cross-product by a countable discrete group acting on a commutative
algebra, but without needing to use monotone tensor products.

In this section, $X$ is a topological space where $X$ is either a G-delta
subset of a compact Hausdorff space or a Polish space (i.e. homeomorphic to a
complete separable metric space). Then $X$ is a Baire space i.e. the Baire
category theorem holds for $X.$ Let $B(X)$ be the set of all bounded complex
valued Borel functions on $X.$ When equipped with the obvious algebraic
operations and the supremum norm, it becomes a commutative $C^{\ast}$-algebra.

In the following it would be easy to use a more general setting, where we do
not assume a topology for $X$, replace the field of Borel sets with a
$\sigma-$field $\mathcal{T}$, and use $\mathcal{T-}$measurable bijections
instead of homeomorphisms. But we stick to a topological setting which is what
we need later.

Let $G$ be a countable group of homeomorphisms of $X$ and let%

\[
E=\{(x,y)\in X\times X:\exists g\in G\text{ such that }y=g(x)\}.
\]

Then $E$ is the graph of the orbit equivalence relation on $X$ arising from
the action of $G$. We shall identify this equivalence relation with its graph.
We know, from the work of Section 7, that the same orbit equivalence relation
can arise from actions by different groups.

Let us recall that for $A\subset X$, the \textit{saturation }of $A$ (by $E$) is%

\begin{align*}
E[A]  &  =\{x\in X:\exists\text{ }z\in A\text{ such that }xEz\}\\
&  =\{x\in X:\exists g\in G\text{ such that }g(x)\in A\}\\
&  =%
%TCIMACRO{\tbigcup }%
%BeginExpansion
{\textstyle\bigcup}
%EndExpansion
\{g[A]:g\in G\}.
\end{align*}

It follows from this that the saturation of a Borel set is also a Borel set.

\begin{definition}
Let $\mathcal{I}$ be a $\sigma-$ideal of the Boolean algebra of Borel subsets
of $X$ with $X\notin\mathcal{I}.$
\end{definition}

\begin{definition}
Let $B_{\mathcal{I}}$ be the set of all $f$ in $B(X)$ such that\ $\{x\in
X:f(x)\neq0\}$ is in $\mathcal{I}$ . Then $B_{\mathcal{I}}$ is a $\sigma
-$ideal of $B(X).$ (See Section 3.) Let $q$ be the quotient homomorphism from
$B(X)$ onto $B(X)/B_{\mathcal{I}}.$
\end{definition}

\begin{lemma}
Let $A\in\mathcal{I}.$ Then $E[A]\in\mathcal{I}$ if, and only if,
$g[A]\in\mathcal{I}$ for every $g\in G.$
\end{lemma}

\begin{proof}
For each $g\in G,$ $g[A]\subset E[A].$ Since $\mathcal{I}$ is an ideal, if
$E[A]\in\mathcal{I}$ then $g[A]\in\mathcal{I}.$

Conversely, if $g[A]\in\mathcal{I}$ for each $g,$ then $E[A]$ is the union of
countably many elements of the $\sigma-$ideal and hence in the ideal.
\end{proof}

In the following we require that the action of $G$ maps the ideal
$\mathcal{I}$ into itself. Equivalently, for any $A\in\mathcal{I}$, its
saturation by $E$ is again in $\mathcal{I}$. This is automatically satisfied
if $\mathcal{I} $ is the ideal of meagre Borel sets but we do not wish to
confine ourselves to this situation.

Following the approach of \cite{zn}, we indicate how orbit equivalence
relations on $X$ give rise to monotone complete $C^{\ast}$-algebras. A key
point, used in \cite{zn}, is that these algebras are constructed from the
equivalence relation without explicit mention of $G.$ But in establishing the
properties of these algebras, the existence of an underlying group is used.
This construction (similar to a groupoid $C^{\ast}$-algebra) seems
particularly natural and transparent. For the reader's convenience we give a
brief, explicit account \ which is reasonably self-contained. For reasons
explained in Section 9, the work of this section makes it possible for the
reader to safely avoid the details of the original monotone cross-product
construction. We could work in greater generality (for example we could weaken
the condition that the elements of $G$ be homeomorphisms or consider more
general groupoid constructions) but for ease and simplicity we avoid this. We
are mainly interested in two situations. First, where $X$ is an "exotic" space
as considered in \cite{zk} but $\mathcal{I}$ is only the ideal of meagre
subsets of $X.$ Secondly, where $X$ is just the Cantor space but $\mathcal{I}$
is an "exotic" ideal of the Borel sets. In \ this paper, only the first
situation will be considered but since we will make use of the second
situation in a later work and since no extra effort is required, we add this
small amount of generality.

Since $G$ is a countable group, each orbit is countable, in other words, each
equivalence class associated with the equivalence relation $E$ is countable.
(Countable Borel equivalence relations and their relationship with von Neumann
algebras were penetratingly analysed in \cite{n,p}.)

For each $x\in X$ let $[x]$ be the equivalence class generated by $x$. Let
$[X]$ be the set of all equivalence classes. Let $\ell^{2}([x])$ be the
Hilbert space of all square summable, complex valued functions from $[x]$ to $%
%TCIMACRO{\U{2102} }%
%BeginExpansion
\mathbb{C}
%EndExpansion
$. For each $y\in\lbrack x]$ let $\delta_{y}\in\ell^{2}([x])$ be defined by%

\[
\delta_{y}(z)=0\text{ for }z\neq y;~\delta_{y}(y)=1.
\]

Then $\{\delta_{y}:y\in\lbrack x]\}$ is an orthonormal basis for $\ell
^{2}([x])$ which we shall call the \textit{canonical basis} for $\ell
^{2}([x]).$ For each $x\in X$, $%
%TCIMACRO{\tciLaplace}%
%BeginExpansion
\mathcal{L}%
%EndExpansion
(\ell^{2}([x]))$ is the von Neumann algebra of all bounded operators on
$\ell^{2}([x]).$ We now form a direct sum of these algebras by:%
\[
\mathcal{S}=%
%TCIMACRO{\tbigoplus \limits_{[x]\in\lbrack X]}}%
%BeginExpansion
{\textstyle\bigoplus\limits_{[x]\in\lbrack X]}}
%EndExpansion%
%TCIMACRO{\tciLaplace}%
%BeginExpansion
\mathcal{L}%
%EndExpansion
(\ell^{2}([x])).
\]

This is a Type I von Neumann algebra, being a direct sum of such algebras. It
is of no independent interest but is a framework in which we embed an algebra
of "Borel matrices" and then take a quotient, obtaining monotone complete
$C^{\ast}$-algebras. To each operator $F$ in $\mathcal{S}$ we can associate,
uniquely, a function $f:E\rightarrow%
%TCIMACRO{\U{2102} }%
%BeginExpansion
\mathbb{C}
%EndExpansion
$ as follows. First we decompose $F$ as:%

\[
F=%
%TCIMACRO{\tbigoplus \limits_{[x]\in\lbrack X]}}%
%BeginExpansion
{\textstyle\bigoplus\limits_{[x]\in\lbrack X]}}
%EndExpansion
F_{[x]}.
\]

Here each $F_{[x]}$ is a bounded operator on $\ell^{2}([x]).$ Now recall that
$(x,y)\in E$ precisely when $y\in\lbrack x].$ We now define $f:E\rightarrow%
%TCIMACRO{\U{2102} }%
%BeginExpansion
\mathbb{C}
%EndExpansion
$ by%

\[
f(x,y)=\,<F_{[x]}\delta_{x},\delta_{y}>.
\]

When $f$ is restricted to $[x]\times\lbrack x]$ then it becomes the matrix
representation of $F_{[x]}$ with respect to the canonical orthonormal basis of
$\ell^{2}([x]).$ It follows that there is a bijection between operators in
$\mathcal{S}$ and those functions $f:E\rightarrow%
%TCIMACRO{\U{2102} }%
%BeginExpansion
\mathbb{C}
%EndExpansion
$ for which there is a constant $k$ such that, for each $[x]\in\lbrack X],$
the restriction of $f$ to $[x]\times\lbrack x]$ is the matrix of a bounded
operator on $\ell^{2}([x])$ whose norm is bounded by $k.$ Call such an $f$
\textit{matrix bounded. }For each matrix bounded $f$ let $L(f)$ be the
corresponding element of $\mathcal{S}.$

When $f$ and $h$ are such functions from $E$ \ to $%
%TCIMACRO{\U{2102} }%
%BeginExpansion
\mathbb{C}
%EndExpansion
$, then straightforward matrix manipulations give%
\[
L(f)L(h)=L(f\circ h)
\]

where $f\circ h(x,z)=\sum_{y\in\lbrack x]}f(x,y)h(y,z).$ Also $L(f)^{\ast
}=L(f^{\ast}),$ where $f^{\ast}(x,y)=\overline{f(y,x)}$ for all $(x,y)\in E.$

Let $||f||=||L(f)||$. Then the matrix bounded functions on $E$ form a
$C^{\ast}$-algebra isomorphic to $\mathcal{S}.$

Let $\Delta$ be the diagonal set $\{(x,x):x\in X\}.$ It is closed, because the
topology of $X$ is Hausdorff. It is an easy calculation to show that
$L(\chi_{\Delta})$ is the unit element of $\mathcal{S}.$ For each $g\in G,$
the map $(x,y)\rightarrow(x,g(y))$ is a homeomorphism.

So $\{(x,g(x)):x\in X\}$ is a closed set.

Let us recall that%

\[
E=%
%TCIMACRO{\tbigcup \limits_{g\in G}}%
%BeginExpansion
{\textstyle\bigcup\limits_{g\in G}}
%EndExpansion
\{(x,g(x)):x\in X\}.
\]

Since $G$ is countable, $E$ is the union of countably many closed sets. Hence
$E$ is a Borel subset of $X\times X.$

\begin{definition}
Let $\mathcal{M}(E)$ be the set of all Borel measurable functions
$f:E\rightarrow%
%TCIMACRO{\U{2102} }%
%BeginExpansion
\mathbb{C}
%EndExpansion
$ which are matrix bounded.
\end{definition}

\begin{lemma}
The set $\{L(f):f\in\mathcal{M}(E)\}$ is a $C^{\ast}$-subalgebra of
$\mathcal{S}$ which is sequentially closed with respect to the weak operator
topology of $\mathcal{S}$. We denote this algebra by $L(\mathcal{M}(E)).$ When
equipped with the appropriate algebraic operations and norm, $\mathcal{M}(E)$
is a $C^{\ast}$-algebra isomorphic to $L[\mathcal{M}(E)]$.
\end{lemma}

\begin{proof}
See Lemma 2.1 \cite{zn}.
\end{proof}

\begin{lemma}
Let $(f_{n})$ be a sequence in the unit ball of $\mathcal{M}(E)$ which
converges pointwise to $f.$ Then $f\in\mathcal{M}(E)$ and $L(f_{n})$ converges
to $L(f)$ in the weak operator topology of $\mathcal{S}.$ Also $f$ is in the
unit ball of $\mathcal{M}(E).$
\end{lemma}

\begin{proof}
The weak operator topology gives a compact Hausdorff topology on the norm
closed unit ball of $\mathcal{S}$.

Let $K_{n}=\{L(f_{j}):j\geq n\}$ and let $clK_{n}$ be its closure in the weak
operator topology of $\mathcal{S}$. Then, by the finite intersection property,
there exists $T\in%
%TCIMACRO{\tbigcap \limits_{n\in\mathbb{N}}}%
%BeginExpansion
{\textstyle\bigcap\limits_{n\in\mathbb{N}}}
%EndExpansion
clK_{n}.$ Let $U$ be an open neighbourhood of $T$, then $U\cap K_{n}$ is
non-empty for all $n.$

Fix $(x,y)$ in $E.$ Fix $\varepsilon>0.$ Let $U=\{S\in\mathcal{S}%
:|<(S-T)\delta_{x},\delta_{y}>|<\varepsilon\}.$ Then we can find a subsequence
$(f_{n(r)})(r=1,2...)$ for which $L(f_{n(r)})\in U$ for each $r.$ Thus
$|(f_{n(r)}(x,y)-\,<T\delta_{x},\delta_{y}>|<\varepsilon.$ So
$|f(x,y)-\,<T\delta_{x},\delta_{y}>|\leq\varepsilon.$ Since this holds for all
positive $\varepsilon,$ we have $f(x,y)=<T\delta_{x},\delta_{y}>.$ So
$T=L(f).$

Let $\mathcal{T}$ be the locally convex topology of $\mathcal{S}$ generated by
all seminorms of the form $V\rightarrow|<V\delta_{x},\delta_{y}>|$ as $(x,y)$
ranges over $E.$ This is a Hausdorff topology which is weaker than the weak
operator topology. Hence it coincides with the weak operator topology on the
unit ball, because the latter topology is compact. But $<L(f_{n})\delta
_{x},\delta_{y}>\rightarrow f(x,y)=<L(f)\delta_{x},\delta_{y}>$ for all
$(x,y)\in E.$

It now follows that $L(f_{n})\rightarrow L(f)$ in the weak operator topology
of $\mathcal{S}.$ Since $f$ is the pointwise limit of a sequence of Borel
measurable functions, it too, is Borel measurable. So $f\in\mathcal{M}(E)$
and, since $T$ is in the unit ball of $\mathcal{S},$ $f$ is in the unit ball
of $\mathcal{M}(E).$
\end{proof}

Let $p$ be the homeomorphism of $\Delta$ onto $X$, given by $p(x,x)=x.$ So
$B(X)$, the algebra of bounded Borel measurable functions on $X,$ is
isometrically $\ast-$ isomorphic to $B(\Delta)$ under the map $h\rightarrow
h\circ p.$

For each $f\in\mathcal{M}(E)$ let $Df$ be the function on $E$ which vanishes
off the diagonal, $\Delta,$ and is such that, for each $x\in X,$%

\[
Df(x,x)=f(x,x).
\]

Then $D$ is a linear idempotent map from $\mathcal{M}(E)$ onto an abelian
subalgebra which we can identify with $B(\Delta),$ which can, in turn, be
identified with $B(X).$ Let $\widetilde{D}f$ be the function on $X$ such that
$\widetilde{D}f(x)=Df(x,x)$ for all $x\in X.$ We shall sometimes abuse our
notation by using $Df$ instead of $\widetilde{D}f$ .

Let $\pi:B(X)\rightarrow\mathcal{M}(E)$ be defined by $\pi(h)(x,x)=h(x)$ for
$x\in X$ and $\pi(h)(x,y)=0$ for $x\neq y.$ Then $\pi$ is a $\ast-$isomorphism
of $B(X)$ onto an abelian $\ast-$subalgebra of $\mathcal{M}(E),$ which can be
identified with the range of $D.$

We have $\pi\widetilde{D}f=Df$ for $f\in\mathcal{M}(E).$ Also, for $g\in
B(X),$ $\widetilde{D}\pi(g)=g.$

\begin{lemma}
$D$ is a positive map.
\end{lemma}

\begin{proof}
Each positive element of $\mathcal{M}(E)$ is of the form $f\circ f^{\ast}.$But%
\[
(\#)\text{ \ \ }D(f\circ f^{\ast})(x,x)=\sum_{y\in\lbrack x]}f(x,y)f^{\ast
}(y,x)=\sum_{y\in\lbrack x]}f(x,y)\overline{f(x,y)}=\sum_{y\in\lbrack
x]}|f(x,y)|^{2}\geq0.
\]

\end{proof}

\begin{definition}%
\begin{align*}
\text{ Let }I_{\mathcal{I}}  &  =\{f\in\mathcal{M}(E):q\widetilde{D}(f\circ
f^{\ast})=0\}\\
&  =\{f\in\mathcal{M}(E):\exists A\in\mathcal{I}\text{ such that
}\widetilde{D}(f\circ f^{\ast})(x)=0\text{ for }x\notin A\}.
\end{align*}

\end{definition}

\begin{lemma}
$I_{\mathcal{I}}$ is a \ two-sided ideal of $\mathcal{M}(E).$
\end{lemma}

\begin{proof}
In any $C^{\ast}$-algebra,%
\[
(a+b)(a+b)^{\ast}\leq2(aa^{\ast}+bb^{\ast}).
\]

So%
\[
D((f+g)\circ(f+g)^{\ast})\leq2D(f\circ f^{\ast})+2D(g\circ g^{\ast}).
\]
From this it follows that if $f$ and $g$ are both in $I_{\mathcal{I}}$ then so
also is $f+g.$

In any $C^{\ast}$-algebra,%
\[
fz(fz)^{\ast}=fzz^{\ast}f^{\ast}\leq||z||^{2}ff^{\ast}.
\]
From this it follows that $f\in I_{\mathcal{I}}$ and $z\in\mathcal{M}(E)$
implies that $f\circ z\in I_{\mathcal{I}}$.

Now suppose that $f\in I_{\mathcal{I}}$. Then \ for some $A\in\mathcal{I},$
$E(f\circ f^{\ast})(x)=0$ for $x\notin A$. Since $E[A]\in\mathcal{I}$ we can
suppose that $A=E[A].$ Hence if $x\notin E[A]$ then $[x]\cap E[A]=\varnothing
.$ For $x\notin E[A],$ we have%
\[
0=D(f\circ f^{\ast})(x)=\sum_{y\in\lbrack x]}|f(x,y)|^{2}.
\]
Thus $f(x,y)=0$ for $xEy$ and $x\notin A.$ Then, for $z\notin A,$ we have
$D(f^{\ast}\circ f)(z)=\sum_{y\in\lbrack z]}|f^{\ast}(z,y)|^{2}=\sum
_{y\in\lbrack z]}|f(y,z)|^{2}=0.$ So $f^{\ast}\in I_{\mathcal{I}}$. So
$I_{\mathcal{I}}$ is a two-sided ideal of $\mathcal{M}(E).$
\end{proof}

\begin{lemma}
If $y\in I_{\mathcal{I}}$ then $q\widetilde{D}(y)=0.$ Furthermore $y\in
I_{\mathcal{I}}$ if, and only if $q\widetilde{D}(y\circ a)=0$ for all
$a\in\mathcal{M}(E).$
\end{lemma}

\begin{proof}
Let $T$ be the (compact Hausdorff) structure space of the algebra
$B(X)/B_{\mathcal{I}}$. By applying the Cauchy-Schwartz inequality we see that
for $x,y$ in $\mathcal{M}(E)$ and $t\in T$,%
\[
|q\widetilde{D}(x^{\ast}\circ y)(t)|\leq q\widetilde{D}(x^{\ast}\circ
x)(t)^{1/2}q\widetilde{D}(y^{\ast}\circ y)(t)^{1/2}.
\]
Let $x=1$ and let $y\in I_{\mathcal{I}}.$ Then $y^{\ast}\in I_{\mathcal{I}}.$
So $q\widetilde{D}(y^{\ast}\circ y)=0.$ From the above inequality it follows
that $q\widetilde{D}(y)=0.$ Since $I_{\mathcal{I}}$ is an ideal, if $y\in
I_{\mathcal{I}}$ then $y\circ a$ is in the ideal for each $a\in M(E)$. It now
follows from the above that $q\widetilde{D}(y\circ a)=0.$ Conversely, if
$q\widetilde{D}(y\circ a)=0$ for all $a\in\mathcal{M}(E)$ then, on putting
$a=y^{\ast}$ we see that $y\in I_{\mathcal{I}}.$
\end{proof}

\begin{lemma}
$L[I_{\mathcal{I}}]$ is a (two-sided) ideal of $L[\mathcal{M}(E)]$ which is
sequentially closed in the weak operator topology of $\mathcal{S}.$

\begin{proof}
Let $(f_{r})$ be a sequence in $I_{\mathcal{I}}$ such that $(L(f_{r}))$ is a
sequence which converges in the weak operator topology to an element $T$ of
$\mathcal{S}.$ Then it follows from the Uniform Boundedness Theorem that the
sequence is bounded in norm. By Lemma 8.5, there exists $f\in\mathcal{M}(E)$
such that $L(f)=T$ where $L(f_{r})\rightarrow L(f)$ in the weak operator
topology. So $f_{r}\rightarrow f$ pointwise. Hence $\widetilde{D}%
(f_{r})\rightarrow\widetilde{D}(f)$ pointwise. For each $r$ there exists
$A_{r}\in\mathcal{I}$ such that $x\notin A_{r}$ implies $\widetilde{D}%
(f_{r})(x)=0.$ Since $\mathcal{I}$ is a Boolean $\sigma-$ideal of the Boolean
algebra of Borel subsets of $X,$ $\cup\{A_{r}:r=1,2...\}$ is in $\mathcal{I}.$
Hence $q\widetilde{D}(f)=0.$
\end{proof}
\end{lemma}

\begin{proof}
For any $a\in\mathcal{M}(E),(f_{r}\circ a)$ is a sequence in $I_{\mathcal{I}}$
such that $(L(f_{r}\circ a))$ converges in the weak operator topology to
$L(f)L(a).$ So, as in the preceding paragraph, $q\widetilde{D}(f\circ a)=0.$
By appealing to Lemma 8.10 we see that $f\in I_{\mathcal{I}}.$ Hence
$L[I_{\mathcal{I}}]$ is sequentially closed in the weak operator topology of
$\mathcal{S}.$
\end{proof}

\begin{corollary}
$\mathcal{M}(E)$ is monotone $\sigma-$complete and $I_{\mathcal{I}}$ is a
$\sigma-$ideal
\end{corollary}

\begin{proof}
Each norm bounded monotone increasing sequence in $L(\mathcal{M}(E))$
converges in the strong operator topology to an element $T$ of $\mathcal{S}$.
By Lemma 8.5, $T\in L(\mathcal{M}(E)).$ Then $T=L(f)$ for some $f\in
\mathcal{M}(E).$ Hence $L(\mathcal{M}(E))$ (and its isomorphic image,
$\mathcal{M}(E)$) are monotone $\sigma-$complete. It now follows from Lemma
8.11, that $I_{\mathcal{I}}$ is a $\sigma-$ideal.
\end{proof}

\begin{definition}
Let $Q$ be the quotient map from $\mathcal{M}(E)$ onto $\mathcal{M}%
(E)/I_{\mathcal{I}}.$
\end{definition}

\begin{proposition}
The algebra $\mathcal{M}(E)/I_{\mathcal{I}}$ is monotone $\sigma-$complete.
There exists a positive, faithful, $\sigma-$normal, conditional expectation
$\widehat{D}$ from $\mathcal{M}(E)/I_{\mathcal{I}}$ onto a commutative
$\sigma-$subalgebra, which is isomorphic to $B(X)/B_{\mathcal{I}}.$
Furthermore, if there exists a strictly positive linear functional on
$B(X)/B_{\mathcal{I}},$ then $\mathcal{M}(E)/I_{\mathcal{I}}$ is monotone
complete and $\widehat{D}$ is normal.
\end{proposition}

\begin{proof}
By Corollary 8.12 and the results of Section 3, the quotient algebra
$\mathcal{M}(E)/I_{\mathcal{I}}$ is monotone $\sigma-$complete.

Let $g\in B(X).$ Then, as remarked before Lemma 8.7, $\widetilde{D}\pi(g)=g.$

Now $\pi(g)\in$ $I_{\mathcal{I}}$ if, and only if, \ $q\widetilde{D}%
(\pi(g)\circ\pi(g)^{\ast})=0.$

But $q\widetilde{D}(\pi(g)\circ\pi(g)^{\ast})=q\widetilde{D}(\pi
(|g|^{2})=q(|g|^{2}).$

So $\pi(g)\in$ $I_{\mathcal{I}}$ if and only if $|g|^{2}\in B_{\mathcal{I}}$
i.e. if and only if $g$ vanishes off some set $A\in\mathcal{I}$ i.e. if and
only if $g\in B_{\mathcal{I}}.$

So $\pi$ induces an isomorphism from $B(X)/B_{\mathcal{I}}$ onto
$D[M(E)]/I_{\mathcal{I}}.$

Let $h\in I_{\mathcal{I}}.$ Then, by Lemma 8.10, $q\widetilde{D}(h)=0.$ That
is, $\widetilde{D}(h)\in B_{\mathcal{I}}.$ So $\pi\widetilde{D}(h)\in
I_{\mathcal{I}}.$ But $\pi\widetilde{D}(h)=Dh.$

So $h\in I_{\mathcal{I}}$ implies $QDh=0.$ It now follows that we can define
$\widehat{D}$ on $\mathcal{M}(E)/I_{\mathcal{I}}$ by $\widehat{D}$
$(f+I_{\mathcal{I}})=QDf.$

It is clear that $\widehat{D}$ is a positive linear map which is faithful. Its
range is an abelian subalgebra of $\mathcal{M}(E)/I_{\mathcal{I}}$. This
subalgebra is $D[\mathcal{M}(E)]/I_{\mathcal{I}}$ which, as we have seen
above, is isomorphic to $B(X)/B_{\mathcal{I}};$ we shall denote $D[\mathcal{M}%
(E)]/I_{\mathcal{I}}$ by $A$ and call it the diagonal algebra. Furthermore,
$\widehat{D}$ is idempotent, so it is a conditional expectation.

Let $(f_{n})$ be a sequence in $\mathcal{M}(E)$ such that $(Qf_{n})$ is an
upper bounded monotone increasing sequence in $\mathcal{M}(E)/I_{\mathcal{I}%
}.$ Then, by Lemma 3.3, we may assume that $(f_{n})$ is an upper bounded,
monotone increasing sequence in $\mathcal{M}(E).$ Let $Lf$ be the limit of
$(Lf_{n})$ in the weak operator topology. (Since the sequence is monotone, $Lf
$ is also its limit in the strong operator topology.) By Lemma 8.5,
$f\in\mathcal{M}(E),$ and $f_{n}(x,x)\rightarrow f(x,x)$ for all $x\in X.$
Thus $Df_{n}\rightarrow Df$ pointwise on $X.$ Also, since $D$ is positive,
$(Df_{n})$ is monotone increasing. Since $Q$ is a $\sigma-$homomorphism, $QDf
$ is the least upper bound of $(QDf_{n})\mathbf{.}$ Since $\widehat{D}$
$(f+I_{\mathcal{I}})=QDf$ it now follows that $\widehat{D}$ is $\sigma-$normal.

If $\mu$ is a strictly positive functional on $B(X)/B_{\mathcal{I}}$ then
$\mu\widehat{D}$ is a strictly positive linear functional on
$M(E)/I_{\mathcal{I}}$. It then follows from Lemma 3.1 that $\mathcal{M}%
(E)/I_{\mathcal{I}}$ is monotone complete. Furthermore, if $\Lambda$ is a
downward directed subset of the self-adjoint part of $\mathcal{M}%
(E)/I_{\mathcal{I}},$ with $0$ as its greatest lower bound, then there exists
a monotone decreasing sequence $(x_{n}),$ with each $x_{n}$ in $\Lambda,$ and
$%
%TCIMACRO{\tbigwedge }%
%BeginExpansion
{\textstyle\bigwedge}
%EndExpansion
x_{n}=0.$ It now follows from the $\sigma-$normality of $\widehat{D}$ that $0$
is the infimum of $\{\widehat{D}(x):x\in\Lambda\}.$ Hence $\widehat{D}$ is normal.
\end{proof}

We now make additional assumptions about the action of $G$ and use this to
construct a natural unitary representation of $G$. We give some technical
results which give an analogue of Mercer-Bures convergence, see \cite{zc} and
\cite{d}. This will be useful in Section 12 when we wish to approximate
elements of $\mathcal{M}(E)/I_{\mathcal{I}}$ by finite dimensional subalgebras.

For the rest of this section we suppose that the action of $G$ on $X$ is free
on each orbit i.e. for each $x\in X,$ $x$ is not a fixed point of $g$, where
$g\in G,$ unless $g$ is the identity element of $G.$

For each $g\in G,$ let $\Delta_{g}=\{(x,gx):x\in X\}$. Then the $\Delta_{g}$
are pairwise disjoint and $E=\cup_{g\in G}\Delta_{g}.$

For each $g\in G,$ let $u_{g}:E\rightarrow\{0,1\}$ be the characteristic
function of $\Delta_{g^{-1}}.$ As we pointed out earlier, $\chi_{\Delta}$is
the unit element of $\mathcal{M}(E),$so, in this notation, $u_{1}$ is the unit
element of $\mathcal{M}(E).$

For each $(x,y)\in E$ we have:%

\[
u_{g}\circ u_{h}(x,y)=\sum_{k\in G}u_{g}(x,kx)u_{h}(kx,y).
\]

But $u_{g}(x,kx)\neq0,$ only if $k=g^{-1}$ and $u_{h}(g^{-1}x,y)\neq0$ only if
$y=h^{-1}g^{-1}x=(gh)^{-1}x.$ So $u_{g}\circ u_{h}=u_{gh}.$

Also $u_{g}^{\ast}(x,y)=\overline{u_{g}(y,x)}=u_{g}(y,x).$ But $u_{g}%
(y,x)\neq0$ only if $x=g^{-1}y$, that is, only if $y=gx.$ So $u_{g}^{\ast
}(x,y)=u_{g^{-1}}(x,y).$ It follows that $g\rightarrow u_{g}$ is a unitary
representation of $G$ in $\mathcal{M}(E).$

Let $f$ be any element of $\mathcal{M}(E).$ Then%

\[
(i)\text{ \ }f\circ u_{g}(x,y)=\sum_{z\in\lbrack x]}f(x,z)u_{g}(z,y)=f(x,gy).
\]

So, for each $x\in X,$ $D(f\circ u_{g})(x,x)=f(x,gx). $ Then $D(f\circ
u_{g})\circ u_{g^{-1}}(x,y)=\sum_{z\in\lbrack x]}D(f\circ u_{g})(x,z)u_{g^{-1}%
}(z,y)=D(f\circ u_{g})(x,x)u_{g^{-1}}(x,y)=f(x,gx)\chi_{\Delta_{g}}(x,y).$

So%

\[
(ii)\text{ \ \ \ }D(f\circ u_{g})\circ u_{g^{-1}}(x,y)=\left\{
\begin{array}
[c]{cc}%
f(x,y) & \text{if }(x,y)\in\Delta_{g}\\
0 & \text{if }(x,y)\notin\Delta_{g}.
\end{array}
\right.
\]

The identity $(\#)$, used in Lemma 8.7,can be re-written as

\bigskip$(iii)$ \ $D(f\circ f^{\ast})(x,x)=\sum_{g\in G}|f(x,gx)|^{2}%
=\sum_{g\in G}|D(f\circ u_{g})(x,x)|^{2}=\sum_{g\in G}|\widetilde{D}(f\circ
u_{g})(x)|^{2}.$

Let $F$ be any finite subset of $G.$ Let $f_{F}=\sum_{g\in F}D(f\circ
u_{g})\circ u_{g^{-1}}.$ Then, using $(iii)$,%

\[
\text{ \ \ \ }(f-f_{F})(x,y)=\left\{
\begin{array}
[c]{cc}%
0 & \text{if }(x,y)\in\Delta_{g}\text{ and }g\in F\\
f(x,y) & \text{if }(x,y)\in\Delta_{g}\text{ and }g\notin F.
\end{array}
\right.
\]

We now replace $f$ by $f-f_{F}$ in $(iii)$ and get:%

\[
(iv)\text{ \ \ }D((f-f_{F})\circ(f-f_{F})^{\ast})(x,x)=\sum_{g\in G\backslash
F}|D(f\circ u_{g})(x,x)|^{2}=\sum_{g\in G\backslash F}|\widetilde{D}(f\circ
u_{g})(x)|^{2}.
\]

Now let $(F_{n})(n=1,2...)$ be any strictly increasing sequence of finite
subsets of $G.$ Write $f_{n}$ for $f_{F_{n}}.$ Then

$D((f-f_{n})\circ(f-f_{n})^{\ast})(x,x)$ decreases monotonically to $0$ as
$n\rightarrow\infty.$ Since $Q$ is a $\sigma-$homomorphism,
\[
(v)\text{ \ \ }%
%TCIMACRO{\tbigwedge }%
%BeginExpansion
{\textstyle\bigwedge}
%EndExpansion
QD((f-f_{n})\circ(f-f_{n})^{\ast})=0.
\]

For each $g\in G,$let $U_{g}=Qu_{g}.$ Since $Q$ is a $\ast-$homomorphism onto
$\mathcal{M}(E)/I_{\mathcal{I}}$, $U_{g}$ is a unitary and $g\rightarrow
U_{g}$ is a unitary representation of $G$ in $\mathcal{M}(E)/I_{\mathcal{I}}.$
On applying the preceding paragraph we get:

\begin{proposition}
Let $z\in\mathcal{M}(E)/I_{\mathcal{I}}.$ Let $(F(n))(n=1,2...)$ be a strictly
increasing sequence of finite subsets of $G.$ Let $z_{n}=\sum_{g\in
F(n)}\widehat{D}(zU_{g})U_{g^{-1}}.$ Then
\end{proposition}

\[%
%TCIMACRO{\tbigwedge }%
%BeginExpansion
{\textstyle\bigwedge}
%EndExpansion
\widehat{D}((z-z_{n})(z-z_{n})^{\ast})=0.
\]

\begin{corollary}
Let $z\in\mathcal{M}(E)/I_{\mathcal{I}}$ such that $\widehat{D}(zU_{g})=0$ for
each $g.$ Then $z=0.$
\end{corollary}

\begin{proof}
This follows from Proposition 8.15 because $z_{n}=0$ for every $n.$
\end{proof}

\begin{lemma}
For each $f\in\mathcal{M}(E)$ and each $g\in G,$ $u_{g}^{\ast}\circ f\circ
u_{g}(x,y)=f(gx,gy).$ In particular,$f$ vanishes off $\Delta$ if, and only if,
$u_{g}^{\ast}\circ f\circ u_{g}$ vanishes off $\Delta.$
\end{lemma}

\begin{proof}
Let $h\in\mathcal{M}(E).$ Then, applying identity (i) we get $(u_{g}^{\ast
}\circ h)(a,b)=(h^{\ast}\circ u_{g})^{\ast}(a,b)=\overline{(h^{\ast}\circ
u_{g})(b,a)}=\overline{h^{\ast}(b,ga)}=h(ga,b).$

Let $h=f\circ u_{g}.$ Then $u_{g}^{\ast}\circ f\circ u_{g}(x,y)=f\circ
u_{g}(gx,y)=f(gx,gy).$
\end{proof}

\begin{corollary}
For each $a\in A\,,$the diagonal algebra, and for each $g\in G,U_{g}%
aU_{g}^{\ast}$ is in $A.$
\end{corollary}

\begin{lemma}
Let $T$ be a compact, totally disconnected space. Let $\theta$ be a
homeomorphism of $T$ onto $T.$ Let $\lambda_{\theta}$ be the automorphism of
$C(T)$ induced by $\theta.$ Let $E$ be a non-empty clopen set such that, for
each clopen $Q\subset T,(\lambda_{\theta}(\chi_{Q})-\chi_{Q})\chi_{E}=0.$ Then
$\theta(t)=t$ for each $t\in E.$ In other words, $\lambda_{\theta}(f)=f$ for
each $f\in\chi_{E}C(T).$
\end{lemma}

\begin{proof}
Let us assume that $t_{0}\in E$ such that $\theta(t_{0})\neq t_{0}$. By total
disconnectedness,there exists a clopen set $Q$ with $\theta(t_{0})\in Q$ and
$t_{0}\notin Q.$ We have $\lambda_{\theta}(\chi_{Q})=\chi_{\theta^{-1}[Q]}.$
So%
\[
\theta^{-1}[Q]\cap E=Q\cap E.
\]
But $t_{0}$ is an element of $\theta^{-1}[Q]\cap E$ and $t_{0}\notin Q.$ This
is a contradiction.
\end{proof}

\bigskip

Let $M$ be a monotone ($\sigma-$)complete $C\ast-$algebra. Then we recall that
an automorphism $\alpha$ is \textit{properly outer} if there does not exist a
non-zero projection $e$ such that $\alpha$ restricts to the identity on $eMe.$

\begin{proposition}
Whenever $g\in G$ and $g$ is not the identity, let $a\rightarrow U_{g}%
aU_{g}^{\ast}$ be a properly outer automorphism of $\widehat{D}[\mathcal{M}%
(E)/I_{\mathcal{I}}].$ Then $\widehat{D}[\mathcal{M}(E)/I_{\mathcal{I}}]$ is a
maximal abelian $\ast-$subalgebra of $\mathcal{M}(E)/I_{\mathcal{I}}$.
\end{proposition}

\begin{proof}
Let $z$ commute with each element of $\widehat{D}[\mathcal{M}%
(E)/I_{\mathcal{I}}].$ Let $a=\widehat{D}(z).$ We shall show that $z=a.$ By
Corollary 8.16, it will suffice to show that $\widehat{D(}(z-a)U_{g})=0$ for
each $g\in G$. We remark that $\widehat{D(}aU_{g})=a\widehat{D(}U_{g})=0$ when
$g$ is not the identity element of $G.$

So it is enough to show that if $g\in G$ and $g$ is not the identity element
of $G$ then $\widehat{D(}zU_{g})=0.$

We have, for each $b\in$ $\widehat{D}[\mathcal{M}(E)/I_{\mathcal{I}}],$
$bz=zb.$ So%
\[
b\widehat{D}(zU_{g})=\widehat{D}(bzU_{g})=\widehat{D}(zbU_{g})=\widehat{D}%
(zU_{g}U_{g}^{\ast}bU_{g})=\widehat{D}(zU_{g})U_{g}^{\ast}bU_{g}=U_{g}^{\ast
}bU_{g}\widehat{D}(zU_{g}).
\]
This implies that $(\lambda_{g^{-1}}(b)-b)\widehat{D}(zU_{g})=0.$ For
shortness put $c=\widehat{D}(zU_{g}).$

Assume that $c\neq0.$ Then, by spectral theory, there exists a non-zero
projection $e$ and a strictly positive real number $\delta$ such that $\delta
e\leq cc^{\ast}.$ Then%
\[
0\leq\delta(\lambda_{g^{-1}}(b)-b)e(\lambda_{g^{-1}}(b)-b)^{\ast}\leq
(\lambda_{g^{-1}}(b)-b)cc^{\ast}(\lambda_{g^{-1}}(b)-b)^{\ast}=0.
\]
So, $(\lambda_{g^{-1}}(b)-b)e=0$ for each $b$ in the range of $\widehat{D}.$
It now follows from Lemma 8.19 that $\lambda_{g^{-1}}(ea)=ea$ for each $a$ in
the range of $\widehat{D}.$ But this contradicts the freeness of the action of
$G$. So $\widehat{D}(zU_{g})=0.$ It now follows that $z$ is in $\widehat{D}%
[\mathcal{M}(E)/I_{\mathcal{I}}].$Hence $\widehat{D}[\mathcal{M}%
(E)/I_{\mathcal{I}}]$ is a maximal abelian $\ast-$subalgebra of $\mathcal{M}%
(E)/I_{\mathcal{I}}.$
\end{proof}

\begin{corollary}
When the action of $G$ is free on $X$ and $\mathcal{I}$ is the Boolean ideal
of meagre Borel subsets of $X$ then $\widehat{D}[\mathcal{M}(E)/I_{\mathcal{I}%
}]$ is a maximal abelian $\ast-$subalgebra of $\mathcal{M}(E)/I_{\mathcal{I}}%
$.\bigskip
\end{corollary}

A unitary $w\in\mathcal{M}(E)/I_{\mathcal{I}}$ is said to normalise a $\ast
-$subalgebra $A$ if $\ wAw^{\ast}=A.\ $For future reference we define the
\textit{normaliser subalgebra }of $\mathcal{M}(E)/I_{\mathcal{I}}$ to be the
smallest monotone complete $\ast-$subalgebra of $\mathcal{M}(E)/I_{\mathcal{I}%
}$ which contains every unitary which normalizes the diagonal subalgebra. It
follows from Corollary 8.18 that each $U_{g}$ is a normalising unitary for the
diagonal algebra, $A=$ $\widehat{D}[\mathcal{M}(E)/I_{\mathcal{I}}];$ since
each element of $A$ is a finite linear combination of unitaries in $A,$ it
follows immediately that $A$ is contained in the normaliser subalgebra.

\section{Cross-product algebras}

First let us recall some familiar facts. Let $A$ be a unital $C^{\ast}%
$-algebra. Let $\alpha$ be an automorphism of $A.$ If there exists a unitary
$u\in A$ such that, for each $z\in A,$ $\alpha(z)=uzu^{\ast}$ then $\alpha$ is
said to be an \textit{inner automorphism. }When no such unitary exists in $A$
then $\alpha$ is an \textit{outer} automorphism.

Let $G$ be a countable group and let $g\rightarrow\beta_{g}$ be an
homomorphism of $G$ into the group of all automorphisms of $A.$ Intuitively, a
cross-product algebra, for this action of $G,$ is a larger $C^{\ast}$-algebra,
$B,$ in which $A$ is embedded as a subalgebra and where each $\beta_{g}$ is
induced by a unitary in $B.$ More precisely, there is an injective
*-homomorphism $j:A\rightarrow B,$ and a group homomorphism $g\rightarrow
U_{g},$ (from $G$ into the group of unitaries in $B)$, such that, for each
$z\in A,$ $j\beta_{g}(z)=U_{g}j(z)U_{g}^{\ast}.$ So when we identify $A$ with
its image $j[A]$ in $B,$ although $\beta_{g}$ need not be an inner
automorphism of $A$ it can be extended to an inner automorphism of the larger
algebra $B.$ We also require that $B$ is "in some sense" generated by $j[A]$
and the collection of unitaries $U_{g}$. When $A$ is a monotone complete
$C^{\ast}$-algebra, we can always construct a $B$ which is monotone complete.

An account of monotone cross-products when $A$ is commutative was given by
Takenouchi \cite{zp}, see also Sait\^{o} \cite{ze}. (This was for abelian
groups, but everything extends without difficulty to non-abelian groups.) This
was later generalised by Hamana \cite{x} to the situation where $A$ is not
commutative. For the purposes of this paper we only need to consider the
situation when $A$ is commutative. So for the rest of this section, $A$ shall
be a monotone complete commutative $C^{\ast}$-algebra.\ Hence $A\simeq C(S)$
where $S$ is a compact, Hausdorff, extremally disconnected space. We shall
outline below\ some of the properties of the monotone cross-product over
abelian algebras. More information can be found in \cite{zp} and \cite{ze}. It
turns out that they can be identified with algebras already constructed in
Section 8. But, historically, the construction of monotone cross-products came first.

We shall suppose for the rest of this section that $S$ has no isolated points.
Then, as remarked in Section 4, any dense subset $Y$ has no isolated points.

We shall use a result from \cite{zn} to relate monotone cross-products to the
monotone complete $C^{\ast}$-algebra of an orbit equivalence relation.

Let $G$ be a countably infinite group of homeomorphisms of $S$, where the
action of $G$ has a free, dense orbit, then we shall show that the
corresponding monotone cross-product is isomorphic to one obtained by an
action of $\bigoplus\mathbb{Z}_{2}.$

Let $X$ be any dense subset of $S$. Then, see Section 4, $S$ can be identified
with the Stone-Czech compactification of $X$. So if $f:X\rightarrow%
%TCIMACRO{\U{2102} }%
%BeginExpansion
\mathbb{C}
%EndExpansion
$ is a bounded continuous function then it has a unique extension to a
continuous function $\widehat{f}:S\rightarrow%
%TCIMACRO{\U{2102} }%
%BeginExpansion
\mathbb{C}
%EndExpansion
.$ It follows that $f\rightarrow\widehat{f}$ is an isometric $\ast
-$isomorphism of $C_{b}(X),$ the algebra of bounded continuous functions on
$X,$ onto $C(S).$ Similarly, as remarked earlier, any homeomorphism $\theta$
from $X$ onto $X$ has a unique extension to a homeomorphism $\widehat{\theta}$
from $S$ onto $S.$ We may abuse our notation by using $\theta$ instead of
$\widehat{\theta}$ i.e. using the same symbol for a homeomorphism of $X$ and
for its unique extension to a homeomorphism of $S$. Slightly more generally,
when $X_{1}$ and $X_{2}$ are dense subsets of $S$, if there exists an
homeomorphism of $X_{1}$ onto $X_{2}$ then it has a unique extension to a
homeomorphism of $S$ onto itself.

Let $g$ be any homeomorphism of $S$ onto itself. Let $\alpha_{g}(f)=f\circ g$
for each $f\in C(S).$ Then $\alpha_{g}$ is a $\ast-$automorphism of $C(S).$
All $\ast-$automorphisms of $C(S)$ arise in this way. Let $G$ be a subgroup of
$Homeo(S),$ the group of homeomorphisms of $S$ \ onto itself. Then the map
$g\rightarrow\alpha_{g}$ is an injective map from $G$ into $Aut(C(S)),$ the
group of $\ast-$automorphisms. If $G$ is not abelian then this is not a group
homomorphism but an injective anti-homomorphism.

Let us recall that for any group $\Gamma,$ the opposite group, $\Gamma^{op},$
is the same underlying set as $\Gamma$ but with a new group operation defined
by $x\times y=yx.$ Also $\Gamma^{op}$ and $\Gamma$ are isomorphic groups, the
map $g\rightarrow g^{-1}$ gives an isomorphism. So, in the \ preceding
paragraph, $g\rightarrow\alpha_{g}$ is a group isomorphism of $G^{op}$ into
$Aut(C(S)).$ Since $G^{op}$ is isomorphic to $G$ this is not of major
significance. But to avoid confusion, we shall define $\beta_{g}%
=\alpha_{g^{-1}}$ for each $g\in G$. Then $g\rightarrow\beta_{g}$ is a group
isomorphism of $G$ into the automorphism group of $C(S).$

Let $\alpha$ be a $\ast-$automorphism of $A.$ Then $\alpha$ is said to be
\textit{properly outer} if, for each non-zero projection $p,$ the restriction
of $\alpha$ to $pA$ is not the identity. Let $\Gamma$ be a sub-group of
$Aut(A)$ such that every element, except the identity, is properly outer. Then
the action of $\Gamma$ on $A$ is said to be \textit{free.}

Let $G$ be a countable group of homeomorphisms of $S.$ Then, see \cite{zn}, if
$g\rightarrow\beta_{g}$ is a free action of $G$ on $A$ then there exists a
dense G-delta set $Y\subset S$, where $Y$ is $G-$invariant, such that,
whenever $g\in G$ is not the identity, then $g$ has no fixed points in $Y.$ In
other words, for each $y\in Y,$ $G[y]$ is a free orbit. Conversely, we have:

\begin{lemma}
Let $X$ be a dense subset of $S,$ where $X$ is $G-$invariant. Let $G[x]$ be a
free orbit for each $x\in X.$ Then $g\rightarrow\beta_{g}$ is a free action of
$G$ on $C(S).$
\end{lemma}

\begin{proof}
Let $g\in G,$ such that, $\beta_{g}$ is not properly outer. So, for some
non-empty clopen set \ $K$,\textbf{\ }%
\[
\chi_{K}a=\beta_{g}(\chi_{K}a)=(\chi_{K}\circ g^{-1})(a\circ g^{-1}%
)=\chi_{g[K]}(a\circ g^{-1}).
\]
In particular, $K=g[K].$ Suppose that $x_{1}\in K$ with $g(x_{1})\neq x_{1}.$
Then we can find a continuous function $a$ which takes the value $1$ at
$g(x_{1})$ and $0$ at $x_{1}.$ But, from the equation above, this implies
$0=1.$ So $g(x)=x$ for each $x\in K.$ Because $X$ is dense in $S$ and $K$ is a
non-empty open set, there exists $y\in X\cap K.$ So $y$ is a fixed point of
$g$ and $G[y]$ is a free orbit. This is only possible if $g$ is the identity
element of $G.$ Hence the action $g\rightarrow\beta_{g}$ \ is free.
\end{proof}

In all that follows, $G$ is a countably infinite group of homeomorphisms of $S
$ and $Y$ is a dense G-delta subset of $S$, where $Y$ is $G-$invariant. Let
$g\rightarrow\beta_{g}$ be the corresponding action of $G$ as automorphisms of
$A.$ Let $M(C(S),G)$ be the associated (Takenouchi) monotone cross-product. We
shall describe the monotone cross-product below.

The key fact is that, provided the $G-$action is free, the monotone
cross-product algebra can be identified with the monotone complete $C^{\ast}%
$-algebra arising from the $G-$orbit equivalence relation. The end part of
Section 8 already makes this plausible. The reader who is willing to assume it
can skip to Theorem 9.5.

Before saying more about the monotone cross-product, we recall some properties
of the Hamana tensor product, as outlined in \cite{zn}. More detailed
information can be found in \cite{w,x,zm}.

(Comment: An alternative, equivalent, approach avoiding the tensor product, is
to use the theory of Kaplansky-Hilbert modules \cite{z,za,zb,zzc}, see below.) \ 

For the rest of this section, $H$ is a separable Hilbert space and $H_{1}$ is
an arbitrary Hilbert space. Let us fix an orthonormal basis for $H.$ Then,
with respect to this basis, each $V\in L(H_{1})\overline{\otimes}L(H)$ has a
unique representation as a matrix $[V_{ij}]$, where each $V_{ij}$ is in
$L(H_{1})$. Let $M$ be a von Neumann subalgebra of $L(H_{1}).$ Then the
elements of $M\overline{\otimes}L(H)$ are those $V$ for which each $V_{ij}$ is
in $M.$ Let $T$ be any set and $Bnd(T)$ the commutative von Neumann algebra of
all bounded functions on $T.$ Then, as explained in \cite{zn},
$Bnd(T)\overline{\otimes}L(H)$ can be identified with the algebra of all
matrices $[m_{ij}]$ over $Bnd(T)$ for which $t\rightarrow\lbrack m_{ij}(t)]$
is a norm bounded function over $T.$

We denote the commutative $C^{\ast}$-algebra of bounded, complex valued Borel
measurable functions on $Y$ by $B(Y).$ Following \cite{zm}, the product
$B(Y)\widetilde{\otimes}L(H)$ may be defined as the Pedersen-Borel envelope of
$B(Y)\otimes_{\min}L(H)$ inside $Bnd(Y)\overline{\otimes}L(\ H).$ The elements
of $B(Y)\widetilde{\otimes}L(H)$ correspond to the matrices $[b_{ij}]$ where
each $b_{ij}\in B(Y)$ \ and $y\rightarrow\lbrack b_{ij}(y)]$ is a norm bounded
function from $Y$ into $L(H).$

Let $M^{\sigma}(B(Y),G)$ be the subalgebra of $B(Y)\widetilde{\otimes}L(H)$
consisting of those elements of the tensor product which have a matrix
representation over $B(Y)$ of the form $[a_{\gamma,\sigma}]$ $(\gamma\in
G,\sigma\in G)$ where $a_{\gamma\tau,\sigma\tau}(y)=a_{\gamma,\sigma}(\tau y)$
for all $y\in Y$ and all $\gamma,\sigma,\tau$ in $G.$ Let $E$ be the orbit
equivalence relation on $Y$ arising from $G,$ that is%

\[
E=\{(y,gy):y\in Y,g\in G\}.
\]

By Lemma 3.1 \cite{zn} we have:

\begin{lemma}
Assume that each $g\in G$ has no fixed points in $Y$ unless $g$ is the
identity element. Then $M^{\sigma}(B(Y),G)$ is naturally isomorphic to
$\mathcal{M}(E).$
\end{lemma}

For the reader's convenience we sketch the argument. The correspondence
between these two algebras is given as follows. Let $f\in\mathcal{M}(E).$ For
each $\sigma,\gamma$ in $G~$and, for all $y\in Y,$let
\[
a_{\gamma,\sigma}(y)=f(\gamma y,\sigma y)\text{ }.
\]

Then $a_{\gamma,\sigma}$ is in $B(Y)$. Also the norm of $[a_{\gamma,\sigma
}(y)]$ is uniformly bounded for $y\in Y.$ So $[a_{\gamma,\sigma}]$ is in
$B(Y)\widetilde{\otimes}L(H).$ Also%

\[
a_{\gamma\tau,\sigma\tau}(y)=f(\gamma\tau y,\sigma\tau y)=a_{\gamma,\sigma
}(\tau y).
\]

It now follows that $[a_{\gamma,\sigma}]$ is in $M^{\sigma}(B(Y),G).$

Conversely, let $[a_{\gamma,\sigma}]$ be in $M^{\sigma}(B(Y),G).$ We now use
the freeness hypothesis on the action of $G$ on $Y$ to deduce that $(\{(y,\tau
y):y\in Y\})(\tau\in G)$ is a countable family of pairwise disjoint closed
subsets of $E.$ So we can now define, unambiguously, a function
$f:E\rightarrow%
%TCIMACRO{\U{2102} }%
%BeginExpansion
\mathbb{C}
%EndExpansion
$ by $f(y,\tau y)=a_{\iota,\tau}(y).$ This is a bounded Borel\textbf{\ }%
function on $E.$ It follows from the definition of $M^{\sigma}(B(Y),G)$ that
$a_{\gamma,\sigma}(y)=a_{\iota,\sigma\gamma^{-1}}(\gamma y)=f(\gamma y,\sigma
y)$ for all $\sigma,\gamma$ in $G$ and all $y$ in $Y$. From this it follows
that $f$ is in $\mathcal{M}(E).$

Let $\pi$ be the quotient homomorphism from $B(Y)$ onto $C(S).$ (Each $F$ in
$B(S)$ differs only on a meagre set from a unique function in $C(S).$ Since
$S\backslash Y$ is a meagre set, each $f$ in $B(Y)$ corresponds to a unique
element of $C(S)$ which we denote $\pi(f)$).

Each element of the Hamana tensor product $C(S)\overline{\otimes}L(H)$ has a
representation as a matrix over $C(S).$ [Remark: the product is not
straightforward. ] By Theorem 2.5 \cite{zm} there exists a $\sigma
-$homomorphism $\Pi$ from $B(Y)\widetilde{\otimes}L(H)$ onto $C(S)\overline
{\otimes}L(H)$ where $\Pi([b_{\gamma,\sigma}])=[\pi(b_{\gamma,\sigma})].$

(Comment: As indicated above, we may use Kaplansky-Hilbert modules as an
alternative approach. We replace the Hilbert space $H$ by the separable
Hilbert space $\ell^{2}(G),$ consisting of all square summable complex
functions on $G$ with the standard basis $\{\xi_{\gamma}\}_{\gamma\in G}$. Let
$\{e_{\gamma,\sigma}\}$ be the standard system of matrix units for
$\mathcal{L}(H)$.

\smallskip\ Let $\mathfrak{M}=\ell^{2}(G,C(S))$ be the Kaplansky-Hilbert
module, over a monotone complete $C^{\ast}$-algebra $C(S)$, of all $\ell^{2}%
$-summable family of elements in $C(S)$ on $G$ and let $\{\delta_{\sigma}\}$
be the standard basis for $\mathfrak{M}$ \cite{zb} and, as Kaplansky defined,
let $\{E_{\gamma,\sigma}\}$ be the standard system of matrix units for the
monotone complete $C^{\ast}$-algebra $\text{End}_{C(S)}(\mathfrak{M})$ of all
bounded module endomorphisms on $\mathfrak{M}$. Then we know that
$(C(S)\overline{\otimes}\mathcal{L}(H),\{1\otimes e_{\gamma,\sigma}\})$ is
isomorphic to ($\text{End}_{C(S)}(\mathfrak{M}),\{E_{\gamma,\sigma}\})$ by
using \cite{z}. Then $\Pi([b_{\gamma,\sigma}])$ can be identified with the
$\ell^{2}$-limit\textbf{\ \ } $\sum_{\sigma\in G}\sum_{\gamma\in G}%
\pi(b_{\gamma,\sigma})E_{\gamma,\sigma}$ using Kadison-Pedersen
order-convergence \cite{y}.)

But the Takenouchi monotone cross-product is the subalgebra of $C(S)\overline
{\otimes}L(\ H)$ corresponding to matrices \thinspace$\lbrack a_{\gamma
,\sigma}]$ over $C(S)$ for which $\beta_{\tau^{-1}}(a_{\gamma,\sigma
})=a_{\gamma\tau,\sigma\tau}$ for all $\gamma,\sigma,\tau$ in $G.$
Equivalently, $a_{\gamma\tau,\sigma\tau}(s)=a_{\gamma,\sigma}(\tau s)$ for all
$\gamma,\sigma,\tau$ in $G$ and $s\in S.$ From this it follows that the
homomorphism $\Pi$ maps $M^{\sigma}(B(Y),G)$ onto $M(C(S),G).$ See Lemma 3.2
\cite{zn}.

The diagonal subalgebra of $M(C(S),G)$ consists of those matrices
\ $[a_{\gamma,\sigma}]$ which vanish off the diagonal i.e. $a_{\gamma,\sigma
}=0$ for $\gamma\neq\sigma.$ Also, $\beta_{\tau^{-1}}(a_{\iota,\iota}%
)=a_{\tau,\tau}$ for each $\tau\in G$. It follows that we can define an
isomorphism from $A$ onto the diagonal of $M(C(S),G)$ by $j(a)=Diag(...,\beta
_{\tau^{-1}}(a),...)$. We recall, see \cite{zp} and \cite{ze}, that the
freeness of the action $G$ implies that the diagonal algebra of $M(C(S),G)$ is
a maximal abelian $\ast-$subalgebra of $M(C(S),G),$ alternatively, apply the
results of Section 8.

Then, by Lemma 3.3 \cite{zn}, we have:

\begin{lemma}
Let $E$ be the graph of the relation of orbit equivalence given by $G$ acting
on $Y.$ Then there exists a $\sigma-$normal homomorphism $\delta$ from
$\mathcal{M}(E)$ onto $M(C(S),\beta,G).$

The kernel of $\delta$ is $J=\{z\in\mathcal{M}(E):D(zz^{\ast})$ vanishes off a
meagre subset of $Y\}.$

Furthermore, $\delta$ maps the diagonal subalgebra of $\mathcal{M}(E)$ onto
the diagonal subalgebra of $M(C(S),\beta,G).$ In particular, $\delta$ induces
an isomorphism of $\mathcal{M}(E)/J$ onto $M(C(S),G).$
\end{lemma}

Let $C(S)\times_{\beta}G$ be the smallest monotone closed $\ast-$subalgebra of
$M(C(S),G)$ which contains the diagonal and each unitary which implements the
$\beta-$action of $G.$ It will sometimes be convenient to call $C(S)\times
_{\beta}G$ the "small" monotone cross-product. It turns out that
$C(S)\times_{\beta}G$ does not depend on $G$ only on the orbit equivalence
relation. This is not at all obvious but is an immediate consequence of
Theorem 10.1. This theorem shows that when $w$ is a unitary in $M(C(S),G)$
such that $w$ normalises the diagonal then $w$ is in $C(S)\times_{\beta}G.$
So, the isomorphism of $M(C(S),G)$ onto $\mathcal{M}(E)/J$ maps $C(S)\times
_{\beta}G$ onto the normalizer subalgebra of $\mathcal{M}(E)/J.$

Does the small monotone cross product equal the "big" monotone cross-product ?
Equivalently, is $\mathcal{M}(E)/J$ equal to its normaliser subalgebra ? This
is unknown, but we can approximate each element of $M(C(S),G)$ by a sequence
in $C(S)\times_{\beta}G,$ in the following precise sense:

\begin{lemma}
Let $z\in M(C(S),G).$ Then there exists a sequence $(z_{n})$ in $C(S)\times
_{\beta}G$ such that the sequence $D((z-z_{n})^{\ast}(z-z_{n}))$ is monotone
decreasing and
\[%
%TCIMACRO{\tbigwedge \limits_{n=1}^{\infty}}%
%BeginExpansion
{\textstyle\bigwedge\limits_{n=1}^{\infty}}
%EndExpansion
D((z-z_{n})^{\ast}(z-z_{n}))=0.
\]

\end{lemma}

\begin{proof}
This follows from Proposition 8.15.
\end{proof}

\begin{theorem}
Let $G_{j}(j=1,2)$ be countable, infinite groups of homeomorphisms of $S$. Let
$g\rightarrow\beta_{g}^{j}$ be the corresponding action of $G_{j}$ as
automorphisms of $C(S).$ Let $Y$ be a G-delta, dense subset of $S$ such that
$G_{j}[Y]=Y$ and $G_{j}$ acts freely on $Y$. Let $E_{j}$ be the orbit
equivalence relation on $Y$ arising from the action of $G_{j}.$ Suppose that
$E_{1}=E_{2}.$ Then there exists an isomorphism of $M(C(S),G_{1})$ onto
$M(C(S),G_{2})$ which maps the diagonal algebra of $M(C(S),G_{1})$ onto the
diagonal algebra of $M(C(S),G_{2}).$ Furthermore, this isomorphism maps
$C(S)\times_{\beta^{1}}G_{1}$ onto $C(S)\times_{\beta^{2}}G_{2}.$
\end{theorem}

\begin{proof}
The first part is a straight forward application of Lemma 9.3. Let
$E=E_{1}=E_{2}.$ Then both algebras are isomorphic to $\mathcal{M}(E)/J.$ The
final part is a consequence of Theorem 10.1.
\end{proof}

In the next result, we require $S$ to be separable.

\begin{corollary}
Let $G$ be a countable, infinite group of homeomorphisms of $S$. Suppose, for
some $s_{0}\in S,G[s_{0}]$ is a free, dense orbit. Then there exists an
isomorphism $\phi$ of $\bigoplus\mathbb{Z}_{2}$ into $Homeo(S),$ such that
there exists an isomorphism of $M(C(S),G)$ onto $M(C(S),\bigoplus
\mathbb{Z}_{2})$ which maps the diagonal algebra of $M(C(S),G)$ onto the
diagonal algebra of $M(C(S),\bigoplus\mathbb{Z}_{2}).$
\end{corollary}

\begin{proof}
By Theorem 7.10, there exists an isomorphism $\phi$ from $\bigoplus
\mathbb{Z}_{2}$ into $Homeo(S)$ such that there exists a dense G-delta set $Y$
in $S$ with the following properties. First, $Y$ is invariant under the action
of both $G$ \ and $\bigoplus\mathbb{Z}_{2}.$ Secondly the induced orbit
equivalence relations coincide on $Y$. Theorem 9.5 then gives the result.
\end{proof}

REMARK: By Lemma 5.1, when $G$ has a dense orbit in $S$ then the action on $S
$ is such that the only invariant clopen set is empty or the whole space. This
implies that the action $g\rightarrow\beta_{g}$ is ergodic, as defined in
\cite{zp,ze}, which implies that the algebra $M(C(S),G)$ is a monotone
complete factor.

\begin{lemma}
Let $\mathcal{B}$ be a Boolean $\sigma-$algebra. Let $(p_{n})$ be a sequence
in $\mathcal{B}$ which $\sigma-$generates $\mathcal{B}$, that is,
$\mathcal{B}$ is the smallest $\sigma-$subalgebra of $\mathcal{B}$ which
contains each $p_{n}$. Let $\Gamma$ be a group of automorphisms of
$\mathcal{B}.$ Let $\Gamma$ be the union of an increasing sequence of finite
subgroups $(\Gamma_{n}).$ Then we can find an increasing sequence of finite
Boolean algebras $(\mathcal{B}_{n})$ where each $\mathcal{B}_{n}$ is invariant
under the action of $\Gamma_{n}$ and $\cup\mathcal{B}_{n}$ \ is a Boolean
algebra which $\sigma-$generates $\mathcal{B}$.
\end{lemma}

\begin{proof}
For any natural number $k$ the free Boolean algebra on $k$ generators has
$2^{k}$ elements. So a Boolean algebra with $k$ generators, being a quotient
of the corresponding free algebra, has a finite number of elements.

We proceed inductively. Let $B_{1}$ be the subalgebra generated by
$\{g(p_{1}):g\in\Gamma_{1}\}.$ Then $B_{1}$ is finite and $\Gamma_{1}%
-$invariant. Suppose we have constructed $B_{1},B_{2}...B_{n}.$ Then
$B_{n}\cup\{p_{n+1}\}$ is a finite set. So its saturation by the finite group
$\Gamma_{n+1}$ is again a finite set. So the Boolean algebra this generates,
call it $B_{n+1},$ is finite. Clearly $B_{n}\subset B_{n+1}$ and $B_{n+1}$ is
invariant under the action of $\Gamma_{n+1}.$
\end{proof}

A commutative monotone complete $C^{\ast}$-algebra is \textit{countably
}$\sigma-$\textit{generated} if its Boolean algebra of projections is
$\sigma-$generated by a countable subset.

\begin{proposition}
Let the Boolean algebra of projections in $C(S)$ be countably $\sigma
-$generated by $(p_{n})$. Let $\Gamma$ be a group of automorphisms of $C(S).$
Let $\Gamma$ be the union of an increasing sequence of finite subgroups
$(\Gamma_{n}).$ Let $g\rightarrow u_{g}$ be the unitary representation of
$\Gamma$ in $M(C(S),\Gamma)$ which implements the action of $\Gamma$ on the
diagonal algebra $A$. Let $\pi$ \ be the canonical isomorphism of $C(S)$ onto
$A.$ Then the $C^{\ast}$-algebra generated by $\{u_{g}:g\in\Gamma\}\cup
\{\pi(p_{n}):n=1,2...\}\ $is the closure of an increasing sequence of finite
dimensional subalgebras.
\end{proposition}

\begin{proof}
Let $\mathcal{B}$ be the complete Boolean algebra of all projections in $A$.
By Lemma 9.7, we can find an increasing sequence of finite Boolean algebras of
projections $(\mathcal{B}_{n})$ where $\mathcal{B}$ is $\sigma-$generated by
$\cup\mathcal{B}_{n}$ and each $\mathcal{B}_{n}$ is invariant under the action
of $\Gamma_{n}.$

Let $A_{n}$ be the (complex) linear span of $\mathcal{B}_{n}.$ Then $A_{n}$ is
a finite dimensional $\ast-$subalgebra of $A$. Also, for $g\in\Gamma_{n}, $
$u_{g}A_{n}u_{g}^{\ast}=A_{n}.$

Now let $B_{n}$ be the linear span of $\{bu_{g}:g\in\Gamma_{n}$ and $b\in
A_{n}\}.$ Then $B_{n}$ is a finite dimensional $\ast-$subalgebra. Clearly
$(B_{n})$ is an increasing sequence and $\cup_{n=1}^{\infty}B_{n}$ is a $\ast-
$subalgebra generated by $\{u_{g}:g\in\Gamma\} \cup\{ \pi(p_{n}):n=1,2...\}.$
\end{proof}

\section{The normaliser algebra}

After the proof of Theorem 9.5, we made some claims concerning the normaliser
algebra of a monotone cross-product. They follow from Theorem 10.1 below.

In this section $M$ is a monotone complete $C^{\ast}$-algebra with a maximal
abelian $\ast-$subalgebra $A$ and $D:M\rightarrow A$ a positive, linear,
idempotent map of $M$ onto $A.$ It follows from a theorem of Tomiyama
\cite{zu} that $D$ is a conditional expectation, that is, $D(azb)=a(Dz)b$ for
each $z\in M$ and every $a,b$ in $A$.\textbf{\ }Clearly the monotone
cross-product algebras considered in Section 9 satisfy these conditions.

We recall that a unitary $w$ in $M$ is a \textit{nomaliser} of $A$ if
$wAw^{\ast}=A.$ It is clear that the normalisers of $A$ form a subgroup of the
unitaries in $M.$ We use $\mathcal{N}(A,M)$ to denote this normaliser
subgroup. Let $M_{\mathcal{N}}$ be the smallest monotone closed $\ast
-$subalgebra of $M$ which contains $\mathcal{N}(A,M)$. Then $M_{\mathcal{N}}$
is said to be the\textit{\ normaliser subalgebra} of $M.$

Let $G$ be a countable group. Let $g\rightarrow u_{g}$ be a unitary
representation of $G$ in $\mathcal{N}(A,M).$ Let $\lambda_{g}(a)=u_{g}%
au_{g}^{\ast}.$

Let $\alpha$ be an automorphism of $A$. We recall that $\alpha$ is
\textit{properly outer} if, for each non-zero projection $e\in A$, the
restriction of $\alpha$ to $eA$ is not the identity map. We further recall
that the action $g\rightarrow\lambda_{g}$ is \textit{free} provided, for each
$g$ other than the identity, $\lambda_{g}$ is properly outer.

\begin{theorem}
Let $M_{0}$ be the smallest monotone closed subalgebra of $M$ which contains
$A\cup\mathcal{\{}u_{g}:g\in G\mathcal{\}}.$We suppose that:

(i) The action $g\rightarrow\lambda_{g}$ is free.

(ii) For each $z\in M,$ if $D(zu_{g})=0$ for every $g\in G$ then $z=0.$

Then $M_{0}$ contains \textit{every }unitary in $M$ which normalises $A$, that
is, $M_{0}=M_{\mathcal{N}}.$
\end{theorem}

\begin{proof}
Let $w$ be a unitary in $M$ which normalises $A$. Let $\sigma$ be the
automorphism of $A$ induced by $w$. Then, for each $a\in A$, we have
$waw^{\ast}=\sigma(a).$ So $wa=\sigma(a)w.$ Hence, for each $g,$ we have%
\[
D(wau_{g})=D(\sigma(a)wu_{g}).
\]
But $D$ is a conditional expectation. So%
\[
D(wu_{g})u_{g}^{\ast}au_{g}=D(wau_{g})=D(\sigma(a)wu_{g})=\sigma(a)D(wu_{g}).
\]
Because $A$ is abelian, it follows that $(\sigma(a)-\lambda_{g}^{-1}%
(a))D(wu_{g})=0.$

Let $p_{g}$ be the range projection of $D(wu_{g})D(wu_{g})^{\ast}$ in $A.$ So,
for every $a\in A,$%
\[
\text{(\#) \ \ \ \ }(\sigma(a)-\lambda_{g}^{-1}(a))p_{g}=0.
\]
Fix $g$ and $h$ with $g\neq h,$ and let $e$ be the projection $p_{g}p_{h}$.
Then we have, for each $a\in A,$%
\[
(\lambda_{h}^{-1}(a)-\lambda_{g}^{-1}(a))e=(\sigma(a)-\lambda_{g}%
^{-1}(a))p_{g}p_{h}-(\sigma(a)-\lambda_{h}^{-1}(a))p_{h}p_{g}=0.
\]
Let $b$ be any element of $A$ and let $a=\lambda_{g}(b)$. Then $(\lambda
_{h^{-1}g}(b)-b)e=0.$ If $e\neq0,$ then by (i) it follows that $h^{-1}g$ is
the identity element of $G.$ But this implies $g=h$, which is a contradiction.
So $0=e=p_{g}p_{h}.$ So $\{p_{g}:g\in G\}$ is a (countable) family of
orthogonal projections.

Let $q$ be a projection in $A$ which is orthogonal to each $p_{g}.$ Then
$qD(wu_{g})=qp_{g}D(wu_{g})=0.$ So $D(qwu_{g})=0$ for each $g\in G.$ Hence, by
applying hypothesis (ii), $qw=0.$ But $ww^{\ast}=1.$ So $q=0.$ Thus $%
%TCIMACRO{\tsum }%
%BeginExpansion
{\textstyle\sum}
%EndExpansion
p_{g}=1.$

From (\#) we see that%
\[
(\lambda_{g}\sigma(a)-a)\lambda_{g}(p_{g})=0.
\]
We define $q_{g}$ to be the projection $\lambda_{g}(p_{g})$. Then%
\[
\text{(\# \#) \ \ \ }(a-\lambda_{g}\sigma(a))q_{g}=0.
\]
By arguing in a similar fashion to the above, we find that $\{q_{g}:g\in G\}$
is a family of orthogonal projections in $A$ with $%
%TCIMACRO{\tsum }%
%BeginExpansion
{\textstyle\sum}
%EndExpansion
q_{g}=1.$

For each $g\in G,$ let $v_{g}=u_{g}p_{g}.$ Then $v_{g}$ is in $M_{0}$ and is a
partial isometry with $v_{g}v_{g}^{\ast}=q_{g}$ and $v_{g}^{\ast}v_{g}=p_{g}.$
By the General Additivity of Equivalence for AW*-algebras, see page 129
\cite{c}, there exists a unitary $v$ in $M_{0}$ such that $q_{g}v=v_{g}$ and
$vp_{g}=v_{g}p_{g}=u_{g}p_{g}.\mathbf{\ }$

From (\#), for each $a\in A,$

$\sigma(a)p_{g}=u_{g}^{\ast}au_{g}p_{g}=p_{g}v^{\ast}avp_{g}=v^{\ast}avp_{g}.
$

So $(\sigma(a)-v^{\ast}av)p_{g}=0.$ Let $y=\sigma(a)-v^{\ast}av.$ Then
$y^{\ast}yp_{g}=0.$ So the range projection of $y^{\ast}y$ is orthogonal to
$p_{g}$ for each $g,$ and hence is $0.$ So $y=0.$ It now follows that
$waw^{\ast}=vav^{\ast}$ for each $a\in A.$ Then $v^{\ast}w$ \ commutes with
each element of $A.$ Since $A$ is maximal abelian in $M$ it follows that
$v^{\ast}w$ is in $A.$ Since $v$ is in $M_{0},$ it now follows that $w$ is in
$M_{0}.$
\end{proof}

We note that the above theorem does not require the action $g\rightarrow
\lambda_{g}$ to be ergodic.

\section{Free dense actions of the Dyadic Group}

We have said a great deal about $G-$actions with a free dense orbit and the
algebras associated with them. It is incumbent on us to provide examples. We
do this in this section. We have seen that when constructing monotone complete
algebras from the action of a countably infinite group $G$ on an extremally
disconnected space $S,$ what matters is the orbit equivalence relation induced
on $S.$ When the action of $G$ has a free, dense orbit in $S $ then we have
shown that the orbit equivalence relation (and hence the associated algebras)
can be obtained from an action of $%
%TCIMACRO{\tbigoplus }%
%BeginExpansion
{\textstyle\bigoplus}
%EndExpansion%
%TCIMACRO{\U{2124} }%
%BeginExpansion
\mathbb{Z}
%EndExpansion
_{2}$ with a free, dense orbit. So, when searching for free, dense group
actions, it suffices to find them when \ the group is $%
%TCIMACRO{\tbigoplus }%
%BeginExpansion
{\textstyle\bigoplus}
%EndExpansion%
%TCIMACRO{\U{2124} }%
%BeginExpansion
\mathbb{Z}
%EndExpansion
_{2}.$

In this section we construct such actions of $%
%TCIMACRO{\tbigoplus }%
%BeginExpansion
{\textstyle\bigoplus}
%EndExpansion%
%TCIMACRO{\U{2124} }%
%BeginExpansion
\mathbb{Z}
%EndExpansion
_{2}.$ As an application, we will find $2^{c}$ hyperfinite factors which take
$2^{c}$\ different values in the weight semigroup \cite{zk}.

We begin with some purely algebraic considerations before introducing
topologies and continuity. We will end up with a huge number of examples.

We use $F(S)$ to denote the collection of all finite subsets of a set $S.$ We
shall always regard the empty set, the set with no elements, as a finite set.
We use $%
%TCIMACRO{\U{2115} }%
%BeginExpansion
\mathbb{N}
%EndExpansion
$ to be the set of natural numbers, excluding $0.$ Let $C=\{f_{\mathbf{k}%
}:\mathbf{k}\in F(%
%TCIMACRO{\U{2115} }%
%BeginExpansion
\mathbb{N}
%EndExpansion
)\}$ be a countable set where $\mathbf{k}\rightarrow f_{\mathbf{k}}$ is a
bijection. For each $n\in%
%TCIMACRO{\U{2115} }%
%BeginExpansion
\mathbb{N}
%EndExpansion
$ let $\sigma_{n}$ be defined on $C$ by%

\[
\sigma_{n}(f_{\mathbf{k}})=\left\{
\begin{array}
[c]{cc}%
f_{\mathbf{k}\setminus\{n\}} & n\in\mathbf{k}\\
f_{\mathbf{k}\cup\{n\}} & \text{if $n\notin\mathbf{k.}$}%
\end{array}
\right.
\]

\begin{lemma}
(i) For each $n,$ $\sigma_{n}$ is a bijection of $C$ onto $C,$ and $\sigma
_{n}\sigma_{n}=id$, where $id$ is the identity map on $C.$

(ii) When $m\neq n$ then $\sigma_{m}\sigma_{n}=\sigma_{n}\sigma_{m}.$
\end{lemma}

\begin{proof}
(i) It is clear that $\sigma_{n}\sigma_{n}=id$ and hence $\sigma_{n}$ is a bijection.\ 

(ii) Fix $f_{\mathbf{k}}$. Then we need to show $\sigma_{m}\sigma
_{n}(f_{\mathbf{k}})=\sigma_{n}\sigma_{m}(f_{\mathbf{k}}).$ This is a
straightforward calculation, considering separately the four cases when
$\mathbf{k}$ contains neither $m$ nor $n,$ contains both $m$ and $n,$ contains
$m$ but not $n$ and contains $n$ but not $m.$
\end{proof}

\bigskip

We recall that the Dyadic Group, $%
%TCIMACRO{\tbigoplus }%
%BeginExpansion
{\textstyle\bigoplus}
%EndExpansion%
%TCIMACRO{\U{2124} }%
%BeginExpansion
\mathbb{Z}
%EndExpansion
_{2},$ can be identified with the additive group of functions from $\mathbb{N}
$ to $%
%TCIMACRO{\U{2124} }%
%BeginExpansion
\mathbb{Z}
%EndExpansion
_{2},$ where each function takes only finitely many non-zero values. For each
$n\in\mathbb{N}$, let $g_{n}$, be the element defined by $g_{n}(m)=\delta
_{m,n}$ for all $m\in\mathbb{N}$. Then $\{g_{n}\ |\ n\in\mathbb{N}\}$ is a set
of generators of $%
%TCIMACRO{\tbigoplus }%
%BeginExpansion
{\textstyle\bigoplus}
%EndExpansion%
%TCIMACRO{\U{2124} }%
%BeginExpansion
\mathbb{Z}
%EndExpansion
_{2}$.

Take any $g\in%
%TCIMACRO{\tbigoplus }%
%BeginExpansion
{\textstyle\bigoplus}
%EndExpansion%
%TCIMACRO{\U{2124} }%
%BeginExpansion
\mathbb{Z}
%EndExpansion
_{2}$ then $g$ has a unique representation as $g=g_{n_{1}}+\cdots+g_{n_{p}}$
where $1\leq n_{1}<\cdots<n_{p}$ or $g$ is the zero. Let us define
\[
\varepsilon_{g}=\sigma_{n_{1}}\sigma_{n_{2}}...\sigma_{n_{p}}.
\]
Here we adopt the notational convention that $\sigma_{n_{1}}\sigma_{n_{2}%
}...\sigma_{n_{p}}$ denotes the identity map of $C$ onto itself when
$\{n_{1},..n_{p}\}=\emptyset.$

Then $g\rightarrow\varepsilon_{g}$ is a group homomorphism of $%
%TCIMACRO{\tbigoplus }%
%BeginExpansion
{\textstyle\bigoplus}
%EndExpansion%
%TCIMACRO{\U{2124} }%
%BeginExpansion
\mathbb{Z}
%EndExpansion
_{2}$ into the group of bijections of $C$ onto $C.$ It will follow from Lemma
11.2 (ii) that this homomorphism is injective.

\begin{lemma}
(i) $C=\{ \varepsilon_{g}(f_{\emptyset}):g\in\allowbreak%
%TCIMACRO{\tbigoplus }%
%BeginExpansion
{\textstyle\bigoplus}
%EndExpansion%
%TCIMACRO{\U{2124} }%
%BeginExpansion
\mathbb{Z}
%EndExpansion
_{2}\}=\{ \sigma_{n_{1}}\sigma_{n_{2}}...\sigma_{n_{p}}(f_{\emptyset}%
):\{n_{1},n_{2}\cdots,n_{p}\} \in F(%
%TCIMACRO{\U{2115} }%
%BeginExpansion
\mathbb{N}
%EndExpansion
)\}$.

In other words $C$ is an orbit.

(ii) For each $\mathbf{k}\in F(%
%TCIMACRO{\U{2115} }%
%BeginExpansion
\mathbb{N}
%EndExpansion
),$ where $\mathbf{k}=\{n_{1},...,n_{p}\},$%
\[
\sigma_{n_{1}}\sigma_{n_{2}}...\sigma_{n_{p}}(f_{\mathbf{k}})=f_{\mathbf{k}}%
\]
only if $\sigma_{n_{1}}\sigma_{n_{2}}...\sigma_{n_{p}}=id.$
\end{lemma}

\begin{proof}
(i) Let $\mathbf{k}=\{n_{1},\cdots,n_{p}\}$ where $n_{i}\neq n_{j}$ for $i\neq
j.$ Then $\sigma_{n_{1}}\sigma_{n_{2}}...\sigma_{n_{p}}(f_{\emptyset
})=f_{\mathbf{k}}.$

(ii) Assume this is false. Then, for some $\mathbf{k}\in F(%
%TCIMACRO{\U{2115} }%
%BeginExpansion
\mathbb{N}
%EndExpansion
)$ we have $\sigma_{n_{1}}\sigma_{n_{2}}...\sigma_{n_{p}}(f_{\mathbf{k}%
})=f_{\mathbf{k}}$ where $\sigma_{n_{1}}\sigma_{n_{2}}...\sigma_{n_{p}}$ is
not the identity map. So we may assume, without loss of generality, that
$\{n_{1},n_{2},...n_{p}\}=\mathbf{m}$ is a non-empty set of $p$ natural numbers.

First consider the case where $\mathbf{k}$ is the empty set. Then
$\sigma_{n_{1}}\sigma_{n_{2}}...\sigma_{n_{p}}(f_{\mathbf{\emptyset}%
})=f_{\mathbf{\emptyset}}.$ So $f_{\mathbf{m}}=f_{\mathbf{\emptyset}}.$ But
this is not possible because the map $\mathbf{k}\rightarrow f_{\mathbf{k}}$ is injective.

So $\mathbf{k}$ cannot be the empty set; let $\mathbf{k=}$ $\{m_{1}%
,m_{2},...m_{q}\}$. Then $\sigma_{m_{1}}\sigma_{m_{2}}...\sigma_{m_{q}%
}(f_{\emptyset})=f_{\mathbf{k}}.$ Hence%
\[
\sigma_{n_{1}}\sigma_{n_{2}}...\sigma_{n_{p}}\sigma_{m_{1}}\sigma_{m_{2}%
}...\sigma_{m_{q}}(f_{\emptyset})=\sigma_{m_{1}}\sigma_{m_{2}}...\sigma
_{m_{q}}(f_{\emptyset}).
\]
On using the fact that the $\sigma_{j}$ are idempotent and mutually
commutative, we find that $\sigma_{n_{1}}\sigma_{n_{2}}...\sigma_{n_{p}%
}(f_{\emptyset})=f_{\emptyset}.$ But, from the above argument, this is
impossible. So (ii) is proved.
\end{proof}

In \cite{zk} we consider the "Big Cantor Space" $\{0,1\}^{%
%TCIMACRO{\U{211d} }%
%BeginExpansion
\mathbb{R}
%EndExpansion
},$ which is compact, totally disconnected and separable but not metrisable or
second countable. In \cite{zk} we pointed out that each compact, separable,
totally disconnected space is homeomorphic to a subspace of $\{0,1\}^{%
%TCIMACRO{\U{211d} }%
%BeginExpansion
\mathbb{R}
%EndExpansion
}.$ Let $C$ be a countable subset of $\{0,1\}^{%
%TCIMACRO{\U{211d} }%
%BeginExpansion
\mathbb{R}
%EndExpansion
}$ then $clC$, the closure of $C,$ is a compact separable, totally
disconnected space. This implies that $C$ is completely regular and hence has
a Stone-Czech compactification $\beta C.$

We recall from the work of Section 6, that the regular $\sigma-$completion of
$C(clC)$ is monotone complete and can be identified with $B^{\infty
}(clC)/M(clC).$ Let $\widehat{clC\text{ }}$ be the \ maximal ideal space of
$B^{\infty}(clC)/M(clC).$ Then this may be identified with the Stone space of
the complete Boolean algebra of regular open subsets of $clC.$ By varying $C$
in a carefully controlled way, we exhibited $2^{c}$ essentially different
extremally disconnected spaces in the form $\widehat{clC}$.

For each of these spaces $\widehat{clC}$\ we shall construct an action of $%
%TCIMACRO{\tbigoplus }%
%BeginExpansion
{\textstyle\bigoplus}
%EndExpansion%
%TCIMACRO{\U{2124} }%
%BeginExpansion
\mathbb{Z}
%EndExpansion
_{2}$ with a free dense orbit.

We need to begin by recalling some notions from \cite{zk}. A pair
$(T,\mathbf{O})$ is said to be \textit{feasible} if it satisfies the following conditions:

(i) $T$ is a set of cardinality $c=2^{\aleph_{0}}; \mathbf{O}=(O_{n}%
)(n=1,2...)$ is an infinite sequence of non-empty subsets of $T$, with
$O_{m}\neq O_{n}$ whenever $m\neq n.$

(ii) Let $M$ be a finite subset of $T$ and $t\in T\backslash M.$ For each
natural number $m$ there exists $n>m$ such that $t\in O_{n}$ and $O_{n}\cap
M=\varnothing.$

In other words\textbf{\ }$\{n\in%
%TCIMACRO{\U{2115} }%
%BeginExpansion
\mathbb{N}
%EndExpansion
:t\in O_{n}$\textbf{\ }and\textbf{\ }$O_{n}\cap M=\varnothing\}$\ is an
infinite set.

An example satisfying these conditions can be obtained by putting $T=2^{%
%TCIMACRO{\U{2115} }%
%BeginExpansion
\mathbb{N}
%EndExpansion
},$ the Cantor space and letting $\mathbf{O}$ be an enumeration (without
repetitions) of the (countable) collection of all non-empty clopen subsets.

For the rest of this section $(T,\mathbf{O})$ will be a fixed but arbitrary
feasible pair.

Let $(T,\mathbf{O})$ be a feasible pair and let $R$ be a subset of $T.$ Then
$R$ is said to be \textit{admissible} if

(i) $R$ is a subset of $T,$with $\#R=\#(T\backslash R)=c.$

(ii) $O_{n}$ is not a subset of $R$ for any natural number $n.$

Return to the example where $T$ is the Cantor space and $\mathbf{O}$ an
enumeration of the non-empty clopen subsets. Then, whenever $R$ $\subset2^{%
%TCIMACRO{\U{2115} }%
%BeginExpansion
\mathbb{N}
%EndExpansion
}$ is nowhere dense and of cardinality $c,$ $R$ is admissible.

Throughout this section the feasible pair is kept fixed and the existence of
at least one admissible set is assumed.\textbf{\ }For the moment, $R$ is a
fixed admissible subset of $T.$ Later on we shall vary $R.$

Since $F(%
%TCIMACRO{\U{2115} }%
%BeginExpansion
\mathbb{N}
%EndExpansion
)\times F(T)$ has cardinality $c,$ we can identify the Big Cantor space with
$2^{F(%
%TCIMACRO{\U{2115} }%
%BeginExpansion
\mathbb{N}
%EndExpansion
)\times F(T)}.$ For each $\mathbf{k}\in F(%
%TCIMACRO{\U{2115} }%
%BeginExpansion
\mathbb{N}
%EndExpansion
),$ let $f_{\mathbf{k}}\in2^{F(%
%TCIMACRO{\U{2115} }%
%BeginExpansion
\mathbb{N}
%EndExpansion
)\times F(T)}$ be the characteristic function of the set
\[
\{(\mathbf{l},L):L\in F(T\backslash R),\mathbf{l}\subset\mathbf{k}%
\ \text{and}\ O_{n}\cap L=\varnothing\ \text{whenever}\ n\in\mathbf{k}%
\ \text{and}\ n\notin\mathbf{l}\}.
\]

As in \cite{zk}, let $N(t)=\{n\in%
%TCIMACRO{\U{2115} }%
%BeginExpansion
\mathbb{N}
%EndExpansion
:t\in O_{n}\}.~$By feasibility, this set is infinite for each $t\in T.$ It is
immediate that%

\[
f_{\mathbf{k}}(\mathbf{l},L)=1\text{ precisely when }L\in F(T\backslash
R),\mathbf{l}\subset\mathbf{k}\ \text{and, for each }t\in L,\ N(t)\cap
(\mathbf{k}\backslash\mathbf{l})=\varnothing.
\]

Let $X_{R}$ be the countable set $\{f_{\mathbf{k}}:\mathbf{k}\in F(%
%TCIMACRO{\U{2115} }%
%BeginExpansion
\mathbb{N}
%EndExpansion
)\}.$ Let $K_{R}$ be the closure of $X_{R}$ in the Big Cantor space. Then
$K_{R}$ is a (separable) compact Hausdorff totally disconnected space with
respect to the relative topology induced by the product topology of the Big
Cantor space. We always suppose $X_{R}$ to be equipped with the relative
topology induced by $K_{R}.$

Let $C=X_{R}.$ If the map $\mathbf{k}\rightarrow f_{\mathbf{k}}$ is an
injection then we can define $\sigma_{n}$ on $X_{R}$ as before.

\begin{lemma}
Let $f_{\mathbf{k}}=f_{\mathbf{m}}.$ Then $\mathbf{k}=\mathbf{m}.$
\end{lemma}

\begin{proof}
By definition, $f_{\mathbf{k}}(\mathbf{l},\emptyset)=1$ precisely when
$\mathbf{l}\subset\mathbf{k}.$ Since $f_{\mathbf{k}}(\mathbf{m},\emptyset
)=f_{\mathbf{m}}(\mathbf{m},\emptyset)=1$ it follows that $\mathbf{m}%
\subset\mathbf{k}.$ Similarly, $\mathbf{k}\subset\mathbf{m}.$ Hence
$\mathbf{m}=\mathbf{k}.$
\end{proof}

For each $(\mathbf{k},K)\in F(%
%TCIMACRO{\U{2115} }%
%BeginExpansion
\mathbb{N}
%EndExpansion
)\times F(T)$ let $E_{(\mathbf{k},K)}=\{x\in K_{R}:x(\mathbf{k},K)=1\}.$ The
definition of the product topology of the Big Cantor space implies that
$E_{(\mathbf{k},K)}$ and its compliment $E_{(\mathbf{k},K)}^{c}$ are clopen
subsets of $K_{R}$. It also follows from the definition of the product
topology that finite intersections of such clopen sets form a base for the
topology of $K_{R}.$ Hence their intersections with $X_{R}$ give a base for
the relative topology of $X_{R}.$ But we saw in \cite{zk} that, in fact,
$\{E_{(\mathbf{k},K)}\cap X_{R}:\mathbf{k}\in F(%
%TCIMACRO{\U{2115} }%
%BeginExpansion
\mathbb{N}
%EndExpansion
),K\in F(T\backslash R)\}$ is a base for the topology of $X_{R}.$ Also
$E_{(\mathbf{k},K)}=\emptyset$ unless $K\subset T\backslash R.$

Since each $E_{(\mathbf{k},K)}$ is clopen, it follows from Lemma 4.1, that
$E_{(\mathbf{k},K)}$ is the closure of $E_{(\mathbf{k},K)}\cap X_{R}.$

To slightly simplify our notation, we shall write $E(\mathbf{k},K)$ for
$E_{(\mathbf{k},K)}\cap X_{R}$ and $E_{n}$ for $E_{(\{n\}.\emptyset)}\cap
X_{R}.$ Also $E_{n}^{c}$ is the compliment of $E_{n}$ in $X_{R}$, which is,
$E_{(\{n\}.\emptyset)}^{c}\cap X_{R}.$ We shall see, below, that
$\{f_{\mathbf{h}}:n\notin\mathbf{h}\}=E_{n}^{c}$, equivalently, $E_{n}%
=\{f_{\mathbf{h}}:n\in\mathbf{h}\}.$

When $G$ is a subset of $X_{R}$ we denote its closure in $\beta X_{R}$ by
$clG.$ When $G$ is a clopen subset of $X_{R}$ then $clG$ is a clopen subset of
$\beta X_{R}.~$So the closure of $E_{n}$ in $\beta X_{R}$ is $clE_{n},$
whereas its closure in $K_{R}$ is $E_{(\{n\}.\emptyset)}.$

We need to show that each $\sigma_{n}$ is continuous on $X_{R}.$ Since
$\sigma_{n}$ is equal to its inverse, this implies that $\sigma_{n}$ is a
homeomorphism of $X_{R}$ onto itself.

Our first step to establish continuity of $\sigma_{n}$ is the following.

\begin{lemma}
We have $E_{n}=\{f_{\mathbf{k}}:n\in\mathbf{k\}}$ and $E_{n}^{c}%
=\{f_{\mathbf{m}}:n\notin\mathbf{m\}.}$ Also $\sigma_{n}$ interchanges $E_{n}$
and $E_{n}^{c}.$ Furthermore, for $m\neq n,$ $\sigma_{m}$ maps $E_{n}$ onto
$E_{n}$ and $E_{n}^{c}$ onto $E_{n}^{c}.$
\end{lemma}

\begin{proof}
\bigskip By definition $f_{\mathbf{k}}(\{n\},\emptyset)=1$ if, and only if
$\{n\}\subset\mathbf{k}.$ So $f_{\mathbf{m}}\in E_{n}^{c}$ precisely when
$n\notin\mathbf{m.}$

For $f_{\mathbf{k}}\in E_{n}$ we have $\sigma_{n}(f_{\mathbf{k}}%
)=f_{\mathbf{k\backslash\{n\}}}.$ So $\sigma_{n}$ maps $E_{n}$ onto $E_{n}%
^{c}.$ Similarly, it maps $E_{n}^{c}$ onto $E_{n}.$

When $m\neq n$, consider $f_{\mathbf{k}}\in E_{n}.$ Then $n\in\mathbf{k.}$ So
$n\in\mathbf{k}\cup\{m\}$ and $n\in\mathbf{k}\backslash\{m\}.~$Thus
$\sigma_{m}(f_{\mathbf{k}})$ is in $E_{n}.$i.e.$\sigma_{m}[E_{n}]\subset
E_{n}.$

Since $\sigma_{m}$ is idempotent, we get $\sigma_{m}[E_{n}]=E_{n}.$ Similarly
$\sigma_{m}[E_{n}^{c}]=E_{n}^{c}.$
\end{proof}

\begin{lemma}
The map $\sigma_{n}:X_{R}\rightarrow X_{R}$ is continuous.
\end{lemma}

\begin{proof}
It suffices to show that $\sigma_{n}^{-1}[E(\mathbf{l},L)]$ is open when
$L\subset T\backslash R.$

Let $f_{\mathbf{h}}$ be in $\sigma_{n}^{-1}[E(\mathbf{l},L)].$ We shall find,
$U,$ an open neighbourhood of $f_{\mathbf{h}}$ such that $\sigma_{n}%
[U]\subset$ $E(\mathbf{l},L).$

We need to consider three possibilities.

(1) First suppose that $n\in\mathbf{h,}$ that is $f_{\mathbf{h}}\in E_{n}$ $.$
Then $f_{\mathbf{h\backslash\{}n\mathbf{\}}}=\sigma_{n}(f_{\mathbf{h}}), $
which is in $E(\mathbf{l},L).$ So $\mathbf{l}\subset\mathbf{h\backslash
\{}n\mathbf{\}}$ which implies $n\notin\mathbf{l.}$ Also $N(t)\cap
((\mathbf{h}\backslash\{n\})\backslash\mathbf{l})=\emptyset$ for all $t\in L.$
It follows that $\mathbf{l}\cup\{n\}\subset\mathbf{h}$\ and, for all $t\in L,$
$N(t)\cap(\mathbf{h}\backslash(\mathbf{l}\cup\{n\}))=\emptyset.$ Hence
$f_{\mathbf{h}}\in E_{n}\cap E(\mathbf{l}\cup\{n\},L).$

Let $f_{\mathbf{k}}\in E(\mathbf{l}\cup\{n\},L).$ Then $\mathbf{l}\cup\{n\}
\subset\mathbf{k.}$ Also, for $t\in L,$ $N(t)\cap(\mathbf{k}\backslash
(\mathbf{l}\cup\{n\}))=\emptyset.$

Hence $\mathbf{l}\subset\mathbf{k\backslash}\{n\}$ and, for $t\in L,$
$N(t)\cap((\mathbf{k}\backslash\{n\})\backslash\mathbf{l})=\emptyset.$ This
implies $\sigma_{n}(f_{\mathbf{k}})=$ $f_{\mathbf{k}\backslash\{n\}}\in
E(\mathbf{l},L).$ Thus $E(\mathbf{l}\cup\{n\},L)$ is a clopen set, which is a
neighbourhood of $f_{\mathbf{h}}$ and a subset of $\sigma_{n}^{-1}%
[E(\mathbf{l},L)].$

(2) Now suppose $n\notin\mathbf{h}$. Then $f_{\mathbf{h}\cup\{n\}}=\sigma
_{n}(f_{\mathbf{h}})$ which is in $E(\mathbf{l},L).$ This gives

(a) \ $\mathbf{h}\cup\{n\}$ contains $\mathbf{l.}$ (b) For all $t\in L,$
$N(t)\cap((\mathbf{h}\cup\{n\})\backslash\mathbf{l})\mathbf{=}\emptyset.$\ (c)
$f_{\mathbf{h}}\in E_{n}^{c}.$

Suppose, additionally, that $n\in\mathbf{l}$. Then $(\mathbf{h}\cup
\{n\})\backslash\mathbf{l}=\mathbf{h}\backslash(\mathbf{l\backslash
}\{n\}\mathbf{).}$ So, for all $t\in L,$ $N(t)\cap(\mathbf{h}\backslash
(\mathbf{l\backslash}\{n\}))\mathbf{=}\emptyset.$ So $f_{\mathbf{h}}$ is in
$E(\mathbf{l}\backslash\{n\},L).$ Hence, by (c) $f_{\mathbf{h}}$ is in
$E_{n}^{c}\cap E(\mathbf{l}\backslash\{n\},L).$

Now let $f_{\mathbf{k}}\in$ $E_{n}^{c}\cap E(\mathbf{l}\backslash\{n\},L).$
Since $f_{\mathbf{k}}\in$ $E_{n}^{c},$ it follows that $n\notin\mathbf{k}$. So
$f_{\mathbf{k}\cup\{n\}}=\sigma_{n}(f_{\mathbf{k}}).$

Since $f_{\mathbf{k}}\in E(\mathbf{l}\backslash\{n\},L),$ we have
$\mathbf{l\backslash}\{n\}\subset\mathbf{k}.$ So $\mathbf{l}\subset
\mathbf{k}\cup\{n\}.$ Also $(\mathbf{k}\cup\{n\})\backslash\mathbf{l}%
=\mathbf{k}\backslash(\mathbf{l}\backslash\{n\}).$ So, for any $t\in
L,N(t)\cap((\mathbf{k}\cup\{n\})\backslash\mathbf{l})=\emptyset$. Thus
$f_{\mathbf{k}\cup\{n\}}\in E(\mathbf{l},L).$ That is, $\sigma_{n}%
(f_{\mathbf{k}})\in E(\mathbf{l},L).$

(3) We now suppose that $n\notin\mathbf{h}$ and $n\notin\mathbf{l}$. As in
(2), statements (a), (b) and (c) hold.

Note$\ \ \mathbf{h}\backslash\mathbf{l=}\ (\mathbf{h}\cup\{n\})\backslash
(\mathbf{l}\cup\{n\}).$ It follows from (b)\ that $N(t)\cap$ $\ (\mathbf{h}%
\backslash\mathbf{l)=\emptyset}$ for each $t\in L.$ Hence $f_{\mathbf{h}}\in
E(\mathbf{l},L)\cap E_{n}^{c}.$

We also observe that, because $n\notin\mathbf{l},$ (b) implies that (d) $\{n\}
\cap N(t)\mathbf{=\emptyset}$ for each $t\in L.$

Now let $f_{\mathbf{k}}\in E(\mathbf{l},L)\cap E_{n}^{c}.$ Then $n\notin%
\mathbf{k}.$So $\sigma_{n}(f_{\mathbf{k}})=f_{\mathbf{k}\cup\{n\}}.$

Also $\mathbf{l}\subset\mathbf{k}$ and $N(t)\cap(\mathbf{k}\backslash
\mathbf{l})=\emptyset$ for any $t\in L.$ It now follows from (d) that
$((\mathbf{k}\cup\{n\})\backslash\mathbf{l})\cap N(t)\mathbf{=\emptyset}$
whenever $t\in L.$ Hence $f_{\mathbf{k}\cup\{n\}}\in E(\mathbf{l},L).$ Thus
$E(\mathbf{l},L)\cap E_{n}^{c}$ is a clopen neighbourhood of $f_{\mathbf{h}}$
and it is a subset of $\sigma_{n}[E(\mathbf{l},L)].$

It follows from (1), (2) and (3) that every point of $\sigma_{n}%
[E(\mathbf{l},L)]$ has an open neighbourhood contained in $\sigma
_{n}[E(\mathbf{l},L)].$ In other words, the set is open.

More precisely: When $n\in\mathbf{l},$ we have $\sigma_{n}[E(\mathbf{l}%
,L)]=E_{n}^{c}\cap E(\mathbf{l}\backslash\{n\},L)$ and for $n\notin%
\mathbf{l},$ we find that

$\sigma_{n}[E(\mathbf{l},L)]=E(\mathbf{l}\cup\{n\},L)%
%TCIMACRO{\tbigcup }%
%BeginExpansion
{\textstyle\bigcup}
%EndExpansion
E(\mathbf{l},L)\cap E_{n}^{c}.$
\end{proof}

We recall that $X_{R\text{ }}$is completely regular because it is a subspace
of the compact Hausdorff space $K_{R}.$ Let $\beta X_{R}$ be its Stone-Czech
compactification. Then each continuous function $f:X_{R\text{ }}\rightarrow
X_{R}$ has a unique extension to a continuous function $F$\ from $\beta X_{R}
$ to $\beta X_{R}.$ When $f$ is a homeomorphism, then by considering the
extension of $f^{-1}$ it follows that $F$ is a homeomorphism of $\beta X_{R}.
$ In particular, each $\sigma_{n}$ has a unique extension to a homeomorphism
of $\beta X_{R}.$ We abuse our notation by also denoting this extension by
$\sigma_{n}.$

Let us recall from Section 6, that when $\theta$ is in $Homeo(\beta X_{R})$
then it induces an automorphism $h_{\theta}$ of $C(\beta X_{R})$ by
$h_{\theta}(f)=f\circ\theta.$ It also induces an automorphism of $B^{\infty
}(\beta X_{R})/M(\beta X_{R})$ by $H_{\theta}([F])=[F\circ\theta].$ Then
$H_{\theta}$ is the unique automorphism of $B^{\infty}(\beta X_{R})/M(\beta
X_{R})$ which extends $h_{\theta}.$ Let $S_{R}$ be the (extremally
disconnected) structure space of $B^{\infty}(\beta X_{R})/M(\beta X_{R})$;
this algebra can then be identified with $C(S_{R}).$ Then $H_{\theta}$
corresponds to $\widehat{\theta},$ an homeomorphism of $S_{R}.$ Then
$\theta\rightarrow h_{\theta}$ is a group anti-isomorphism of $Homeo(\beta
X_{R})$ onto $AutC(\beta X_{R});$ $h_{\theta}\rightarrow H_{\theta}$ is an
isomorphism of $AutC(\beta X_{R})$ into $AutC(S_{R}).$ Also $H_{\theta
}\rightarrow\widehat{\theta}$ is a group anti-isomorphism of $AutC(S_{R})$
into $Homeo(S_{R}).$ When $G$ \ is an Abelian subgroup of $Homeo(\beta X_{R})$
it follows that $\theta\rightarrow H_{\theta}$ and $\theta\rightarrow
\widehat{\theta}$ are group isomorphisms of $G$ into $AutC(S_{R})$ and
$Homeo(S_{R})$, respectively.

We recall that $g\rightarrow\varepsilon_{g}$ is an injective group
homomorphism of $%
%TCIMACRO{\tbigoplus }%
%BeginExpansion
{\textstyle\bigoplus}
%EndExpansion%
%TCIMACRO{\U{2124} }%
%BeginExpansion
\mathbb{Z}
%EndExpansion
_{2}$ into the group of bijections of $C$ onto $C.$ By taking the natural
bijection from $C$ onto $X_{R},$ and by applying Lemma 11.2 and Lemma 11.5, we
may regard $\varepsilon_{\ast}$ as an injective group homomorphism of $%
%TCIMACRO{\tbigoplus }%
%BeginExpansion
{\textstyle\bigoplus}
%EndExpansion%
%TCIMACRO{\U{2124} }%
%BeginExpansion
\mathbb{Z}
%EndExpansion
_{2}$ into $Homeo(X_{R}).$ Since each homeomorphism of $X_{R}$ onto itself has
a unique extension to a homeomorphism of $\beta X_{R}$ onto itself, we may
identify $\varepsilon_{\ast}$ with an injective group homomorphism of $%
%TCIMACRO{\tbigoplus }%
%BeginExpansion
{\textstyle\bigoplus}
%EndExpansion%
%TCIMACRO{\U{2124} }%
%BeginExpansion
\mathbb{Z}
%EndExpansion
_{2}$ into the group $Homeo(\beta X_{R}).$ This induces a group isomorphism,
$g\rightarrow\widehat{\varepsilon}^{g}$ from $%
%TCIMACRO{\tbigoplus }%
%BeginExpansion
{\textstyle\bigoplus}
%EndExpansion%
%TCIMACRO{\U{2124} }%
%BeginExpansion
\mathbb{Z}
%EndExpansion
_{2}$ into $AutoC(S_{R})$ by putting $\widehat{\varepsilon}^{g}=H_{\varepsilon
_{g}}.$ The corresponding isomorphism, $g\rightarrow\widehat{\varepsilon}%
_{g},$ from $%
%TCIMACRO{\tbigoplus }%
%BeginExpansion
{\textstyle\bigoplus}
%EndExpansion%
%TCIMACRO{\U{2124} }%
%BeginExpansion
\mathbb{Z}
%EndExpansion
_{2}$ into $Homeo(S_{R}),$ is defined by
\[
\widehat{\varepsilon}_{g}=\widehat{\varepsilon_{g}}\text{ for each }g\in%
%TCIMACRO{\tbigoplus }%
%BeginExpansion
{\textstyle\bigoplus}
%EndExpansion%
%TCIMACRO{\U{2124} }%
%BeginExpansion
\mathbb{Z}
%EndExpansion
_{2}.
\]

As in Section 6, $\rho$ is the continuous surjection from $S_{R}$ onto $\beta
X_{R}$ which is dual to the natural injection from $C(\beta X_{R})$ into
$\ B^{\infty}(\beta X_{R})/M(\beta X_{R})\simeq C(S_{R}).$ Let $s_{0}\in
S_{R}$ such that $\rho(s_{0})=f_{\emptyset}.$

\begin{theorem}
Let $g\rightarrow\widehat{\varepsilon}_{g}$ be the representation of $%
%TCIMACRO{\tbigoplus }%
%BeginExpansion
{\textstyle\bigoplus}
%EndExpansion%
%TCIMACRO{\U{2124} }%
%BeginExpansion
\mathbb{Z}
%EndExpansion
_{2}$, as homeomorphisms of $S_{R},$ as defined above. Then the orbit
$\{\widehat{\varepsilon_{g}}(s_{0}):g\in%
%TCIMACRO{\tbigoplus }%
%BeginExpansion
{\textstyle\bigoplus}
%EndExpansion%
%TCIMACRO{\U{2124} }%
%BeginExpansion
\mathbb{Z}
%EndExpansion
_{2}\}$ is a free, dense orbit in $S_{R}.$ There exists $Y,$ a dense
$G_{\delta}$ subset of $S_{R},$ with $s_{0}\in Y,$ such that $Y$ is invariant
under the action $\widehat{\varepsilon}$ and the action $\widehat{\varepsilon
}$ is free on $Y.$
\end{theorem}

\begin{proof}
By Lemma 11.2(i), $X_{R}=\{\varepsilon_{g}(f_{\emptyset}):g\in%
%TCIMACRO{\tbigoplus }%
%BeginExpansion
{\textstyle\bigoplus}
%EndExpansion%
%TCIMACRO{\U{2124} }%
%BeginExpansion
\mathbb{Z}
%EndExpansion
_{2}\}.$ By Proposition 6.4 this implies the orbit $\{\widehat{\varepsilon
_{g}}(s_{0}):g\in%
%TCIMACRO{\tbigoplus }%
%BeginExpansion
{\textstyle\bigoplus}
%EndExpansion%
%TCIMACRO{\U{2124} }%
%BeginExpansion
\mathbb{Z}
%EndExpansion
_{2}\}$ is dense in $S_{R}$.

By Lemma 11.2(ii), $\{ \varepsilon_{g}(f_{\emptyset}):g\in%
%TCIMACRO{\tbigoplus }%
%BeginExpansion
{\textstyle\bigoplus}
%EndExpansion%
%TCIMACRO{\U{2124} }%
%BeginExpansion
\mathbb{Z}
%EndExpansion
_{2}\}$ is a free orbit. This theorem now follows from Theorem 6.8.
\end{proof}

\begin{corollary}
The group isomorphism, $g\rightarrow\widehat{\varepsilon}^{g},$ from $%
%TCIMACRO{\tbigoplus }%
%BeginExpansion
{\textstyle\bigoplus}
%EndExpansion%
%TCIMACRO{\U{2124} }%
%BeginExpansion
\mathbb{Z}
%EndExpansion
_{2}$ into $AutC(S_{R})$ is free and ergodic.
\end{corollary}

We shall see below that we can now obtain some additional information about
this action of $%
%TCIMACRO{\tbigoplus }%
%BeginExpansion
{\textstyle\bigoplus}
%EndExpansion%
%TCIMACRO{\U{2124} }%
%BeginExpansion
\mathbb{Z}
%EndExpansion
_{2}$ as automorphisms of $C(S_{R})$. This will enable us to construct huge
numbers of hyperfinite, small wild factors.

We have seen that, for each natural number $n$, $\sigma_{n}$ is a
homeomorphism of $X_{R}$ onto itself with the following properties. First,
$\sigma_{n}=\sigma_{n}^{-1}.$ Secondly, $\sigma_{n}[E_{n}]=E_{n}^{c}$ and, for
$m\neq n,$we have $\sigma_{n}[E_{m}]=E_{m}.$ (This notation was introduced
just before Lemma 11.4, above.)

We have seen that $\sigma_{n}$ has a unique extension to a homeomorphism of
$\beta X_{R},$which we again denote by $\sigma_{n}.$ Then

$\sigma_{n}[clE_{n}]=clE_{n}^{c}$ and, for $m\neq n,$we have $\sigma
_{n}[clE_{m}]=clE_{m}.$

We define $e_{n}\in C(\beta X_{R})$ as the characteristic function of the
clopen set $clE_{n}.$

Using the above notation, $\widehat{\varepsilon}^{\sigma_{n}}$ is the $\ast
-$automorphism of $B^{\infty}(\beta X_{R})/M(\beta X_{R})\simeq C(S_{R})$
induced by $\sigma_{n}.$ We have%

\[
\widehat{\varepsilon}^{\sigma_{n}}(e_{n})=1-e_{n}\text{ and, for }m\neq
n,\text{ }\widehat{\varepsilon}^{\sigma_{n}}(e_{m})=e_{m}.\text{ }%
\]

We recall, see the final paragraph of Section 6, that $B^{\infty}%
(K_{R})/M(K_{R})$ can be identified with $B^{\infty}(\beta X_{R})/M(\beta
X_{R})$ and so with $C(S_{R}).$ By Proposition 13 \cite{zk} the smallest
monotone $\sigma-$complete $\ast-$subalgebra of $B^{\infty}(\beta
X_{R})/M(\beta X_{R})$ which contains $\{e_{n}:n=1,2...\}$ is $B^{\infty
}(\beta X_{R})/M(\beta X_{R})$ itself. We shall see that the (norm-closed)
$\ast-$algebra generated by $\{e_{n}:n=1,2...\}$ is naturally isomorphic to
$C(2^{%
%TCIMACRO{\U{2115} }%
%BeginExpansion
\mathbb{N}
%EndExpansion
}).$

When $S\subset%
%TCIMACRO{\U{2115} }%
%BeginExpansion
\mathbb{N}
%EndExpansion
$ we use $\eta_{S}$ to denote the element of $2^{%
%TCIMACRO{\U{2115} }%
%BeginExpansion
\mathbb{N}
%EndExpansion
}$ which takes the value $1$ when $n\in S$ and $0$ otherwise. Let $G_{n}$ be
the clopen set $\{\eta_{S}\in2^{%
%TCIMACRO{\U{2115} }%
%BeginExpansion
\mathbb{N}
%EndExpansion
}:n\in S\}.$ These clopen sets generate the (countable) Boolean algebra of
clopen subsets of $2^{%
%TCIMACRO{\U{2115} }%
%BeginExpansion
\mathbb{N}
%EndExpansion
}.$ An application of the Stone-Weierstrass Theorem shows that the $\ast
-$subalgebra of $C(2^{%
%TCIMACRO{\U{2115} }%
%BeginExpansion
\mathbb{N}
%EndExpansion
}),$ containing each $\chi_{G_{n}}$ , is dense in $C(2^{%
%TCIMACRO{\U{2115} }%
%BeginExpansion
\mathbb{N}
%EndExpansion
}).$

\begin{lemma}
There exists an isometric isomorphism, $\pi_{0},$ from $C(2^{%
%TCIMACRO{\U{2115} }%
%BeginExpansion
\mathbb{N}
%EndExpansion
})$ into $C(\beta X_{R})$ such that $\pi_{0}(\chi_{G_{n}})=e_{n}.$
\end{lemma}

\begin{proof}
As in Section 6 of \cite{zk} we define a map $\Gamma$ from the Big Cantor
space, $2^{F(%
%TCIMACRO{\U{2115} }%
%BeginExpansion
\mathbb{N}
%EndExpansion
)\times F(T)},$ onto the classical Cantor space,$2^{%
%TCIMACRO{\U{2115} }%
%BeginExpansion
\mathbb{N}
%EndExpansion
},$ by $\Gamma(\mathbf{x})(n)=\mathbf{x}((\{n\},\emptyset)).$ Put
$J=\{(\{n\},\emptyset):n=1,2...\}$. Then we may identify $2^{%
%TCIMACRO{\U{2115} }%
%BeginExpansion
\mathbb{N}
%EndExpansion
}$ with $2^{J}.$ So $\Gamma$ may be regarded as a restriction map and, by
definition of the topology for product spaces, it is continuous.

From the definition of $f_{\mathbf{k}},$ we see that $f_{\mathbf{k}%
}(\{n\},\emptyset)=1$ precisely when $n\in\mathbf{k}$. So $\Gamma
f_{\mathbf{k}}=\eta_{\mathbf{k}}.$ Hence
\[
\Gamma\lbrack E_{n}]=\{\eta_{\mathbf{k}}:n\in\mathbf{k}\text{ and }%
\mathbf{k}\in F(%
%TCIMACRO{\U{2115} }%
%BeginExpansion
\mathbb{N}
%EndExpansion
)\}.
\]

By the basic property of the Stone-Czech compactification, the natural
embedding of $X_{R}$ into $K_{R}$ factors through $\beta X_{R}$ . So there
exists a continuous surjection $\phi$ from $\beta X_{R}$ onto $K_{R}$ which
restricts to the identity map on $X_{R}.$ Then $\Gamma\phi$ maps $clE_{n}$
onto $G_{n}$ and $clE_{n}^{c}$ onto $G_{n}^{c}.$ For $f\in C(2^{%
%TCIMACRO{\U{2115} }%
%BeginExpansion
\mathbb{N}
%EndExpansion
})$ let $\pi_{0}(f)=f\circ\Gamma\phi.$ Then $\pi_{0}$ is the required
isometric isomorphism into $C(\beta X_{R})\subset B^{\infty}(\beta
X_{R})/M(\beta X_{R}).$

Let $\widehat{\varepsilon}$ be the action of the Dyadic Group on $C(S_{R})$
considered above. Let $M_{R}$ be the corresponding monotone cross-product
algebra. So there exists an isomorphism $\pi_{R}$ from $C(S_{R})$ onto the
diagonal subalgebra of $M_{R}$ and a group representation $g\rightarrow u_{g}
$ of the Dyadic Group in the unitary group of $M_{R}$ such that $u_{g}\pi
_{R}(a)u_{g}^{\ast}=\pi_{R}(\widehat{\varepsilon}^{g}(a)).$ Since each element
of the Dyadic Group is its own inverse, we see that each $u_{g}$ is
self-adjoint. Since the Dyadic Group is Abelian, $u_{g}u_{h}=u_{h}u_{g}$ for
each $g$ and $h.$

As before, let $g_{n}$ be the $n^{th}$ term in the standard sequence of
generators of $%
%TCIMACRO{\tbigoplus }%
%BeginExpansion
{\textstyle\bigoplus}
%EndExpansion%
%TCIMACRO{\U{2124} }%
%BeginExpansion
\mathbb{Z}
%EndExpansion
_{2}$ that is, $g_{n}$ takes the value $1$ in the $n^{th}$ coordinate and $0$
elsewhere. We abuse our notation by writing "$u_{n}"$ for the unitary
$u_{g_{n}}$ and "$e_{n}$" for the projection $\pi_{R}(e_{n})$ in the diagonal
subalgebra of $M_{R}.$ We then have:%
\[
u_{n}e_{n}u_{n}=1-e_{n}~\ \text{and, for }m\neq n,\text{ }u_{n}e_{m}%
u_{n}=e_{m}.
\]
Let $A_{R}=\pi_{R}[C(S_{R})]$ be the diagonal algebra of $M_{R}.$ We recall
that the Boolean $\sigma-$subalgebra of the projections of $A_{R},$ generated
by $\{e_{n}:n=1,2...\},$ contains all the projections of $A_{R}.$
\end{proof}

\begin{lemma}
Let $\mathcal{F}$ be the Fermion algebra. Then there exists an isomorphism
$\Pi$ from $\mathcal{F}$ onto the smallest norm \ closed $\ast-$subalgebra of
$M_{R}$ which contains $\{u_{n}:n=1,2...\}\ $and $\{e_{n}:n=1,2...\}.$ This
isomorphism takes the diagonal of $\mathcal{F}$ onto the smallest closed
abelian $\ast-$subalgebra containing $\{e_{n}:n=1,2...\}.$
\end{lemma}

\begin{proof}
For any projection $p$ we define $p^{(0)}=p$ and $p^{(1)}=1-p.$

For each choice of $n$ and for each choice of $\alpha_{1},...\alpha_{n}$ from
$%
%TCIMACRO{\U{2124} }%
%BeginExpansion
\mathbb{Z}
%EndExpansion
_{2},$ it follows from Lemma 11.8 that the product $e_{1}^{\alpha_{1}}%
e_{2}^{\alpha_{2}}...e_{n}^{\alpha_{n}}$ is neither $1$ nor $0.$ In the
notation of \cite{zzd}, $(e_{n})$ is a sequence of (mutually commutative)
independent projections. The Lemma now follows from (the easy part) of the
proof of Proposition 2.1 \cite{zzd}. In particular, for each $n, $
$\{u_{j}:j=1,2,...,n\}$ $\cup\{e_{j}:j=1,2,...,n\}$ generates a subalgebra
isomorphic to the algebra of all $2^{n}\times2^{n}$ complex matrices.
\end{proof}

\begin{definition}
\bigskip Let $B_{R}$ be the the smallest monotone $\sigma-$complete $\ast
-$subalgebra of $M_{R}$ which contains $\Pi\lbrack\mathcal{F}].$
\end{definition}

\begin{lemma}
$B_{R}$ is a monotone complete factor which contains $A_{R}$ as a maximal
abelian $\ast-$subalgebra. There exists a faithful normal conditional
expectation from $B_{R}$ onto $A_{R}.$ The state space of $B_{R}$ is
separable. The factor $B_{R}$ is wild and of Type III. It is also a small
$C^{\ast}$-algebra.
\end{lemma}

\begin{proof}
Let $D_{R}$ be the faithful normal conditional expectation from $M_{R}$ onto
$A_{R}.$ The maximal ideal space of $A_{R}$ can be identified with the
separable space $S_{R}.$ Then, arguing as in Corollary 3.2, there exists a
faithful state $\phi$ on $A_{R}$. Hence $\phi D_{R}$ is a faithful state on
$M_{R}$ and restricts to a faithful state on $B_{R}$. So, by Lemma 3.1,
$B_{R}$ is monotone complete. Let $D$ be the restriction of $D_{R}$ to $B_{R}$
then $D$ is a faithful and normal conditional expectation from $B_{R}$ onto
$A_{R}.$

Since each $e_{n}$ is in $B_{R}$ it follows that $A_{R}$ is a $\ast
-$subalgebra of $B_{R}.$ Since it is maximal abelian in $M_{R}$ it must be a
maximal abelian $\ast-$subalgebra of $B_{R}.$\ So the centre of $B_{R}$ is a
subalgebra of $A_{R}.$ Each $u_{n}$ is in $B_{R}$ and so each central
projection of $B_{R}$ commutes with each $u_{n}.$ Since the action
$\widehat{\varepsilon}$ of the Dyadic group is ergodic, it follows that the
only projections in $A_{R}$ which commute with every $u_{n}$ are $0$ and $1.$
So $B_{R}$ is a (monotone complete) factor.

The state space of every unital $C^{\ast}$-subalgebra of $M_{R}$ is a
surjective image of the state space of\ $M_{R}\,,$ which is separable. So the
state space of $B_{R}$ is separable. Equivalently, $B_{R}$ is almost separably
representable. A slightly more elaborate argument shows that \ this algebra is
small. See the remark preceding Theorem 6 \cite{zk}.

Since $B_{R}$ contains a maximal abelian $\ast-$subalgebra which is not a von
Neumann algebra it is a wild factor. Also $M_{R}$ is almost separably
representable, hence it possesses a strictly positive state and so is a Type
III factor \cite{zzg}, see also \cite{zf}.

It now follows immediately that the factor is a small $C^{\ast}$-algebra. For,
by work of K.Sait\^{o} \cite{zg}, a monotone complete factor is a small
$C^{\ast}$-algebra whenever it has a separable state space.
\end{proof}

\begin{proposition}
The homomorphism $\Pi$ extends to a $\sigma-$homomorphism $\Pi^{\infty}$ from
$\mathcal{F}^{\infty}$, the Pedersen-Borel envelope of the Fermion algebra,
onto $B_{R}.$ Let $J_{R}$ be the kernel of $\Pi^{\infty}$ then $\mathcal{F}%
^{\infty}/J_{R}$ is isomorphic to $B_{R}.$
\end{proposition}

\begin{proof}
This follows by Corollary 3.5.
\end{proof}

\section{Approximately finite dimensional algebras}

We could proceed in greater generality, but for ease and simplicity, we shall
only consider monotone complete $C^{\ast}$-algebra which possess a faithful
state. \ Every almost separably representable algebra has this property and
hence so does every small $C^{\ast}$-algebra.

\begin{definition}
Let $B$ be a monotone complete $C^{\ast}$-algebra with a faithful state. Then
$B$ is said to be \textbf{approximately finite dimensional}\ if there exists
an increasing sequence of finite dimensional $\ast-$subalgebras $(F_{n})$ such
that the smallest monotone closed subalgebra of $B$ which contains $\cup
_{n=1}^{\infty}F_{n}$ is $B$ itself.
\end{definition}

\begin{definition}
If (i) $B$ is a monotone complete $C^{\ast}$-algebra which satisfies the
conditions of Definition 12.1 and (ii) we can take each $F_{n}$ to be a full
matrix algebra, then $B$ is said to be\textit{\ }\textbf{strongly hyperfinite}.
\end{definition}

\begin{definition}
Let $M$ be a monotone complete $C^{\ast}$-algebra with a faithful state. We
call $M$ \textbf{nearly approximately}\textit{\ \textbf{finite dimensional}
(with respect to }$B$\textit{)} if it satisfies the following conditions:

(i) $M$ contains a monotone closed subalgebra $B$, where $B$ is
\textbf{approximately finite dimensional}.

(ii) There exists a linear map $D:M\rightarrow B$ which is positive, faithful
and normal.

(iii) For each $z\in M,$ there exists a sequence $(z_{n})(n=1,2...)$ in $B,$
such that

$D((z-z_{n})(z-z_{n})^{\ast})\geq D((z-z_{n+1})(z-z_{n+1})^{\ast})$ for each
$n,$ and%
\[%
%TCIMACRO{\tbigwedge \limits_{n=1}^{\infty}}%
%BeginExpansion
{\textstyle\bigwedge\limits_{n=1}^{\infty}}
%EndExpansion
D((z-z_{n})(z-z_{n})^{\ast})=0.
\]

\end{definition}

\begin{definition}
Let $M$ be a monotone complete $C^{\ast}$-algebra with a faithful state. We
call $M$ \textbf{hyperfinite} if it contains a monotone closed subalgebra $B$
such \ that (i) $M$ is nearly approximately finite dimensional with respect to
$B$ and (ii) $B$ is strongly hyperfinite.
\end{definition}

\bigskip

\begin{proposition}
`Let $S$ be a compact Hausdorff extremally disconnected space. Let $G$ be a
countably infinite group and $g\rightarrow\beta^{g}$ be a free action of $G$
as automorphisms of $C(S).$ Let $M(C(S),G)$ be the corresponding monotone
cross-product; let $\pi\lbrack C(S)]$ be the diagonal subalgebra; let
$D:M(C(S),G)\rightarrow\pi\lbrack C(S)]$ be the diagonal map. Let
$g\rightarrow u_{g}$ be a unitary representation of $G$ in $M(C(S),G)$ such
that $\beta^{g}(a)=u_{g}au_{g}^{\ast}$ for each $a\in\pi\lbrack C(S)].$ Let
$B$ be the monotone closure of the $\ast-$algebra generated by $\pi\lbrack
C(S)]$ $\cup\{u_{g}:g\in G\}.$ If $B$ is AFD then $M(C(S),G)$ is nearly AFD.
\end{proposition}

\begin{proof}
By Theorem 10.1, $B$ is the normalizer subalgebra of $M(C(S),G).$

By Lemma 9.4, $M(C(S),G)$ satisfies condition (iii) of Definition 12.3, with
respect to $B$ and the diagonal map $D.$

It follows immediately that $M(C(S),G)$ is nearly AFD whenever $B$ is
AFD.\bigskip
\end{proof}

We recall that in Section 3 \cite{zk} we constructed a weight semigroup,
$\mathcal{W},$ which classifies monotone complete $C^{\ast}$-algebras. In
particular, for algebras $B_{1}$ and $B_{2},$ they are equivalent (as defined
in \cite{zk}) precisely when their values in the weight semigroup, $wB_{1}$
and $wB_{2},$ are the same.

\ 

REMARK \ The theory of von Neumann algebras would lead us to expect
$M(C(S),G)=C(S)\times_{\beta}G.$ But this is an open problem. However it is
easy to show that $w(M(C(S),G))=w(C(S)\times_{\beta}G)$, that is, the two
algebras are equivalent.

Let $(T,\mathbf{O})$ be a feasible pair as in Section 11. Let $\mathcal{R}$ be
the collection of all admissible subsets of $T$. For each $R\in\mathcal{R}$
let $A_{R}=C(S_{R}).$ Then, by Corollary 20\textbf{\ } \cite{zk}, we can find
$\mathcal{R}_{0}\mathcal{\subset}$ $\mathcal{R}$ such that $\# \mathcal{R}%
_{0}=2^{c}$ where $c=2^{\aleph_{0}}$ with the following property. Whenever
$R_{1}$ and $R_{2}$ are distinct elements of $\mathcal{R}_{0}$ then $A_{R_{1}%
}$ is not equivalent to $A_{R_{2}}$that is, $wA_{R_{1}}\neq wA_{R_{2}}.$

\begin{theorem}
There exists a family of monotone complete $C^{\ast}$-algebras, $(B_{\lambda
},\lambda\in\Lambda)$ with the following properties: Each $B_{\lambda}$ is a
strongly hyperfinite Type III factor, each $B_{\lambda}$ is a small $C^{\ast}%
$-algebra, each $B_{\lambda}$ is a quotient of the Pedersen-Borel envelope of
the Fermion algebra. The cardinality of $\Lambda$ is $2^{c}$, where
$c=2^{\aleph_{0}}.$ When $\lambda\neq\mu$ then $B_{\lambda}$ and $B_{\mu}$
take different values in the classification semi-group $\mathcal{W}$; in
particular, they cannot be isomorphic.
\end{theorem}

\begin{proof}
First we put $\Lambda=\mathcal{R}_{0}.$ For each $R\in\mathcal{R}_{0}$ we have
a faithful normal conditional expectation from $B_{R}$ onto the maximal
abelian $\ast-$subalgebra $A_{R}.$ We use the partial ordering defined in
\cite{zk}. Since $A_{R}$ is a monotone closed subalgebra of $B_{R}$ and
$B_{R}$ is a monotone closed subalgebra of $M_{R}$ we obtain $A_{R}\precsim
B_{R}\precsim M_{R}.$ Since $D_{R}$ is a faithful normal map from $M_{R}$ onto
$A_{R},$ we have $M_{R}$ $\precsim A_{R}.$ Hence $A_{R}\symbol{126}%
B_{R}\symbol{126}M_{R}.$ By using the classification weight semi-group
$\mathcal{W},$ we get $wA_{R}=wB_{R}=wM_{R}.$ Since, for $R_{1}\neq R_{2},$ we
have $wA_{R_{1}}\neq wA_{R_{2}}$ $\ $this implies $wB_{R_{1}}\neq wB_{R_{2}}$.

The only item left to prove is that each factor $B_{R}$ is strongly
hyperfinite. But the Fermion algebra is isomorphic to $\Pi\lbrack
\mathcal{F]}.$ So $\Pi\lbrack\mathcal{F]}$ is the closure of an increasing
sequence of full matrix algebras. It now follows that $B_{R}$ is strongly hyperfinite.
\end{proof}

\begin{corollary}
For each orbit equivalence relation $E(R),$ corresponding to $R$, the orbit
equivalence factor $M_{E(R)}$ is hyperfinite.\textbf{\ }
\end{corollary}

\begin{theorem}
Let $S$ be a separable compact Hausdorff extremally disconnected \ space. Let
$C(S)$ be countably $\sigma-$generated. Let $G$ be a countably infinite group
of homeomorphisms of $S$ with a free, dense orbit. Let $E$ be the orbit
equivalence engendered by $G$ and $M_{E}$ the corresponding monotone complete
factor. Suppose that the Boolean algebra of projections of $C(S)$ is countably
generated. Then $M_{E}$ is nearly AFD. Let $B$ be the smallest monotone closed
$\ast-$subalgebra of $M_{E}$ containing the diagonal of $M_{E}$ and the
unitaries induced by the action of $G.\ $\ Then $B$ is AFD.
\end{theorem}

\begin{proof}
By Corollary 9.6,we may identify $M_{E}$ with $M(C(S),\bigoplus\mathbb{Z}%
_{2}).$ In other words we can assume that $G=\bigoplus\mathbb{Z}_{2}$. We can
further assume that $g\rightarrow\beta_{g}$, the action of $G$ as
homeomorphisms of $S$, is free and ergodic. Hence the corresponding action
$g\rightarrow\beta^{g},$ as automorphisms of $C(S),$\ is free and ergodic. Let
$\pi$ be the isomorphism from $C(S)$ onto the diagonal. Let $g\rightarrow
u_{g}$ be a unitary representation of $\bigoplus\mathbb{Z}_{2}$ such that
$u_{g}\pi(a)u_{g}^{\ast}=\pi(\beta^{g}(a))$ for each $a\in C(S).$

Let $(p_{n})$ be a sequence of projections in $C(S)$ which $\sigma-$generate
$C(S).$ By Proposition 9.8 the $C^{\ast}$-algebra,$B_{0}$, generated by
$\{p_{n}:n=1,2...\}\cup\{u_{n}:n=1,2...\}$ is the closure of the union of an
increasing sequence of finite dimensional subalgebras. Let $B$ be the smallest
monotone $\sigma-$closed subalgebra containing $B_{0}.$ ($B$ is the normalizer
subalgebra.) Then $B$ is monotone closed (because $M_{E}$ has a faithful
state) and AFD. Hence $M_{E}$ is nearly AFD.
\end{proof}

In Theorem 12.8 the hypotheses allow us to deduce that we can approximate
factors by increasing sequences of finite dimensional subalgebras (AFD) but in
the $2^{c}$ examples \ constructed in Section 11, we can do better. We can
approximate by sequences of full matrix algebras (strongly hyperfinite). Let
$M$ be a monotone factor which is AFD. Is $M$ strongly hyperfinite? Experience
with von Neumann algebras would suggest a positive answer but for wild factors
this is unknown.

Is $M_{E}$ equal to its normalizer subalgebra? Equivalently is the "small"
monotone cross-product equal to the "big" monotone cross-product? In general,
this is unknown. This is closely related to the following question, which has
been unanswered for over thirty years: Let $A$ be a closed $\ast-$subalgebra
of $L(H).$ Let $A^{\sigma}$ be the Pedersen-Borel envelope of $A.$ Let
$A^{\Sigma}$be the sequential closure of $A$ in the weak operator topology. By
a theorem of Davies \cite{h} $A^{\Sigma}$ is a $C^{\ast}$-algebra. Clearly
$A^{\sigma}\subset A^{\Sigma}.$ Are these algebras the same? For some special
cases, a positive answer is known, see the discussion in \cite{zd}.

Let $A$ be a monotone complete factor which is almost separably representable.
When $A$ is a von Neumann algebra being AFD is equivalent to being strongly
hyperfinite and is equivalent to being injective. But, as we pointed out in
Section 1, when $A$ is a wild factor the relationship between injectivity and
being AFD is a mystery.

\end{document}